\documentclass[12pt,reqno]{amsart}

\usepackage{color}
\usepackage{amsmath, amsthm, amssymb}

\usepackage{amsfonts,mathtools}

\usepackage[ansinew]{inputenc}
\usepackage[dvips]{epsfig}
\usepackage{graphicx}
\usepackage[english]{babel}
\usepackage{hyperref}

\usepackage[mathscr]{euscript}
\usepackage[margin=1.5cm,tmargin=2.5cm,bmargin=2cm]{geometry}
\usepackage{wrapfig}
\usepackage[usenames,dvipsnames,svgnames,table]{xcolor}

\theoremstyle{plain}
\newtheorem{thm}{Theorem}[section]
\newtheorem{cor}[thm]{Corollary}
\newtheorem{lem}[thm]{Lemma}
\newtheorem{prop}[thm]{Proposition}

\theoremstyle{definition}

\theoremstyle{remark}
\newtheorem{rem}[thm]{Remark}

\numberwithin{equation}{section}

\newcommand{\sign}{\,{\operatorname{sign}}\,}

\renewcommand{\epsilon}{\varepsilon}

\newcommand{\average}{{\mathchoice {\kern1ex\vcenter{\hrule height.4pt
width 6pt depth0pt} \kern-9.7pt} {\kern1ex\vcenter{\hrule
height.4pt width 4.3pt depth0pt} \kern-7pt} {} {} }}
\newcommand{\ave}{\average\int}
\def\R{\mathbb{R}}
\def\N{\mathbb{N}}

\allowdisplaybreaks

\begin{document}

\title[Non-symmetric stable operators]{Non-symmetric stable operators: \\ regularity theory and integration by parts}

\author{Serena Dipierro}
\address{University of Western Australia,
Department of Mathematics and Statistics, 35 Stirling Highway,
WA6009 Crawley, Australia}
\email{serena.dipierro@uwa.edu.au}

\author{Xavier Ros-Oton}
\address{ICREA, Pg.\ Llu\'is Companys 23, 08010 Barcelona, Spain \&  \newline
\indent Universitat de Barcelona, Departament de Matem\`atiques i Inform\`atica, Gran Via de les Corts Catalanes 585, 08007 Barcelona, Spain.}
\email{xros@ub.edu}

\author{Joaquim Serra}
\address{ETH Z\"urich, Department of Mathematics, Raemistrasse 101, 8092 Z\"urich, Switzerland}
\email{joaquim.serra@math.ethz.ch}

\author{Enrico Valdinoci}
\address{University of Western Australia,
Department of Mathematics and Statistics, 35 Stirling Highway,
WA6009 Crawley, Australia}
\email{enrico.valdinoci@uwa.edu.au}




\maketitle

\begin{abstract}
We study solutions to $Lu=f$ in $\Omega\subset\mathbb R^n$, being $L$ the generator of any,
possibly \emph{non-symmetric}, stable L\'evy process.

On the one hand, we study the regularity of solutions to $Lu=f$ in $\Omega$, $u=0$ in $\Omega^c$, in $C^{1,\alpha}$ domains~$\Omega$.
We show that solutions $u$ satisfy $u/d^\gamma\in C^{\varepsilon_\circ}\big(\overline\Omega\big)$, where $d$ is the distance to $\partial\Omega$, and $\gamma=\gamma(L,\nu)$
is an explicit exponent that depends on the Fourier symbol of operator $L$ and on the unit normal $\nu$ to the boundary $\partial\Omega$.

On the other hand, we establish  new integration by parts identities in half spaces for such operators.
These new identities extend previous ones for the fractional Laplacian, but the non-symmetric setting presents some new interesting features.

Finally, we generalize the integration by parts identities  in half spaces to  the case of bounded $C^{1,\alpha}$ domains. We do it via a new efficient approximation argument, which exploits the  H\"older regularity of $u/d^\gamma$. This new approximation argument is interesting, we believe, even in the case of the fractional Laplacian.
\end{abstract}


\section{Introduction and results}

The regularity of solutions to integro-differential equations has attracted a lot of attention in the last decade; see \cite{BCI2,BL02,CCV,CS,CS2,CS3,CD,CKS,Fal,GS,IS,K,KM,KRS,Kriv,PK,Ser2,SS,Sil06}
as well as  \cite{AR,Bog97,BKK08,Grubb,Grubb2,RS-Duke,SW99}.
This type of equations arises naturally in the study of L\'evy processes, and in particular appear in a variety of models, from Physics to Biology or Finance;
see e.g.~\cite{Bucur}.

An important class of L\'evy processes are the so-called $\alpha$-stable processes; see \cite{Bertoin,ST}.
These are somehow the equivalent of Gaussian processes when dealing with infinite variance random variables.

\vspace{3mm}
 
In this paper we study in detail  \emph{non-symmetric} stable operators. 
More precisely, we will:
\begin{enumerate}
\item[(a)] Develop an interior regularity theory for all non-symmetric stable operators, including those with kernels that might vanish or be singular.
\item[(b)] Establish fine estimates for the boundary behavior of solutions.
\item[(c)] Find new integration by parts identities in bounded domains.
\end{enumerate}

The interior regularity results are more standard, in the sense that they are quite similar to the symmetric case.
Still, we improve and simplify the proofs from \cite{RS-elliptic}, including a very general proof of a Liouville theorem and a simple proof of the Harnack inequality in this setting.

The boundary regularity, instead, turns out to be much more delicate in the non-symmetric setting, and it is not a straightforward extension of previous known results for symmetric operators.
A key reason for this is that, while in the symmetric setting all solutions behave like a fixed power of the distance function \cite{AR,Grubb,Grubb2,RS-Duke,RS-elliptic}, in the present context we will see that solutions have a more complicated behaviour, and the analysis turns out to be quite delicate.

Finally, based on the boundary regularity results that we develop here, we are able to find completely new integration by parts identities which extend in a quite surprising way the previous known identities for symmetric operators \cite{RS-Poh,Grubb3,RSV}.
Moreover, the proof that
we present here is new, and simplifies those in \cite{RS-Poh,RSV}.

\addtocontents{toc}{\protect\setcounter{tocdepth}{1}}
\subsection{The setting}

The operators that we consider are of the form
\begin{equation}\label{operator-L}
Lu(x):=\begin{dcases}
\displaystyle\int_{\R^n}\big(u(x)-u(x+y)\big)\,K(y)\,dy, \qquad& {\mbox{ if }}s\in\left(0,{\textstyle\frac12}\right),\\
{\rm P.V.}
\displaystyle\int_{\R^n}\big(u(x)-u(x+y)\big)\,K(y)\,dy
+b\cdot\nabla u(x), \qquad& {\mbox{ if }}s={\textstyle\frac12},\\
\displaystyle\int_{\R^n}\big(u(x)-u(x+y)+\nabla u(x)\cdot y\big)\,K(y)\,dy, \qquad& {\mbox{ if }}s\in\displaystyle\left({\textstyle\frac12},1\right),
\end{dcases}
\end{equation}
where $b\in\R^n$, $s\in(0,1)$, and the kernel 
\begin{equation}\label{homog}
K\in L^1_{\rm loc}(\R^n\setminus\{0\})\quad \textrm{is nonnegative and
positively homogeneous of degree~$-n-2s$.}
\end{equation}
When~$s=\frac12$, the kernel must satisfy the additional cancellation property
\begin{equation}\label{98988w019375=A}
\int_{\partial B_1} y\,K(y)\,d\mathcal{H}^{n-1}(y)=0.
\end{equation}

More generally,
the kernel $K$ could be taken to be a measure, with the obvious modifications in the statements.
However, for simplicity of notation, we prefer to assume $K\in L^1_{\rm loc}(\R^n\setminus\{0\})$.
\medskip

The only ellipticity assumption in all our results will be
\begin{equation}\label{ellipt-const}
0<\lambda\leq \inf_{\nu\in S^{n-1}}
\int_{S^{n-1}}|\nu\cdot\theta|^{2s} K(\theta)d\theta
\qquad{\mbox{ and }}\qquad  \int_{S^{n-1}}K(\theta)\,d\theta\leq \Lambda,
\end{equation}
for some~$\Lambda\ge\lambda>0$, and, without loss of generality,
we suppose that~$|b|\leq \Lambda$ when~$s=\frac12$.
Notice that without such assumption no regularity result holds.

\addtocontents{toc}{\protect\setcounter{tocdepth}{1}}
\subsection{Interior regularity}

Our first result reads as follows.

\begin{thm}\label{thm-interior}
Let $s\in(0,1)$, and let $L$ be any operator of the form \eqref{operator-L}-\eqref{ellipt-const}.
Let $u$ be any bounded weak solution to
\begin{equation}\label{eq-ball}
L u=f\quad {in}\ B_1.
\end{equation}
Then,
\begin{itemize}
\item[(a)] If $f\in L^\infty(B_1)$, 
\[\|u\|_{C^{2s}(B_{1/2})}\leq C\left(\|u\|_{L^\infty(\R^n)}+\|f\|_{L^\infty(B_1)}\right)\quad \textrm{if}\ s\neq{\textstyle\frac12}\]
and
\[\|u\|_{\Lambda^1(B_{1/2})}\leq C\left(\|u\|_{L^\infty(\R^n)}+\|f\|_{L^\infty(B_1)}\right)\quad \textrm{if}\ s={\textstyle\frac12},\]
where $\Lambda^1$ is the Zygmund space; see Remark \ref{Zygmund-space}.
In particular, $u\in C^{2s-\varepsilon}(B_{1/2})$ for all $\varepsilon>0$.

\vspace{2mm}

\item[(b)] If $f\in C^\alpha(B_1)$ for some $\alpha>0$, 
\begin{equation}\label{estimate-thm-1-b}
\|u\|_{C^{\alpha+2s}(B_{1/2})}\leq C\left(\|u\|_{C^\alpha(\R^n)}+\|f\|_{C^\alpha(B_1)}\right)
\end{equation}
whenever $\alpha+2s$ is not an integer.

\vspace{2mm}

\item[(c)] If $f\in C^\alpha(B_1)$ for some $\alpha>0$ and in addition $K\in C^\alpha(S^{n-1})$, 
\begin{equation}\label{estimate-thm-1-c}
\|u\|_{C^{\alpha+2s}(B_{1/2})}\leq C\left(\|u\|_{L^\infty(\R^n)}+\|f\|_{C^\alpha(B_1)}\right)
\end{equation}
whenever $\alpha+2s$ is not an integer.
\end{itemize}

The constants $C$ in (a) and (b) depend only on $n$, $s$, $\alpha$, and the ellipticity constants in~\eqref{ellipt-const}.
The constant $C$ in (c) depends in addition on $\|K\|_{C^\alpha(S^{n-1})}$.
\end{thm}

The notion of weak solution will be explicitly recalled in~\eqref{WEAKDE}.
We also prove a Harnack inequality under the additional assumption that the kernel satisfies $K(y)\asymp |y|^{-n-2s}$; see Theorems \ref{half-Harnack-sub} and \ref{half-Harnack-sup}.

\subsection{Boundary regularity}

Concerning the boundary regularity of solutions, our main result reads as follows.

\begin{thm}\label{thm-bdry}
Let $s\in(0,1)$, $L$ be any operator of the form \eqref{operator-L}-\eqref{ellipt-const}, and $\Omega$ be any bounded $C^{1,\alpha}$ domain.
Let $f\in L^\infty(\Omega)$, and $u$ be any weak solution of
\begin{equation}\label{eq-bdry}
\left\{ \begin{array}{rcll}
L u &=&f&\textrm{in }\Omega \cap B_1 \\
u&=&0&\textrm{in }B_1\setminus\Omega,
\end{array}\right.
\end{equation}
and let $C_0:=\|f\|_{L^\infty(\Omega)} + \|u\|_{L^\infty(\R^n)}$.
Let $d$ be the distance to $\partial\Omega$, and for every $z\in \partial\Omega\cap B_{1/2}$ we denote by~$\nu\in S^{n-1}$ the inward unit normal to $\partial\Omega$ at $z$.

Let~$\mathcal{A}(\xi)+i\mathcal{B}(\xi)$ be the Fourier symbol of~$L$, given by \eqref{Fourier-A}-\eqref{Fourier-B}, and let $\gamma(z)$ be given by
\begin{equation}\label{gamma-intro}
\qquad \qquad \qquad \qquad \gamma(z) = s+\frac1\pi\,\arctan\left(\frac{\mathcal B(\nu)}{\mathcal A(\nu)}\right),\qquad z\in \partial\Omega\cap B_{1/2}.
\end{equation}
Then, $$\gamma(z)\in (0,2s)\cap (2s-1,1),$$
and, for every $z\in \partial\Omega\cap B_{1/2}$ there exists $c_z\in \R$
such that\footnote{As customary, here and in the rest of the paper
we use the notation~$r_+:=\max\{r,0\}$ and~$r_-:=\max\{-r,0\}$.} 
\begin{equation}\label{bdry-expansion}
\qquad \qquad \qquad \qquad \left|u(x)-c_z\big((x-z)\cdot \nu\big)_+^{\gamma(z)}\right| \leq CC_0|x-z|^{\gamma(z)+\varepsilon_\circ}\qquad \qquad \textrm{for all}\quad x\in B_{1/4}(z),
\end{equation}
where $C$ and $\varepsilon_\circ>0$ depend only on $n$, $s$, $\Omega$, and the ellipticity constants in~\eqref{ellipt-const}.

Furthermore, if in addition we assume that the kernel of the operator $L$ satisfies $K(y)\leq \Lambda|y|^{-n-2s}$, and that $\Omega$ is $C^{1,1}$, then the improved estimate
\[\left|u(x)-c_z d^{\gamma(z)}(x)\right|\leq C_\delta C_0 |x-z|^{2s-\delta}\qquad \textrm{for all}\ x\in B_{1/4}(z)\]
holds, for any $\delta>0$.
\end{thm}

Notice that the boundary behaviour of solutions is quite different from the symmetric case, in which all solutions behave like $d^s$ near the boundary.
Here, we need a quite delicate analysis to find \eqref{bdry-expansion}.

When~$\Omega:=\R^n_+:=\{ x=(x',x_n)\in\R^{n-1}\times\R {\mbox{ s.t. }}
x_n>0\}$, a similar result has been proved in~\cite[Theorem~4.4]{Grubb}
when the symbol of the operator is smooth outside the origin.
Furthermore, it is interesting to notice that a similar result was recently obtained independently in \cite{Jus} using a probabilistic approach. 
However, our result above gives stronger consequences ---which we need in our application--- under less restrictive hypothesis.

On the other hand, when $\Omega$ is $C^\infty$, the study of Dirichlet problems for general pseudodifferential operators --- which includes our operators when the kernel $K$ in \eqref{operator-L} is $C^\infty$ outside the origin --- has been carried out in Eskin's book \cite{Esk}.
In particular, at each boundary point $x'\in \R^{n-1}\times\{0\}$ a factorization index $\kappa(x')$ is defined, and a boundary expansion of the form $u(x)\approx c_0(x')(x_n)_+^{\kappa(x')}$ is established in (26.3) therein.
We refer to \cite{ChD,VE,Sha} for further boundary regularity results for general pseudodifferential operators in Sobolev-type spaces, under appropriate smoothness assumptions on $\Omega$ and on the Fourier symbol of $L$.

As a consequence of Theorem \ref{thm-bdry}, we have the following
boundary regularity result:

\begin{cor}\label{cor-bdry}
Let $L$, $\Omega$, $\alpha$, $u$, $d$, $\gamma$, and $\varepsilon_\circ$ be as in Theorem \ref{thm-bdry}.

Then, there exists a function $\bar\gamma\in C^{\alpha}(\overline\Omega\cap B_{1/2})$, such that $\bar \gamma\equiv\gamma$ on $\partial\Omega$, and such that 
\[\|u/d^{\bar\gamma}\|_{C^{\varepsilon_\circ}(\overline\Omega\cap B_{1/2})}\leq C\big(\|f\|_{L^\infty(\Omega\cap B_1)}+\|u\|_{L^\infty(\R^n)}\big).\]
In particular, we may define the function 
\[u/d^\gamma \in C^{\varepsilon_\circ}(\partial\Omega\cap B_{1/2})\]
as  $u/d^\gamma:= u/d^{\bar\gamma}|_{\partial\Omega}$.
Moreover, the value of such function at $z\in \partial\Omega\cap B_{1/2}$ coincides with $c_z$ in \eqref{bdry-expansion}.
\end{cor}

This function $u/d^\gamma$ will be important in our next result.

\subsection{New integration by parts identities}

As said above, an important result that we establish here is a family of completely new integration by parts identities which extend the previous known identities for symmetric operators.
Before stating the result, we need the following notation.

We denote by~$K_e$ and~$K_o$ the even and odd parts of~$K$, respectively,
that is
\begin{equation}\label{SPLIT}
K_e(x):=\frac{K(x)+K(-x)}{2}\qquad{\mbox{and}}\qquad
K_o(x):=\frac{K(x)-K(-x)}{2}.
\end{equation}
By construction~$K_e$ and~$K_o$ are positively homogeneous of degree~$-n-2s$,
thanks to~\eqref{homog},
and $K=K_e+K_o.$

We also define
\begin{equation}\label{OPERST}
K^*:=K_e-K_o
\end{equation}
and
\begin{equation}\label{OPER STAR}
L^*u(x):=\begin{dcases}
\displaystyle\int_{\R^n}\big(u(x)-u(x+y)\big)\,K^*(y)\,dy, \qquad& {\mbox{ if }}s\in\displaystyle\left(0,{\textstyle\frac12}\right),\\
{\rm P.V.}\displaystyle\int_{\R^n}\big(u(x)-u(x+y)\big)\,K^*(y)\,dy
-b\cdot\nabla u(x)
, \qquad& {\mbox{ if }}s={\textstyle\frac12},\\
\displaystyle\int_{\R^n}\big(u(x)-u(x+y)+\nabla u(x)\cdot y\big)\,K^*(y)\,dy, \qquad& {\mbox{ if }}s\in\displaystyle\left({\textstyle\frac12},1\right).
\end{dcases}
\end{equation}
We point out that~$L^*$ is the adjoint operator of~$L$
(see e.g. Lemma~\ref{C:ADJ} for a simple proof of this fact).
If~$K$ is symmetric, then~$K^*=K$ and~$L^*=L$.
\medskip

We state next a first new integration by parts identity in half spaces. 

\begin{prop}\label{flat-case}
Let $s\in(0,1)$, let $L$ be any operator of the form \eqref{operator-L}-\eqref{ellipt-const},
let~${\mathcal{A}}(\xi)+i{\mathcal{B}}(\xi)$ be the Fourier symbol of~$L$,
given by \eqref{Fourier-A}-\eqref{Fourier-B},
and let $\gamma=\gamma(L,e_n)$ be given by \begin{equation}\label{GAMMAGIUST}
\gamma:=s+\frac1\pi\,\arctan\left(\frac{\mathcal B(e_n)}{\mathcal A(e_n)}\right)
.\end{equation}
Let~$\eta$, $\tau\in C^\infty_c(\R^n)$.
Let~$\gamma^*\in(0,2s)$ be such that~$\gamma+\gamma^*=2s$,
and
\begin{equation}\label{uev-9731} 
u(x):=(x_n)_+^\gamma\,\eta(x) \qquad{\mbox{and}}\qquad
v(x):=(x_n)_+^{\gamma^*}\,\tau(x).
\end{equation}
Then,
\begin{equation}\label{M:X}
\int_{\R^n_+} \big(
\partial_n u(x)\,L^* v(x)+Lu(x)\,\partial_n v(x)\big)\,dx=
c(L,e_n)\,\int_{\R^{n-1}\times\{0\}}\eta(x',0)\,\tau(x',0)\,dx',
\end{equation}
and the explicit value of  $c(L,e_n)$ is given by the formula
\begin{equation}\label{VAL:c}
c(L,e_n)=\Gamma(\gamma + 1)\, \Gamma(\gamma^* + 1)\,\sqrt{\mathcal A^2(e_n)+\mathcal B^2(e_n)}.
\end{equation}
\end{prop}

Combining Proposition \ref{flat-case} with the interior and boundary regularity results in Theorems \ref{thm-interior} and \ref{thm-bdry}, we are able to establish, via a new efficient approximation argument, the following integration by parts formula in general domains $\Omega$.

\begin{thm}\label{thm-Poh}
Let $s\in(0,1)$, let $L$ be any operator of the form \eqref{operator-L}-\eqref{ellipt-const}, and let $L^*$ be its adjoint operator.
Let $\Omega$ be any bounded $C^{1,\alpha}$ domain, and let $u$ and $v$ be solutions of 
\begin{equation}\label{eq-Poh}
\left\{ \begin{array}{rcll}
L u &=&f&\textrm{in }\Omega \\
u&=&0&\textrm{in }\Omega^c,
\end{array}\right. \qquad\qquad\qquad
\left\{ \begin{array}{rcll}
L^* v &=&g&\textrm{in }\Omega \\
v&=&0&\textrm{in }\Omega^c,
\end{array}\right.
\end{equation}
with $f,g\in L^\infty(\Omega)$.
When $s\leq \frac12$, assume in addition\footnote{See Remark \ref{rem111}.} that $f,g\in C^{1-2s+\varepsilon}(\overline\Omega)$ and $K\in C^{1-2s}(S^{n-1})$
for some small $\varepsilon>0$.

For every $z\in \partial\Omega\cap B_{1/2}$ we denote by~$\nu\in S^{n-1}$ the inward unit normal to $\partial\Omega$ at $z$.
Let~$\mathcal{A}(\xi)+i\mathcal{B}(\xi)$ be the Fourier symbol of~$L$, given by \eqref{Fourier-A}-\eqref{Fourier-B}, and let
\begin{equation}\label{gammas}
\gamma(z) = s+\frac1\pi\,\arctan\left(\frac{\mathcal B(\nu)}{\mathcal A(\nu)}\right) \qquad{\mbox{ and }} \qquad \gamma^*(z)=s-\frac1\pi\,\arctan\left(\frac{\mathcal B(\nu)}{\mathcal A(\nu)}\right), 
\end{equation}
for $z\in \partial\Omega$. Let~$d$ be the distance to~$\partial\Omega$ and
define the functions $u/d^\gamma$ and $v/d^{\gamma^*}$ on $\partial\Omega$ as in Corollary~\ref{cor-bdry}.

Then, 
\begin{equation}\label{Poh}
\int_{\Omega} \big(
\partial_e u\,L^* v+Lu\,\partial_e v\big)\,dx=
\int_{\partial\Omega}\,\frac{u}{d^{\gamma}}\,\frac{v}{d^{\gamma^*}}\, \Gamma(\gamma + 1)\, \Gamma(\gamma^* + 1)\,\sqrt{\mathcal A^2(\nu)+\mathcal B^2(\nu)}\,(\nu\cdot e)\,dz
\end{equation}
for any $e\in \mathbb S^{n-1}$.
\end{thm}

\begin{rem}\label{rem111}
The extra assumptions needed in the case $s\leq \frac12$ of  Theorem \ref{thm-Poh} guarantee that, for every $s\in (0,1)$, we have
\begin{equation}\label{hypotesis-Poh}
|\nabla u|+|\nabla v|\leq Cd^{\varepsilon-1}\quad\textrm{in}\quad \Omega.
\end{equation}
As we will see,  for $s>\frac 1 2$ these gradient bounds follow from \eqref{eq-Poh} since  $f$ and $g$ are bounded (this will be shown combining the  interior and boundary regularity  results from Theorems \ref{thm-interior} and \ref{thm-bdry}). However in the case $s\le \frac 1 2$ the order of the elliptic operators $L$, $L^*$ is one or less, and that is why extra assumptions are required in order to control the gradient.
\end{rem}

We stress that $\gamma$
and~$\gamma^*$ in~\eqref{gammas} and~\eqref{Poh}
are functions defined on $\partial\Omega$, not constants.
Notice also that, since $\Omega$ is $C^{1,\alpha}$ and the Fourier symbol of $L$ is H\"older continuous, then $\gamma(z)$ and $\gamma^*(z)$ are H\"older continuous.

When $L$ is symmetric ---i.e., when $K$ is even---, then $\mathcal B\equiv0$ and $\gamma\equiv\gamma^*\equiv s$, so that we recover the identity established in \cite{RSV,RS-Poh}; see also \cite{Grubb3} for an extension of the results in \cite{RSV,RS-Poh} to the case of $x$-dependent operators.

 Let us emphasize that the  approach towards fractional integration by parts identities in this paper (the approximation argument given in Section \ref{sec6})  is quite different from ---and, we believe, shorter and cleaner than--- that in \cite{RS-Poh,RSV}.
In this respect,  Section \ref{sec6} in this paper might  also be interesting to readers who are concerned only with the fractional Laplacian $(-\Delta)^s$ and not with more general non-symmetric and anisotropic operators.

We also mention that we expect our results to hold for more general
pseudodifferential operators, in particular for kernels which are not
necessarily nonnegative, but we do not pursue this direction in this paper. 


\subsection{Organization of the paper}

In Section \ref{sec2} we provide some basic properties of the operators $L$ under consideration, including the explicit expression of their Fourier symbol.
In Section \ref{sec3} we establish our interior regularity results,
proving Theorem~\ref{thm-interior}.
Then, in Section \ref{sec4} we prove the fine boundary regularity estimates in Theorem~\ref{thm-bdry}.
In Section \ref{sec5} we prove the integration by parts identities in half-spaces, and then in Section \ref{sec6} we give the approximation argument to establish Theorem~\ref{thm-Poh}.
Finally, in Appendices \ref{secA} and \ref{secB} we prove the auxiliary
results in Propositions \ref{1D-solution} and \ref{stable-operators}, respectively.

\subsection*{Acknowledgements}

We thank G. Grubb and M. Kwa\'snicki for several comments and remarks on a previous version of this paper.

SD was supported by the
Australian Research Council (ARC) under the DECRA Grant Agreement No DE180100957.
EV was supported by the Australian Research Council (ARC) under the
Australian Laureate Fellowship Agreement FL190100081.
SD and EV were supported by 
the Australian Research Council (ARC) under the
Discovery Project Grant Agreement No DP170104880,
and they are members of AustMS and INdAM.
XR was supported by the European Research Council (ERC) under the Grant Agreement No 801867.
JS has received funding from the European Research Council (ERC) under the Grant Agreement No 721675.
JS was supported by Swiss NSF Ambizione Grant PZ00P2 180042.
XR and JS were supported by MINECO grant MTM2017-84214-C2-1-P (Spain).

\section{Basic properties}
\label{sec2}

In this section, we gather some pivotal results on the integro-differential
operators of the form given in~\eqref{operator-L}.

Throughout this paper, given~$u\in L^\infty(\R^n)$ and $f\in L^\infty(\Omega)$,
we say that~$Lu=f$ in~$\Omega$ in the weak sense if
\begin{equation}\label{WEAKDE}
\int_{\R^n} u(x)\,L^* \eta(x)\,dx=\int_{\Omega} f(x)\,\eta(x)\,dx
\end{equation}
for every~$\eta\in C^\infty_c(\Omega)$, being~$L^*$ the adjoint operator
introduced in~\eqref{OPER STAR}.

In the next result, we recall that all $\alpha$-stable
processes have the operator in~\eqref{operator-L} as their infinitesimal generator:

\begin{prop}\label{stable-operators}
Let $\alpha\in(0,2)$ and $X_t$ be an $\alpha$-stable, $n$-dimensional, L\'evy process.
Then, the infinitesimal generator of $X_t$ is an operator of the form
\[
Lu(x):=\begin{dcases}
\displaystyle\int_{\R^n}\big(u(x)-u(x+y)\big)\,d\mu(y), \qquad& {\mbox{ if }}\alpha\in(0,1),\\
{\rm P.V.}
\displaystyle\int_{\R^n}\big(u(x)-u(x+y)\big)\,d\mu(y)
+b\cdot\nabla u(x), \qquad& {\mbox{ if }}\alpha=1,\\
\displaystyle\int_{\R^n}\big(u(x)-u(x+y)+\nabla u(x)\cdot y\big)\,d\mu(y), \qquad& {\mbox{ if }}\alpha\in(1,2),
\end{dcases}\]
for some $b\in\R^n$, and a measure $\mu$ satisfying
\[\int_{\R^n} \min\{1,|y|^2\}\,d\mu(y)<+\infty\]
and which is positively homogeneous of degree $-n-\alpha$.
Furthermore, when~$\alpha=1$,
\begin{equation}\label{wnhiowhgowgerg9ewg}
\int_{B_{R}\setminus B_r}y\,d\mu(y)=0,
\end{equation}
for all~$R>r>0$.
\end{prop}

We defer the proof of this result to Appendix~\ref{secB}.
In the following lemma, 
we compute the Fourier symbol of the operator in~\eqref{operator-L}:

\begin{lem} \label{AeBFOU}
Let $L$ be any operator of the form \eqref{operator-L}-\eqref{ellipt-const}.
Then, we have that
\begin{equation}\label{FS}
\begin{split}
&{\mbox{the Fourier symbol of~$L$ can be written as~${\mathcal{A}}(\xi)+i{\mathcal{B}}(\xi)$,}}\\
&{\mbox{and the Fourier symbol of~$L^*$ can be written as~${\mathcal{A}}(\xi)-i{\mathcal{B}}(\xi)$,}}\end{split}\end{equation}
with~${\mathcal{A}}:\R^n\to[0,+\infty)$ and~${\mathcal{B}}:\R^n\to\R$.
Also,
\begin{equation}\label{EOO}
{\mbox{${\mathcal{A}}$ and~${\mathcal{B}}$ are positively homogeneous of degree~$2s$,}}\end{equation}
\begin{equation}\label{EO}
{\mbox{${\mathcal{A}}$ is even and~${\mathcal{B}}$ is odd,}}
\end{equation}
and
\begin{equation}\label{Fourier-A}
{\mathcal{A}}(\xi)=|\Gamma(-2s)|\cos(\pi s)\,\int_{S^{n-1}} |\theta\cdot \xi|^{2s}\,K_e(\theta)\,d\theta,
\end{equation}
\begin{equation}\label{Fourier-B}
{\mathcal{B}}(\xi)=\begin{dcases}\displaystyle
-|\Gamma(-2s)|\sin(\pi s)\,\int_{S^{n-1}} |\theta\cdot \xi|^{2s-1}(\theta\cdot\xi)\,K_o(\theta)
\,d\theta &\qquad{\mbox{if }}s\neq {\textstyle\frac12},\\ \displaystyle
\int_{S^{n-1}} (\theta\cdot\xi) \,\log|\theta\cdot\xi| \,K_o(\theta)\,d\theta\,+\,b\cdot \xi
&\qquad{\mbox{if }}s={\textstyle\frac12}.\end{dcases}
\end{equation}
In particular, we have that $\mathcal A(\xi)>0$ for $\xi\neq0$.
\end{lem}

\begin{proof}
Let us consider the case $s\in (0,\frac12)$.
Then, we have
\[\widehat{Lu}(\xi) = \int_{\R^n}\big(1-e^{i\xi\cdot y}\big)K(y)dy\, \hat u(\xi) \]
and
\[\int_{\R^n}(1-e^{i\xi\cdot y})K(y)dy= \int_{\R^n}\big(1-\cos(\xi\cdot y)\big)K_e(y)dy -i \int_{\R^n}\sin(\xi\cdot y)\,K_o(y)dy,\]
where $K_e$ and $K_o$ denote the even and odd parts of the kernel,
as given in~\eqref{SPLIT}.
Therefore,
\[\begin{split}
\mathcal A(\xi)& = \int_{\R^n}\big(1-\cos(\xi\cdot y)\big)K_e(y)dy = \int_{S^{n-1}}\int_0^\infty \big(1-\cos(\xi\cdot \theta\,r)\big)\frac{K_e(\theta)}{r^{1+2s}}\,dr\,d\theta \\
& = \int_{S^{n-1}}\int_0^\infty \big(1-\cos t\big)\frac{K_e(\theta)}{t^{1+2s}}\,|\xi\cdot \theta|^{2s}\,dt\,d\theta = c_s  \int_{S^{n-1}}|\xi\cdot \theta|^{2s}\,K_e(\theta)\,d\theta,
\end{split}\]
where $c_s = \int_0^\infty (1-\cos t)\,t^{-1-2s}\,dt=|\Gamma(-2s)|\cos(\pi s)$.

Moreover, 
\[\begin{split}
-\mathcal B(\xi)& = \int_{\R^n}\sin(\xi\cdot y)\,K_o(y)dy = \int_{S^{n-1}}\int_0^\infty \sin(\xi\cdot \theta\,r)\,\frac{K_o(\theta)}{r^{1+2s}}\,dr\,d\theta \\
& = \int_{S^{n-1}}\int_0^\infty \sin t\,\frac{K_o(\theta)}{t^{1+2s}}\,|\xi\cdot \theta|^{2s-1}(\xi\cdot\theta)\,dt\,d\theta = \tilde c_s  \int_{S^{n-1}}|\xi\cdot \theta|^{2s-1}(\xi\cdot \theta)\,K_o(\theta)\,d\theta,
\end{split}\]
where $\tilde c_s = \int_0^\infty \sin t\ t^{-1-2s}\,dt=|\Gamma(-2s)|\sin(\pi s)$.

In the case $s\in(\frac12,1)$ the proof is basically the same.
Finally, for $s=\frac12$, we have ---using \eqref{wnhiowhgowgerg9ewg}--- 
\[\widehat{Lu}(\xi) = \int_{\R^n}\big(1-e^{i\xi\cdot y}+i\,\xi\cdot y\chi_{B_1}(y)\big)K(y)dy\, \hat u(\xi),\]
so that
\[\begin{split}
-\mathcal B(\xi)& = \int_{\R^n}\big(\sin(\xi\cdot y)-(\xi\cdot y)\chi_{B_1}(y)\big)\,K_o(y)dy \\&= \int_{S^{n-1}}\int_0^\infty \big(\sin(\xi\cdot \theta\,r)-(\xi\cdot \theta\,r)\chi_{(0,1)}(r)\big)\,\frac{K_o(\theta)}{r^{2}}\,dr\,d\theta \\
& = \int_{S^{n-1}}\int_0^\infty \big(\sin t - t\,\chi_{(0,|\xi\cdot\theta|)}(t)\big)\frac{K_o(\theta)}{t^{2}}\,(\xi\cdot\theta)\,dt\,d\theta \\&= \int_{S^{n-1}}\int_0^\infty  t\,\chi_{(|\xi\cdot\theta|,1)}(t)\,\frac{K_o(\theta)}{t^{2}}\,(\xi\cdot\theta)\,dt\,d\theta,
\end{split}\]
where we used that 
\[\int_{S^{n-1}}\int_0^\infty  \big(\sin t-t\,\chi_{(0,1)}(t)\big)\,\frac{K_o(\theta)}{t^{2}}\,(\xi\cdot\theta)\,dt\,d\theta = C\int_{S^{n-1}} K_o(\theta)\,(\xi\cdot\theta)\,d\theta = 0,\]
by assumption.
Hence, since 
\[\int_0^\infty  t\,\chi_{(|\xi\cdot\theta|,1)}(t)\,\frac{1}{t^{2}}\,dt = \int_{|\xi\cdot\theta|}^1 \frac{dt}{t} = -\log |\xi\cdot\theta|,\]
we deduce that
\[\mathcal B(\xi) = \int_{S^{n-1}} (\xi\cdot\theta)\,\log|\xi\cdot \theta|\,K_o(\theta)\,d\theta,\]
as claimed.
\end{proof}

Notice that, as a consequence of Lemma~\ref{AeBFOU},
one can show that $\mathcal A(\xi)$ and $\mathcal B(\xi)$ are H\"older continuous.
Moreover, if $K$ is smooth outside the origin, then $\mathcal A$ and $\mathcal B$ are smooth outside the origin, too.

The following result will be useful when evaluating the operators $L$ on 1D functions.

\begin{lem}\label{polar-coordinates}
Let $L$ be any operator of the form \eqref{operator-L}-\eqref{ellipt-const}.
For any~$\nu\in S^{n-1}$, let $\mathcal A(\nu)$ and $\mathcal B(\nu)$ be given by \eqref{Fourier-A}-\eqref{Fourier-B}.
Then,
\begin{equation}\label{IND}
\inf_{\nu\in S^{n-1}} \big\{\mathcal A(\nu)-\cot(\pi s)\,|\mathcal B(\nu)|\big\}>0.
\end{equation}
Furthermore, if~$u=u(x_n)$,
\begin{equation}\label{POL}
Lu(x)=\begin{dcases}
c_s\int_{\R}
\big(u(x_n)-u(x_n+r)\big)\,\frac{\mathcal A(e_n)-\cot(\pi s)\mathcal B(e_n)\,\sign r}{|r|^{1+2s}}\,dr
& {\mbox{ if }} s\in\left(0,{\textstyle\frac12}\right),
\\ c_{s}\,{\rm P.V.}
\int_{\R}
\big(u(x_n)-u(x_n+r)\big)\,\frac{\mathcal A(e_n)}{|r|^{2}}\,dr
+\mathcal B(e_n)\,\partial_n u(x)
& {\mbox{ if }} s={\textstyle\frac12},\\
c_s\int_{\R}
\big(u(x_n)-u(x_n+r)+
r\,\partial_n u(x_n)\big)\,\frac{\mathcal A(e_n)-\cot(\pi s)\mathcal B(e_n)\,\sign r}{|r|^{1+2s}}\,dr
& {\mbox{ if }} s\in\left({\textstyle\frac12},1\right).
\end{dcases}
\end{equation}
\end{lem}

\begin{proof}
First, notice that the inequality $\mathcal A(\nu)-\cot(\pi s)\,|\mathcal B(\nu)|\geq0$ follows from the expressions for $\mathcal A(\nu)$ and $\mathcal B(\nu)$, \eqref{Fourier-A}-\eqref{Fourier-B}, and the fact that $K\geq0$.
The strict inequality follows from the fact that $K$ cannot be supported on a hyperplane.

The proof of \eqref{POL} is somewhat similar to that of Lemma \ref{AeBFOU}; we give the proof in the case $s\in(0,\frac12)$.
In such case, we have
\[\begin{split}
Lu(x)& = \int_{\R^n}\big(u(x)-u(x+y)\big)K(y)dy = \frac12\int_{S^{n-1}}\int_{-\infty}^\infty \big(u(x_n)-u(x+r\theta_n)\big)\frac{K_e(\theta)+\textrm{sign}(r)K_o(\theta)}{|r|^{1+2s}}\,dr\,d\theta \\
& = \frac12\int_{S^{n-1}}\int_{-\infty}^\infty \big(u(x_n)-u(x_n+t)\big)\frac{K_e(\theta)+\textrm{sign}(t)K_o(\theta)}{|t|^{1+2s}}\,|\theta_n|^{2s}\,dt\,d\theta \\
& = c_s\int_{\R}\big(u(x_n)-u(x_n+t)\big)\,\frac{\mathcal A(e_n)-\cot(\pi s)\mathcal B(e_n)\,\sign t}{|t|^{1+2s}}\,dt,
\end{split}\]
where we used \eqref{Fourier-A}-\eqref{Fourier-B}.
\end{proof}

Now, we give a straightforward global integration by parts formula.

\begin{lem}\label{C:ADJ}
Let $L$ be any operator of the form \eqref{operator-L}-\eqref{ellipt-const}, and let $L^*$ be its adjoint operator.
Then, for every~$\eta$, $\tau\in C^\infty_c(\R^n)$, we have
$$ \int_{\R^n} L\eta(x)\,\tau(x)\,dx=
\int_{\R^n} \eta(x)\,L^*\tau(x)\,dx.$$
\end{lem}

\begin{proof} 
We employ Plancherel's Theorem and~\eqref{FS} to see that
\begin{eqnarray*}
&& \int_{\R^n} L\eta(x)\,\tau(x)\,dx-\int_{\R^n} \eta(x)\,L^*\tau(x)\,dx
= \int_{\R^n} \widehat{L\eta}(\xi)\,\overline{\hat\tau(\xi)}\,d\xi
-\int_{\R^n} \hat\eta(\xi)\,\overline{\widehat{L^*\tau}(\xi)}\,d\xi=\\
&&\qquad=
\int_{\R^n} \big({\mathcal{A}}(\xi)+i{\mathcal{B}}(\xi)\big)
\hat{\eta}(\xi)\,\overline{\hat\tau(\xi)}\,d\xi
-\int_{\R^n} \hat\eta(\xi)\,\overline{
\big({\mathcal{A}}(\xi)-i{\mathcal{B}}(\xi)\big)\hat\tau(\xi)}\,d\xi=0,\end{eqnarray*}
as desired.
\end{proof}

\section{Interior regularity}
\label{sec3}

The aim of this section is to prove Theorem \ref{thm-interior}.

\subsection{Liouville theorem}

We start with the following Liouville theorem.
Notice that the proof is very general, and does not really use the homogeneity of the kernel.
We believe that such new proof could have its own interest, as it could be useful in other situations.

\begin{thm}\label{Liouv-entire}
Let $s\in(0,1)$, and let $L$ be any operator of the form \eqref{operator-L}-\eqref{ellipt-const} --- with $K$ possibly being a measure.

Let $u$ be any bounded weak solution of $Lu=0$ in $\R^n$.
Then, $u$ is constant.
\end{thm}

\begin{proof}
Let $p(t,x)$ be the heat kernel associated to the operator $L$. We will use the following properties of $p(x,t)$:
\begin{itemize}

\item[(i)] $\int_{\R^n}p(1,x)dx=1$.

\vspace{2mm}

\item[(ii)] $p(t,x)\geq0$ for all $t>0$ and all~$x\in \R^n$.

\vspace{2mm}

\item[(iii)]  $\partial_t p +Lp=0$ in $(0,\infty)\times \R^n$, and $p(0,x)=\delta_0$.

\vspace{2mm}

\item[(iv)] The Fourier transform of $p$ is given by  $\displaystyle\hat{p}= e^{-t\{\mathcal A(\xi)+i\mathcal B(\xi)\}}$, being~$\mathcal A(\xi)+i\mathcal B(\xi)$
the Fourier symbol of~$L$.

\vspace{2mm}

\item[(v)]  $\|\nabla_x p(1,\cdot)\|_{L^\infty(\R^n)} \leq C$.

\vspace{2mm}

\end{itemize}

The first four properties are general facts that hold for any infinitesimal generator of a L\'evy process (see \cite{Bertoin}), while (v) follows directly from the decay of the Fourier transform $\hat p$.
Indeed, since $\mathcal A(\xi)\geq \lambda|\xi|^{2s}$ then we have that~$|\hat p|\leq e^{-\lambda t|\xi|^{2s}}$.

Let now $R\geq1$ and $u_R(x):=u(Rx)$.
Notice that $u_R$ is a weak solution $Lu_R=0$ in $\R^n$, and we may assume that
\begin{equation}\label{tfqgeuyberv}
\|u_R\|_{L^\infty(\R^n)}=\|u\|_{L^\infty(\R^n)}=1.\end{equation}

Now, it is not difficult to check that 
\begin{equation}\label{poi86575}
u_R \ast p(1,\cdot\,) - u_R   = \int_0^1 u_R\ast \partial_t p\,dt
= -\int_0^1 u_R\ast Lp \,dt= -\int_0^1 Lu_R\ast p\,dt=0,\end{equation}
where $\ast$ denotes the convolution.

Furthermore, in light of~$(v)$, we have
that~$\big|p(1,x-y)-p(1,x'-y)\big|\leq C|x-x'|$ for all~$x$, $x'$ and $y\in \R^n$.
As a consequence of this, (ii),
\eqref{tfqgeuyberv} and~\eqref{poi86575}, given~$x$, $x'\in \R^n$,
we see that
\[\big|u_R(x)-u_R(x')\big| =\left|\int_{\R^n}\big(p(1,x-y)-p(1,x'-y)\big) u_R(y)dy \right|\leq CM^n|x-x'|+ 2\int_{|y|\geq M}p(1,\cdot).\]
Hence, setting $M:=|x-x'|^{-\frac{1}{2n}}$, we get 
\[\big|u_R(x)-u_R(x')\big| \leq C|x-x'|^{1/2} + 2\int_{|y|\geq |x-x'|^{-\frac{1}{2n}}}p(1,\cdot) := \omega(|x-x'|). \]
Notice that, since $p(1,\cdot)\in L^1(\R^n)$, then $\omega(|x-x'|)\to 0$ as $|x-x'|\to 0$.

Recalling that $u_R(x)=u(Rx)$, this means that
\[\big|u(Rx)-u(Rx')\big| \leq \omega(|x-x'|)\qquad\textrm{for all}\quad x,x'\in \R^n,\]
or equivalently,
\[\big|u(x)-u(x')\big| \leq \omega\left(\frac{|x-x'|}{R}\right)\qquad\textrm{for all}\quad x,x'\in \R^n.\]
Letting $R\to\infty$, we find that $u(x)=u(x')$ for all $x,x'\in \R^n$.
\end{proof}

\subsection{Interior Schauder estimates}

We will need the following.

\begin{lem}\label{lem-subseq}
Let $s\in(0,1)$, and let $\lambda$ and $\Lambda$ be fixed positive constants.
Let $\{L_k\}_{k\geq1}$ be any sequence of operators of the form \eqref{operator-L} satisfying \eqref{ellipt-const}.

Then, a subsequence of $\{L_k\}$ converges weakly to an operator $L$ of the form \eqref{operator-L}-\eqref{ellipt-const}.
More precisely, if $L_k$ have kernels\footnote{or more precisely L\'evy measures, which are $(-n-2s)$-homogeneous Radon measures in $\R^n$.} $K_k$ then, up to a subsequence, $K_k|_{\mathbb S^{n-1}}$ converge (weakly$^*$ as Radon measures) to $K|_{\mathbb S^{n-1}}$ for some kernel $K$ satisfying \eqref{ellipt-const}  --- with $K$ possibly being a measure. 
Moreover, in the case $s=\frac12$, then the drift terms $b_k\in \R^n$ converge to $b\in \R^n$.

Moreover, assume that $(u_k)$ and $(f_k)$ are sequences of functions satisfying, in the weak sense,
\[L_k u_k = f_k \ \textrm{ in } \ \Omega\]
for a given bounded domain $\Omega \subset\R^n$.
Assume also that
\begin{enumerate}
\item $u_k\to u$ uniformly in compact sets of $\R^n$,
\item $f_k \to f$ uniformly in $\Omega$,
\item $|u_k(x)| \leq M\left(1+|x|^{2s-\epsilon}\right)$ for some $\epsilon > 0$, and for all $x\in \R^n$.
\end{enumerate}
Then, $u$ satisfies
\[Lu = f \textrm{ in } \Omega\]
in the weak sense, where $L$ is the operator associated to $K$.
\end{lem}

\begin{proof}
Recall that the kernels $K_k$ are homogeneous, so that they are determined by $K_k|_{\mathbb S^{n-1}}$.
Then, using the weak compactness of probability measures on $\mathbb S^{n-1}$, we find that up to a subsequence $K_k|_{\mathbb S^{n-1}}$ converges to a measure $\mu$ on the sphere.
By homogeneity, this determines a kernel $K$, and by taking limits we find that $K$ satisfies  \eqref{ellipt-const}.

Now, the weak formulation of the equations for $u_k$ are
\[\int_{\R^n}  u_k L_k^* \eta = \int_\Omega f_k\,\eta \quad\textrm{for all}\quad \eta\in C^\infty_c(\Omega).\]
Since $|\eta(x)-\eta(x+y)+\nabla \eta(x)\cdot y|\leq C\min\big\{1,\,|y|^2\big\}$, then by the dominated convergence theorem we find that $L_k^*\eta\to L^*\eta$ uniformly in compact sets of $\R^n$.

We now claim that\footnote{Notice that, when the kernels satisfy $K(y)\leq \Lambda|y|^{-n-2s}$ then we have the pointise bound $|L^*\eta|\leq C(1+|x|^{n+2s})^{-1}$. However, such pointwise bound is false for the more general class of kernels under consideration here.}
\[\int_{B_{2R}\setminus B_R} |L^*_k\eta| \leq CR^{-2s} \int_{\Omega} |\eta| \quad \mbox{for any $\eta\in C^\infty_c(\Omega)$ and $R>0$ such that $\Omega\subset B_{R/2}$}. \]
Indeed, by a simple approximation argument it suffices to prove it for $\eta\geq0$.
Let $\varphi\in C^\infty(\R^n, [0,1])$ with $\varphi\equiv1$
in $B_1^c$ and $\varphi\equiv0$ in $B_{1/2}$, and let $\varphi_R(x):=\varphi(x/R)$. 
Since $|L_k\varphi|\leq C$ in $\R^n$, then\footnote{Here we are using that  $L$ satisfies the scaling property $L_k (u(r\,\cdot)\,)  = r^{2s}(L_k u)(r\,\cdot)$.} $|L_k\varphi_R|\leq CR^{-2s}$.
But then, since $L^*_k\eta\leq 0$ in $\Omega^c$ and $\Omega\subset B_{R/2}$ then
\[\int_{B_{2R}\setminus B_R} |L^*_k\eta| 
\leq -\int_{\Omega^c}L^*_k\eta 
=
-\int_{\R^n}\varphi_R\,L^*_k\eta = -\int_{\R^n}L_k\varphi_R\, \eta\leq CR^{-2s}\int_\Omega \eta,\]
as claimed.

Therefore, by the growth of $u_k$ we deduce that
\[\int_{\R^n\setminus B_R} |u_k\, L^*_k\eta| \leq CR^{-\varepsilon},\]
and then by the Vitali convergence theorem
\[\int_{\R^n} u_k\, L^*_k\eta \longrightarrow \int_{\R^n} u\, L^*\eta. \]

Finally, since 
\[\int_\Omega f_k\eta\longrightarrow \int_\Omega f\eta,\]
then it follows that $Lu=f$ in $\Omega$ in the weak sense, as
desired.
\end{proof}

We next establish the following result, which is the main step towards Theorem \ref{thm-interior} (b) and (c).

\begin{prop}\label{claim-a}
Let $s\in(0,1)$, and let $L$ be any operator of the form \eqref{operator-L}-\eqref{ellipt-const}.
Let $\alpha\in(0,1)$ be such that $\alpha+2s$ is not an integer, and $p=\lfloor \alpha+2s\rfloor$ (the integer part).

Let $u$ be any $C_c^\infty(\R^n)$ function satisfying $L u=f$ in $B_1$.
Then, for any $\delta>0$ we have the estimate
\begin{equation}\label{estw}
[u]_{C^{\alpha+2s}(B_{1/2})} \le \delta [u]_{C^{\alpha+2s}(\R^n)} + C\bigl( [f]_{C^\alpha(B_1)} + \|u\|_{C^p(B_1)}\bigr).
\end{equation}
The constant $C$ depends only on $n$, $s$, $\alpha$, $\delta$, and the ellipticity constants in~\eqref{ellipt-const}.
\end{prop}

\begin{proof}
Assume by contradiction that there exists a sequence $u_k\in C^\infty_c(\R^n)$ such that
$ L_ku_k =f_k $  in $B_{1}$, where~$L_k$~is of the form \eqref{operator-L}-\eqref{ellipt-const}, and $$
[u_k]_{C^{\alpha+2s}(B_{1/2})} > \delta [u_k]_{C^{\alpha+2s}(\R^n)} + k\bigl( [f_k]_{C^\alpha(B_1)} + \|u_k\|_{C^p(B_1)}\bigr).$$

Let $x_k,y_k\in B_{1/2}$ such that 
\[\frac12[u_k]_{C^{\alpha+2s}(B_{1/2})} \leq \frac{|D^p u_k(x_k) - D^p u_k(y_k)|}{|x_k-y_k|^{\alpha+2s-p}}\]
and define 
\[\rho_k := |x_k-y_k|.\]
Notice that, by assumption,
\[\frac12[u_k]_{C^{\alpha+2s}(B_{1/2})} \leq \frac{2\|D^p u_k\|_{L^\infty(B_{1/2})}}{\rho_k^{\alpha+2s-p}} \leq \frac2k \frac{[u_k]_{C^{\alpha+2s}(B_{1/2})}}{\rho_k^{\alpha+2s}},\]
and thus $\rho_k\to0$.

Next, define the blow-up sequence 
\begin{equation}\label{eqvm}
 v_k(x) := \frac{u_k(x_k+\rho_k x)-p_k(x)}{\rho_k^{{2s}+\alpha}[u_k]_{C^{\alpha+2s}(\R^n)}},
 \end{equation}
where $p_k$ is the Taylor polynomial of order $p$, so that 
\[v_k(0)=...=|D^p v_k(0)|=0.\]
Notice that 
\[[v_k]_{C^{\alpha+2s}(\R^n)} \leq 1.\]

Now, for all $x\in B_{1/(2\rho_k)}$ and $h\in B_1$ we have 
\begin{equation}\label{po7456hgfgfsj}
\big|L_k\big(v_k(x+h)-v_k(x)\big) \big| = \frac{\big| f_k(x_k+\rho_k x+\rho_k h) - f_k(x_k+\rho_k x)\big|}{\rho_k^\alpha [u_k]_{C^{\alpha+2s}(\R^n)} }
\leq  \frac{[f_k]_{C^{\alpha}(B_1)} |h|^\alpha}{ [u_k]_{C^{\alpha+2s}(\R^n)} } \leq \frac1k \longrightarrow 0.\end{equation}

On the other hand, let $z_k:= \frac{x_k-y_k}{\rho_k}\in S^{n-1}$.
Then,
\[\big|D^p v_k(z_k)\big| = |D^p v_k(z_k)- D^p v_k(0)| = \frac{|D^p u_k(y_k)-D^pu_k(x_k)|}{\rho_k^{\alpha+2s-p} [u_k]_{C^{\alpha+2s}(\R^n)}} > \frac{\frac12 [u_k]_{C^{\alpha+2s}(B_{1/2})}}{[u_k]_{C^{\alpha+2s}(\R^n)}} > \frac{\delta}{2}.\]

Up to a subsequence, we will have that $z_k\to z\in S^{n-1}$, and that $v_k\to v$ in $C^{\alpha'+2s}_{\rm loc}(\R^n)$ for any $\alpha'<\alpha$, to some function $v$ satisfying
\[\big|D^p v(z)\big| \geq \frac{\delta}{2}, \qquad v(0)=...=|D^p v(0)|=0, \qquad 
[v]_{C^{\alpha+2s}(\R^n)}\leq1.\]
(Here we used that $p<\alpha+2s$, i.e., $\alpha+2s$ is not an integer.)

Define 
\[U(x):= \begin{dcases}
\displaystyle v(x+h)-v(x) \qquad& {\mbox{ if }}p=0,\\
v(x+2h)-2v(x+h)+v(x) \qquad& {\mbox{ if }}p=1,\\
v(x+3h)-3v(x+2h)+3v(x+h)-v(x) \qquad& {\mbox{ if }}p=2,
\end{dcases}\]
and define $U_k$ analogously (with~$v$ replaced by~$v_k$).

Then, since $[v]_{C^{\alpha+2s}(\R^n)}\leq1$, we have 
\[|U(x)|\leq C|h|^{\alpha+2s}, \qquad |U_k(x)|\leq C|h|^{\alpha+2s},\]
and in particular they are bounded in $\R^n$.

Moreover, since $|L_kU_k|\to0$ in $B_{1/(2\rho_k)}$ and $U_k\to U$ uniformly in compact sets, then by Lemma \ref{lem-subseq} it follows that $LU=0$ in $\R^n$.
But then, by Theorem \ref{Liouv-entire} we have that $U$ must be constant.
This implies that $v$ is a polynomial (of degree~$p+1$),
which combined with $v(0)=...=|D^p v(0)|=0$ and $[v]_{C^{\alpha+2s}(\R^n)}\leq1$ yields that $v\equiv0$ in $\R^n$, and this is a contradiction with $\big|D^p v(z)\big| \geq \frac{\delta}{2}$.
\end{proof}

\begin{rem}\label{Zygmund-space}
For $\alpha\in(0,2)$ the Zygmund space $\Lambda^\alpha(\overline\Omega)$ is defined via the norm
$\|w\|_{\Lambda^\alpha(\Omega)} := [w]_{\Lambda^\alpha(\Omega)} + \|w\|_{L^\infty(\Omega)}$, where
\[[w]_{\Lambda^\alpha(\Omega)} := \sup_{x,x\pm h\in \Omega} \frac{|w(x+h)+w(x-h)-2w(x)|}{|h|^\alpha}.\]

Such norm is equivalent to $C^\alpha$ if $\alpha\neq1$, but they are not in case $\alpha=1$; see for example \cite[Chapter~V]{Stein} or \cite[Appendix A]{book}.
For $\alpha=1$, we have the inequality $\|w\|_{C^{1-\epsilon}(\Omega)}\leq C_\epsilon\|w\|_{\Lambda^1(\Omega)}$ for any $\epsilon\in(0,1)$.
\end{rem}

We will also need the following:

\begin{prop}\label{claim-a2}
Let $s\in(0,1)$ and let $L$ be any operator of the form \eqref{operator-L}-\eqref{ellipt-const}.

Let $u$ be any $C_c^\infty(\R^n)$ function satisfying $L u=f$ in $B_1$.
Then, for any $\delta>0$ we have the estimate
\[
[u]_{\Lambda^{2s}(B_{1/2})} \le \delta [u]_{\Lambda^{2s}(\R^n)} + C\bigl( [f]_{L^\infty(B_1)} + \|u\|_{L^\infty(B_1)}\bigr).
\]
The constant $C$ depends only on $n$, $s$, $\delta$, and the ellipticity constants in~\eqref{ellipt-const}.
\end{prop}

\begin{proof}
The proof is very similar to that of Proposition \ref{claim-a} (but with $\alpha=0$).

Assume by contradiction that there exists a sequence $u_k\in C^\infty_c(\R^n)$ such that
$ L_ku_k =f_k $  in $B_{1}$, where~$L_k$~is of the form \eqref{operator-L}-\eqref{ellipt-const}, and $$
[u_k]_{\Lambda^{2s}(B_{1/2})} > \delta [u_k]_{\Lambda^{2s}(\R^n)} + k\bigl( [f_k]_{L^\infty(B_1)} + \|u_k\|_{L^\infty(B_1)}\bigr).$$

Take $x_k,y_k\in B_{1/2}$ such that
\[\frac12[u_k]_{\Lambda^{2s}(B_{1/2})} \leq \frac{|u_k(x_k) + u_k(y_k) -2u_k(\frac{x_k+y_k}{2})|}{|x_k-y_k|^{2s}},\]
define~$\rho_k:=|x_k-y_k|$
and let 
\begin{equation}\label{eqvm2}
 v_k(x) := \frac{u_k(\frac{x_k+y_k}{2}+\rho_k x)-p_k(x)}{\rho_k^{{2s}}[u_k]_{\Lambda^{2s}(\R^n)}},
 \end{equation}
where $p_k$ is the Taylor polynomial of order $p$ (where $p=0$ for $s\leq \frac12$ and $p=1$ for $s>1/2$), so that  $v_k(0)=...=|D^p v_k(0)|=0$.
Moreover, $[v_k]_{\Lambda^{2s}(\R^n)} \leq 1$.

As in~\eqref{po7456hgfgfsj}, one can check that for all $x\in B_{1/(2\rho_k)}$ and $h\in B_1$ 
\[\big|L_k\big(v_k(x+h)+v_k(x-h)-2v_k(x)\big) \big| \leq \frac{C}{k} \longrightarrow 0.\]
On the other hand, if we denote by~$z_k:= \frac{x_k-y_k}{2\rho_k}\in S^{n-1}$, then
\[\big|v_k(z_k)+v_k(-z_k)\big| > \frac{\delta}{2}.\]
Up to a subsequence, we will have that $z_k\to z\in S^{n-1}$ and $v_k\to v$ uniformly in compact sets, for some function $v$ satisfying
\[\big|v(z)+v(-z)\big| \geq \frac{\delta}{2}, \qquad v(0)=...=|D^p v(0)|=0, \qquad 
[v]_{\Lambda^{2s}(\R^n)}\leq1.\]

Define  $U(x):= v(x+h)+v(x+h)-2v(x)$ and $U_k(x):=v_k(x+h)+v_k(x+h)-2v_k(x)$.
Then, since $[v_k]_{\Lambda^{2s}(\R^n)}\leq1$, we have that~$|U(x)|\leq C|h|^{2s}$ and $|U_k(x)|\leq C|h|^{2s}$.

Moreover, since $|L_kU_k|\to 0$ in $B_{1/(2\rho_k)}$ and $U_k\to U$ uniformly in compact sets, then by Lemma \ref{lem-subseq} it follows that $LU=0$ in $\R^n$.
But then, by Theorem \ref{Liouv-entire} we have that $U$ must be constant.
Since this holds for all $h\in B_1$ then this implies that $v$ is a quadratic polynomial, and since $[v]_{\Lambda^{2s}(\R^n)}\leq1$ then $v$ is an affine function, which is in contradiction with $\big|v(z)+v(-z)\big| \geq \frac{\delta}{2}$.
\end{proof}

We can now give the:

\begin{proof}[Proof of Theorem \ref{thm-interior}]
The proof is based on Propositions \ref{claim-a} and \ref{claim-a2}.
We give the proof of parts (b) and (c) of Theorem~\ref{thm-interior}, being part (a) analogous.

Assume first that~$u\in C^\infty(\R^n)$, and let $\eta\in C^\infty_c(B_2)$ such that $\eta\equiv1$ in $B_{3/2}$.
Then, thanks to Proposition~\ref{claim-a}, we know that for any $\delta>0$ we have 
\[[u]_{C^{\alpha+2s}(B_{1/2})} \le \delta [u\eta]_{C^{\alpha+2s}(\R^n)} + C_\delta\bigl( [L(u\eta)]_{C^\alpha(B_1)} + \|u\eta\|_{C^p(B_1)}\bigr),\]
where $p=\lfloor \alpha+2s\rfloor$.

Now, we observe that
\begin{equation}\label{ewytr7487}
L(u\eta)= f + L(u-u\eta)\quad \textrm{in}\quad B_.\end{equation}
Since~$u- u\eta\equiv 0$ in $B_{3/2}$, it is easy to see that
\begin{equation}\label{sdfg1}
[L(u\eta-u)]_{C^\alpha(B_1)}\leq C[u]_{C^\alpha(\R^n)}.
\end{equation}
In the case in which $K\in C^\alpha(S^{n-1})$ --that is, part (c) of Theorem~\ref{thm-interior}--- then one can get 
\begin{equation}\label{sdfg2}
[L(u\eta-u)]_{C^\alpha(B_1)}\leq C\|u\|_{L^\infty(\R^n)}.
\end{equation}
Now, since the rest of the proof is the same in both cases, we will only prove part~(b).

To this aim, we notice that, thanks to~\eqref{ewytr7487} and~\eqref{sdfg1}, and using that $[u\eta]_{C^{\alpha+2s}(\R^n)}\leq C\|u\|_{C^{\alpha+2s}(B_2)}$ and $\|u\eta\|_{C^p(B_1)}\leq C\|u\|_{C^p(B_1)}$, for any $\delta>0$ we have
\[[u]_{C^{\alpha+2s}(B_{1/2})} \le C\delta \|u\|_{C^{\alpha+2s}(B_2)} + C_\delta\bigl( [f]_{C^\alpha(B_1)} +  \|u\|_{C^\alpha(\R^n)} + \|u\|_{C^p(B_1)}\bigr).\]
Moreover, by interpolation inequalities,
\[\|u\|_{C^p(B_2)} \leq \delta [u]_{C^{\alpha+2s}(B_2)}+ C_\delta\|u\|_{L^\infty(B_2)},\]
and hence we have
\[[u]_{C^{\alpha+2s}(B_{1/2})} \le C\delta [u]_{C^{\alpha+2s}(B_2)} + C_\delta\bigl( [f]_{C^\alpha(B_1)} +  \|u\|_{C^\alpha(\R^n)}\bigr).\]

Furthermore, whenever $B_{2\rho}(z)\subset B_1$ we can repeat the previous argument replacing $u$ by  $\tilde u(x) : = u(z+\rho x)$, which solves $L\tilde u =\tilde f$ in $B_1$
with right hand side $\tilde f(x): = \rho^{2s}f(z+\rho x)$.
Noting that  $[\tilde f]_{C^\alpha(B_1)} = \rho^{2s+\alpha}[ f]_{C^\alpha(B_\rho(z))} \le    [f]_{C^\alpha(B_1)}$ and  similarly $\|\tilde u\|_{C^\alpha(\R^n)} \le \|\tilde u\|_{C^\alpha(\R^n)}$ and using $[\tilde u]_{C^{\alpha+2s}(B_{1/2})} = \rho^{\alpha+2s}[\tilde u]_{C^{\alpha+2s}(B_{r/2(x)})}$, we obtain
\[ \rho^{\alpha+2s}[u]_{C^{\alpha+2s}(B_{r/2(z)})} \le C\delta \rho^{\alpha+2s} [u]_{C^{\alpha+2s}(B_{2r}(z))} + C_\delta\bigl( [f]_{C^\alpha(B_1)} +  \|u\|_{C^\alpha(\R^n)}\bigr).\]
Since this is valid for any $\delta>0$ and for all balls $B_{2\rho}(z)\subset B_1$, we conclude that
\[[u]_{C^{\alpha+2s}(B_{1/2})} \le C\bigl( [f]_{C^\alpha(B_1)} +  \|u\|_{C^\alpha(\R^n)}\bigr);\]
see \cite[Lemma 2.23]{book} or the proof of Theorem 1.1 in \cite{RS-elliptic}.

Finally, by using a standard approximation argument (see Remark \ref{approxX} below), the result follows for any weak solution $u\in C^\alpha(\R^n)$, and thus we are done.
\end{proof}

\begin{rem}\label{approxX}
Thanks to the translation invariance of the operator, we have the following useful fact.
If $u$ satisfies $Lu=f$ in $\Omega$ in the weak sense
and $\eta_\epsilon\in C^\infty_c(B_\epsilon)$ is a mollifier, then $u_\epsilon:=u\ast \eta_\epsilon$ satisfies $Lu_\epsilon=f\ast\eta_\epsilon$ in $\Omega_\epsilon$ in the weak sense, where $\Omega_\epsilon=\Omega\cap \{\textrm{dist}(x,\partial\Omega)>\epsilon\}$.
\end{rem}

Finally, we give an immediate consequence of Theorem \ref{thm-interior} that will be used later.

\begin{cor}\label{cor-interior-growth}
Let $s\in(0,1)$ and let $L$ be any operator of the form \eqref{operator-L}-\eqref{ellipt-const}.
Let $u$ be any weak solution of
\[L u =f\quad {in}\ B_1,\]
with $f\in L^\infty(B_1)$.
Then, for any $\varepsilon>0$,
\[\|u\|_{C^{2s}(B_{1/2})}\leq C\left(\sup_{R\geq1}\left\{R^{\varepsilon-2s}\|u\|_{L^\infty(B_R)}\right\}+\|f\|_{L^\infty(B_1)}\right)\quad \textrm{if}\ s\neq\frac12,\]
and
\[\|u\|_{C^{2s-\epsilon}(B_{1/2})}\leq C\left(\sup_{R\geq1}\left\{R^{\varepsilon-2s}\|u\|_{L^\infty(B_R)}\right\}+\|f\|_{L^\infty(B_1)}\right)\quad \textrm{if}\ s=\frac12.\]
The constant $C$ depends only on $n$, $s$, $\epsilon$, and the ellipticity constants in~\eqref{ellipt-const}.

Moreover, if $K\in C^\alpha(\R^n\setminus\{0\})$ and $f\in C^\alpha(B_1)$, then 
\[\|u\|_{C^{\alpha+2s}(B_{1/2})}\leq C\left(\sup_{R\geq1}\left\{R^{\varepsilon-2s}\|u\|_{L^\infty(B_R)}\right\}+\|f\|_{C^\alpha(B_1)}\right),\]
provided that $\alpha+2s$ is not an integer.
The constant $C$ depends only on $n$, $s$, $\alpha$, $\epsilon$, $\|K\|_{C^\alpha(S^{n-1})}$, and the ellipticity constants in~\eqref{ellipt-const}.
\end{cor}

\begin{proof}
The proof follows by using that the truncated function $\tilde u=u\chi_{B_2}$ satisfies the hypotheses of Theorem \ref{thm-interior} (a) or (c).
\end{proof}

\subsection{Non-homogeneous kernels}

It is important to notice that the proofs presented above do not really require the kernel $K(y)$ to be homogeneous, and in particular one can prove the same results under the following assumptions:

\begin{equation}\label{non-homogeneous0}
K\in L^1_{\rm loc}(\R^n\setminus\{0\}),\quad \textrm{with}\quad K\geq0\, \  \textrm{in}\ \,\R^n,
\end{equation}
\begin{equation}\label{non-homogeneous}
0<\lambda \leq \inf_{\nu\in S^{n-1}}\int_{B_{2r}\setminus B_r} |y\cdot\nu|^{2s} K(y)\,dy\qquad \textrm{and}\qquad r^{2s}\int_{B_{2r}\setminus B_r} K(y)\,dy \leq \Lambda
\end{equation}
for every $r>0$, and 
\begin{equation}\label{non-homogeneous2}
\int_{B_{2r}\setminus B_r} y\,K(y)\,dy =0\qquad \textrm{and}\qquad |b|\leq \Lambda,\qquad \textrm{if} \quad s={\textstyle\frac12},
\end{equation}
for every $r>0$.

\begin{thm}\label{thm-interior-nonhomogeneous}
Let $s\in(0,1)$ and let $L$ be any operator of the form \eqref{operator-L} satisfying \eqref{non-homogeneous0}-\eqref{non-homogeneous2} for all~$r>0$.
Let $u$ be any bounded weak solution to
\begin{equation}\label{eq-ball}
L u=f\quad {in}\ B_1.
\end{equation}
Then,
\begin{itemize}
\item[(a)] If $f\in L^\infty(B_1)$, 
\[\|u\|_{C^{2s}(B_{1/2})}\leq C\left(\|u\|_{L^\infty(\R^n)}+\|f\|_{L^\infty(B_1)}\right)\quad \textrm{if}\ s\neq{\textstyle\frac12}\]
and
\[\|u\|_{\Lambda^1(B_{1/2})}\leq C\left(\|u\|_{L^\infty(\R^n)}+\|f\|_{L^\infty(B_1)}\right)\quad \textrm{if}\ s={\textstyle\frac12},\]
where $\Lambda^1$ is the Zygmund space.
In particular, $u\in C^{2s-\varepsilon}(B_{1/2})$ for all $\varepsilon>0$.

\vspace{2mm}

\item[(b)] If $f\in C^\alpha(B_1)$ for some $\alpha>0$, 
\begin{equation}\label{estimate-thm-1-b}
\|u\|_{C^{\alpha+2s}(B_{1/2})}\leq C\left(\|u\|_{C^\alpha(\R^n)}+\|f\|_{C^\alpha(B_1)}\right)
\end{equation}
whenever $\alpha+2s$ is not an integer.

\vspace{2mm}

\item[(c)] If $f\in C^\alpha(B_1)$ for some $\alpha>0$ and in addition $[K]_{C^\alpha(\R^n\setminus B_r)} \leq \Lambda \,r^{-n-2s-\alpha}$ for all $r>0$, 
\begin{equation}\label{estimate-thm-1-c}
\|u\|_{C^{\alpha+2s}(B_{1/2})}\leq C\left(\|u\|_{L^\infty(\R^n)}+\|f\|_{C^\alpha(B_1)}\right)
\end{equation}
whenever $\alpha+2s$ is not an integer.
\end{itemize}

The constant $C$ depends only on $n$, $s$, $\alpha$, and the ellipticity constants.
\end{thm}

\begin{proof}
First, notice that the Fourier symbol of the operator $L$ will be of the form $\mathcal A(\xi)+i\,\mathcal B(\xi)$, with $\mathcal A$ and~$\mathcal{B}$ corresponding to the even part and to the odd part of the kernel, respectively.
Moreover, it is immediate to check that 
\[\lambda|\xi|^{2s} \lesssim \mathcal A(\xi)\lesssim \Lambda |\xi|^{2s} \qquad \textrm{and} \qquad |\mathcal B(\xi)|\lesssim \Lambda |\xi|^{2s},\]
and therefore the same proof of the Liouville Theorem \ref{Liouv-entire} holds.

On the other hand, we claim that the proof Lemma 3.2 works as well.
Indeed, if we have a sequence of kernels $K_k(y)$ satisfying our assumptions then 
\[d\mu_k(y) := \min\{1,\,|y|^2\}K_k(y)dy\]
are finite measures in $\R^n$.
Then, it is immediate to check that for every $\varepsilon>0$ there exists $R>1$ large enough so that $\mu_k(B_R)>1-\varepsilon$ for all $k$ (i.e., the mass cannot escape to infinity).
Thus, it follows (from Prokhorov's theorem, or from Banach-Alaoglu theorem) that, up to a subsequence, $\mu_k$ converge weakly to some probability measure $\mu$ in~$\R^n$.
Such measure $\mu$ will correspond to a kernel $K(y)$ (possibly a measure), and hence
\[\min\{1,\,|y|^2\}K_k(y)dy\quad \textrm{converges weakly to}\quad \min\{1,\,|y|^2\}K(y)dy,\]
up to subsequence.
The rest of the proof of the lemma is then identical.

Finally, thanks to the previous ingredients and the scale invariance of this class of kernels, the proof of the desired result is the same as in Theorem \ref{thm-interior}.
\end{proof}

\subsection{Interior Harnack inequality}

Finally, we will also need the following two results.
The first one is a half Harnack inequality, for subsolutions.

\begin{thm}\label{half-Harnack-sub}
Let $L$ be any operator of the form \eqref{operator-L}-\eqref{98988w019375=A}, satisfying in addition 
\begin{equation}\label{uniform-ell}
0<\frac{\lambda}{|y|^{n+2s}} \leq K(y) \leq \frac{\Lambda}{|y|^{n+2s}}.
\end{equation}

Assume that $u$ satisfies
\[Lu\leq C_0\quad \textrm{in}\ B_1\]
in the weak (or viscosity) sense, for some~$C_0>0$.
Then,
\[\sup_{B_{1/2}}u\leq C\left(\int_{\R^n}\frac{|u(x)|}{1+|x|^{n+2s}}\,dx+C_0\right).\]
The constant $C$ depends only on $n$, $s$, and the ellipticity constants.
\end{thm}

The second one is the other half Harnack inequality, for supersolutions.

\begin{thm} \label{half-Harnack-sup}
Let $L$ be any operator of the form \eqref{operator-L}-\eqref{98988w019375=A}, satisfying in addition \eqref{uniform-ell}.

Assume that $u$ satisfies
\[Lu\geq -C_0\quad \textrm{in}\ B_1\]
in the weak (or viscosity) sense, for some~$C_0>0$.
Assume in addition that $u\geq0$ in $\R^n$.
Then,
\[\int_{\R^n}\frac{u(x)}{1+|x|^{n+2s}}\,dx\leq C\left(\inf_{B_{1/2}}u+C_0\right).\]
The constant $C$ depends only on $n$, $s$, and the ellipticity constants.
\end{thm}

The following proof of Theorem \ref{half-Harnack-sup} is the same as the one presented in \cite{RS-bdryH} for symmetric operators in non-divergence form.
The proof works as well for non-symmetric operators; we present it here for completeness.

\begin{proof}[Proof of Theorem \ref{half-Harnack-sup}]
By a simple approximation argument, we may assume $u\in C^2(\R^n)$;
see Remark~\ref{approxX}.

Let $\eta\in C^\infty_c(B_{3/4})$ be such that $0\leq \eta\leq 1$ and $\eta\equiv1$ in $B_{1/2}$.
Let $t>0$ be the maximum value for which $u\geq t\eta$.
Notice that $t\leq \inf_{B_{1/2}}u$.
Moreover, since $u$ and $\eta$ are continuous in $B_1$, there exists $x_0\in B_{3/4}$ such that $u(x_0)=t\eta(x_0)$.

We observe that
\begin{equation}\label{ufgeruerug}
L(u-t\eta)(x_0) = Lu(x_0)-t\,L\eta \geq -C_0-C_1t,\end{equation}
for some~$C_0$, $C_1>0$.
Furthermore, since $u-t\eta\geq0$ in $\R^n$ and $(u-t\eta)(x_0)=0$ then
\[L(u-t\eta)(x_0)\leq -\Lambda\int_{\R^n}\frac{u(z)-t\eta(z)}{|x_0-z|^{n+2s}}dz
\leq -c\int_{\R^n}\frac{u(z)-t\eta(z)}{1+|z|^{n+2s}}\,dz
\leq -c\int_{\R^n}\frac{u(z)}{1+|z|^{n+2s}}dz+C_2t,\]
for some~$c$, $C_2>0$.
{F}rom this and \eqref{ufgeruerug}, we obtain that
\[ \inf_{B_{1/2}} u\geq t\geq -c_1C_0+c_2\int_{\R^n}\frac{u(z)}{1+|z|^{n+2s}}dz,\]
for some~$c_1$, $c_2>0$, as desired.
\end{proof}

The proof of Theorem \ref{half-Harnack-sub} that we present here is new.

When $s\geq\frac12$, a different proof of the result can be found in \cite[Corollary 6.2]{CD}, where it is proved in the more general setting of parabolic and nonsymmetric operators with drift.

\begin{proof}[Proof of Theorem \ref{half-Harnack-sub}]
First, by an approximation argument, we may assume that $u\in C^2(\R^n)$; recall Remark~\ref{approxX}.

Also, we may assume that $C_0=0$; otherwise we consider $\tilde u:=u-CC_0\eta$, with $\eta\in C^\infty_c(B_2)$, $\eta\equiv1$ in $B_1$.
Then, we have that~$L\eta \geq c>0$ in $B_1$, and hence $L\tilde u \leq C_0 -CC_0\,L\eta\leq 0$ in $B_1$, provided that $C$ is large enough.

Furthermore, we may assume that~$u\geq0$; otherwise we consider $u_+:=\max\{u,0\}$ instead of $u$, which satisfies $L(u_+)\leq 0$ as well.

In this setting, it suffices to prove that
\begin{equation}\label{eoirgeorig}
u(0) \leq C\int_{\R^n} \frac{u(x)}{1+|x|^{n+2s}}\,dx.
\end{equation}
Once this is proved, we simply apply this to every point in $B_{1/2}$, and the result follows.

We also observe that, after dividing $u$ by a constant,
we may assume in addition that 
\begin{equation}\label{soigjer}
\int_{\R^n} \frac{u(x)}{1+|x|^{n+2s}}\,dx=1.
\end{equation}

We now prove the following.

\vspace{3mm}

\noindent \textbf{Claim}. There exists $\delta>0$ such that if $x_\circ\in B_{1/2}$ satisfies $u(x_\circ)>M>>1$ then $\sup_{B_{r_M}(x_\circ)} u>(1+\delta)M$, for $r_M:= (c_1M)^{-1/n} <1/2$, and $c_1$ is small enough, depending only on $n$.

\vspace{3mm}

To prove the claim, we argue by contradiction and we suppose that~$\sup_{B_{r_M}(x_\circ)} u\leq (1+\delta)M$. We will reach a contradiction by using Theorem~\ref{half-Harnack-sup}.

Indeed, consider 
\[v(x) := \big((1+\delta)M - u(x_\circ+(r_M/2) x)\big)_+.\]
Then, $v\geq0$ everywhere, and
\begin{equation}\label{alsoriyhur}
v(x)\equiv (1+\delta)M - u(x_\circ+(r_M/2)x) \quad{\mbox{ for all }}x\in B_2.
\end{equation}
Thus, since $Lu\leq 0$ in $B_1$ by assumption, we have, for every~$x\in B_1$,
\[\begin{split}
Lv (x)& \geq L\big[\big((1+\delta)M - u(x_\circ+(r_M/2) x)\big)_-\big] \\
 & \geq  -\lambda \int_{\R^n\setminus B_2}\big((1+\delta)M - u(x_\circ+(r_M/2) y)\big)_-\;\frac{dy}{|x-y|^{n+2s}}\\
 &=-\frac{\lambda\, 2^n}{r_M^n} \int_{\R^n\setminus B_{r_M}(x_\circ)}\big((1+\delta)M - u(z)\big)_-\;\frac{dz}{\left|x-\frac2{r_M}(z-x_\circ)\right|^{n+2s}}.
\end{split}\]
Now we use the fact that
$$ \big((1+\delta)M - u(z)\big)_- =\begin{cases} u(z)-(1+\delta)M &{\mbox{ if }}
(1+\delta)M - u(z)<0,\\
0 &{\mbox{ if }}(1+\delta)M - u(z)\ge0
\end{cases}
$$
and we obtain that, for every~$x\in B_1$,
\begin{equation}\begin{split}\label{oet4ighty387}
Lv (x)&\geq
-\frac{\lambda\, 2^n}{r_M^n} \int_{{\R^n\setminus B_{r_M}(x_\circ)}\atop{\{
(1+\delta)M - u(z)<0\}}}\big(  u(z)-(1+\delta)M\big)\,\frac{dz}{\left|x-\frac2{r_M}(z-x_\circ)\right|^{n+2s}}\\
&\geq -\frac{\lambda\, 2^n}{r_M^n} \int_{{\R^n\setminus B_{r_M}(x_\circ)}} 
\frac{u(z)}{\left|x-\frac2{r_M}(z-x_\circ)\right|^{n+2s}}\,dz
\\&\qquad
+\frac{\lambda\, 2^n(1+\delta)M}{r_M^n} \int_{{\R^n\setminus B_{r_M}(x_\circ)}} 
\frac{dz}{\left|x-\frac2{r_M}(z-x_\circ)\right|^{n+2s}}.
\end{split}\end{equation}
Now we claim that, for every~$z\in \R^n\setminus B_{r_M}(x_\circ)$,
\begin{equation}\label{bv46b754vb4}
\left|x-\frac2{r_M}(z-x_\circ)\right|\geq C(1+|z|),
\end{equation}
for some~$C>0$ independent of~$M$.
Indeed, if~$z\in B_2$ then
$$ \left|x-\frac2{r_M}(z-x_\circ)\right|\geq \frac2{r_M}|z-x_\circ|-|x|\geq 2-|x|\geq1
\geq\frac12+\frac{|z|}4\geq \frac14(1+|z|),
$$
which proves~\eqref{bv46b754vb4} in this case. If instead~$z\in\R^n\setminus B_2$,
then~$|z-x_\circ|\ge |z|-|x_\circ|\ge |z|-|z|/4=(3/4)|z|$ and therefore
$$ \left|x-\frac2{r_M}(z-x_\circ)\right|\geq \frac2{r_M}|z-x_\circ|-|x|\geq
\frac3{2r_M}|z|-\frac12|z|\geq \frac52|z|\geq \frac52(1+|z|),
$$
which completes the proof of~\eqref{bv46b754vb4}.

As a consequence of~\eqref{bv46b754vb4} and~\eqref{soigjer},
we obtain that
\begin{equation}\label{lroti9867b676}
\int_{{\R^n\setminus B_{r_M}(x_\circ)}} 
\frac{u(z)}{\left|x-\frac2{r_M}(z-x_\circ)\right|^{n+2s}}\,dz\leq C_1
\int_{{\R^n\setminus B_{r_M}(x_\circ)}} 
\frac{u(z)}{1+|z|^{n+2s}}\,dz\le C_1,
\end{equation}
for some~$C_1>0$.
Furthermore, for every~$z\in \R^n\setminus B_{r_M}(x_\circ)$,
$$ \left|x-\frac2{r_M}(z-x_\circ)\right|\leq\frac2{r_M}|z-x_\circ|+|x|
\leq \frac2{r_M}|z-x_\circ|+1
\leq\frac3{r_M}|z-x_\circ|,$$
and thus
\begin{equation}\label{lroti9867b6762}
\int_{{\R^n\setminus B_{r_M}(x_\circ)}} 
\frac{dz}{\left|x-\frac2{r_M}(z-x_\circ)\right|^{n+2s}}
\geq \left(\frac{r_M}3\right)^{n+2s}\int_{{\R^n\setminus B_{r_M}(x_\circ)}} 
\frac{dz}{|z-x_\circ|^{n+2s}}= C_2 r_M^{n},
\end{equation}
for some~$C_2>0$.

Using~\eqref{lroti9867b676} and~\eqref{lroti9867b6762}
into~\eqref{oet4ighty387}, up to renaming constants, we thereby obtain that,
for every~$x\in B_1$,
\[\begin{split} Lv(x)
&\geq -\frac{C_1}{r_M^n}+C_2M = -C_1\,c_1 M +C_2M\geq
 -1,
\end{split}\]
as long as~$c_1$ is chosen sufficiently small.
Also, notice that $v(0)=(1+\delta)M-u(x_\circ)<\delta M$.
Using now Theorem~\ref{half-Harnack-sup}, we deduce that
\[\ave_{B_2} v(x) \,dx\leq C\big(v(0)+1\big)\leq \frac{M}{2},\]
where we used that $M>>1$. 
Equivalently, recalling~\eqref{alsoriyhur}
and employing the change of variable~$y:=x_\circ+(r_M/2)x$, this means that
\[\ave_{B_{r_M}(x_\circ)} \big((1+\delta)M - u(y)\big)\,dy \leq \frac{M}{2},\]
which gives that
$$ \ave_{B_{r_M}(x_\circ)} u(y)\,dy\geq (1+\delta)M
-\frac{M}2\geq \frac{M}2.$$
Now using  that~$B_{r_M}(x_\circ) \subset B_1$ and that $1+|x|^{n+2s} \le 2$  in $B_1$, we finally obtain that
\[\frac{M}{2} \leq \ave_{B_{r_M}(x_\circ)} u(x)\,dx = c_2r_M^{-n}\int_{B_{r_M}(x_\circ)} u(x)\,dx \leq  2 c_2r_M^{-n} \int_{\R^n} \frac{u(x)}{1+|x|^{n+2s}}\,dx = 2c_1c_2 M,\]
where we used \eqref{soigjer} and the definition of~$r_M$.
This gives a contradiction provided that $c_1$ is chosen small enough, and therefore the claim is proved.

\vspace{3mm}

We now use the Claim to finish the proof of Theorem~\ref{half-Harnack-sub}.
Namely, we will show that if $u(0)>M_\circ$, with $M_0$ sufficiently large, then this leads to $\sup_{B_{1/2}} u=\infty$, which is a contradiction.

Indeed, let $r_\circ := (c_1M_\circ)^{-1/n}<<1$, with $c_1$ given by the Claim.
Then, we deduce that there exists~$x_1\in B_{r_\circ}$ such that 
\[u(x_1)>(1+\delta)M_\circ =:M_1.\]
Applying iteratively the claim, we find $z_k$, $M_k$ and $r_k$ satisfying
\[z_{k+1}\in B_{r_{k}(z_k)},\qquad r_k= (c_1M_k)^{-1/n} \qquad {\mbox{and}}
\qquad
u(z_k)> M_k := (1+\delta)M_{k-1}.\]
Since $M_k=(1+\delta)^k M_\circ \to \infty$ as~$k\to\infty$, and 
\[|z_k| \leq \sum_{j=0}^{k-1} r_j \leq CM_\circ^{-1/n}\sum_{j=0}^{k-1} (1+\delta)^{-j/n} < \frac12,\]
we obtain that~$\sup_{B_{1/2}} u\ge u(z_k) > (1+\delta)^k M_\circ \to \infty$
as~$k\to\infty$, thus concluding the proof.
\end{proof}

\section{Boundary regularity}
\label{sec4}

The goal of this section is to prove Theorem \ref{thm-bdry}.
This will require several steps and intermediate results.

\subsection{A comparison principle}

We start by proving a comparison principle for bounded weak solutions.

\begin{lem} \label{comp}
Let $\Omega\subset \R^n$ be any bounded Lipschitz domain and $L$ be any operator of the form~\eqref{operator-L}-\eqref{ellipt-const}.

Let $u,v\in L^\infty(\R^n)$ be two functions satisfying $Lu= f$ in $\Omega$ and $Lv=g$ in $\Omega$ in the weak sense, with $f\leq g$ in $\Omega$ and $u\leq v$ in $\Omega^c$, and $f,g$ locally bounded in $\Omega$.
Then, $u\leq v$ in $\Omega$.
\end{lem}

\begin{proof}
By linearity, it suffices to treat the case $v\equiv0$, so that $u\leq 0$ in $\Omega^c$ and $Lu=f\leq0$ in $\Omega$.
Moreover, by interior regularity we know that $u$ is continuous inside $\Omega$.

Notice that the characteristic function $\chi_\Omega$ satisfies
\[L\chi_\Omega \geq c_1>0\qquad \textrm{in}\quad \Omega.\]

Let $d$ be a regularised distance to $\Omega^c$, more precisely, a function such that $d\equiv0$ in $\Omega^c$, $d\in C^2(\Omega)\cap {\rm Lip}(\overline\Omega)$, with\footnote{We recall that the symbol~$\asymp$ means that
the ratios~$d/{\rm dist}(x,\Omega^c)$ and~${\rm dist}(x,\Omega^c)/d$ are
positive and uniformly bounded.}
$d\asymp {\rm dist}(x,\Omega^c)$ and $|D^2d|\leq Cd^{-1}$,
see e.g.~\cite{Liebe}.

Then, it is not difficult to see that, since $d^{-\varepsilon}$ converges to $\chi_\Omega$ as $\varepsilon\to 0$ (with the convention that~$d^{-\varepsilon}=0$
in~$\Omega^c$), then for $\varepsilon>0$ small enough we have
\[L( d^{-\varepsilon}) \geq c_2>0\qquad \textrm{in}\quad \Omega.\]

Let us consider the functions $Ad^{-\varepsilon}$, with $A>0$.
For $A$ large enough, such function is above $u$, since $u$ is bounded.
We consider the smallest value of $A$ for which $Ad^{-\varepsilon}\geq u$ in $\Omega$, and assume by contradiction that such $A$ is strictly positive.

Then, since $Ad^{-\varepsilon}>>u$ near $\partial\Omega$, and since $u$ is continuous inside $\Omega$, we must have a point $x_\circ\in\Omega$ such that $Ad^{-\varepsilon}(x_\circ)=u(x_\circ)$.

Therefore, the function
\[w:= u-Ad^{-\varepsilon}\]
satisfies
\begin{itemize}
\item $w\leq 0$ in $\Omega^c$,
\item $Lw \leq -Ac_2<0$ in $\Omega$, in the weak sense,
\item $w$ is continuous inside $\Omega$,
\item $w(x_\circ)=0$, with $x_\circ\in\Omega$,
\item $w\leq -M$ in $\Omega_\rho:=\{x\in\Omega : {\rm dist}(x,\Omega^c)<\rho\}$,
for some $M>1$ and some small $\rho>0$.
\end{itemize}
We now consider a regularization $w_\delta$ of $w$, as in Remark \ref{approxX}.
Then, for $\delta>0$ small enough, we will have:
\begin{itemize}
\item $w_\delta\leq 0$ in $\R^n\setminus\overline{\Omega}$,
\item $Lw_\delta < 0$ in $\Omega\setminus\Omega_\rho$,
\item $w_\delta$ is smooth inside $\Omega$,
\item $|w_\delta(x_\circ)|\leq \theta<<1$, with $x_\circ\in\Omega$,
\item $w_\delta\leq -M$ in $\Omega_\rho$,
for some (possibly smaller) $M>1$ and $\rho>0$.
\end{itemize}
This implies that $w_\delta$ will achieve its maximum over $\Omega$ at a point $x_1\in \Omega\setminus\Omega_\rho$.
Moreover, $w_\delta(x_1)\geq-\theta$, with $\theta<<1$.
Hence, since we can evaluate $L$ pointwise on $w_\delta$, $\nabla w_\delta(x_1)=0$ and $w_\delta(x_1)\geq w_\delta(y)$ for all $y\in\Omega$, we find that
\[Lw_\delta(x_1) = \int_{\R^n} \big(w(x_1)-w(y)\big)K(x_1-y)dy \geq (M-\theta)|\Omega_\rho| - \int_{\Omega^c} \theta K(x_1-y)dy \geq M|\Omega_\rho|-C\theta.\]
Since $\theta$ can be taken arbitrarily small, while $M$ and $\rho$ are fixed, we find that $Lw_\delta(x_1)>0$, which is a contradiction.
\end{proof}

\subsection{Boundary Harnack inequality}

In order to classify global 1D solutions in a half-line, we will need the following boundary Harnack principle.
We will only use it for $n=1$, but we state it in full generality for completeness.

\begin{thm}\label{thm-bdryH}
Let $L$ be any operator of the form \eqref{operator-L}-\eqref{98988w019375=A}, satisfying in addition \eqref{uniform-ell}.

Let $\Omega\subset\R^n$ be any Lipschitz domain, with $0\in\partial\Omega$.
Then, there exists~$\delta>0$, depending only on $n$, $s$, $\Omega$, $\lambda$, $\Lambda$, such that the following statement holds.

Let $u_1,u_2$ be weak (or viscosity) solutions of 
\[\left\{ \begin{array}{rcll}
Lu_i & = & f_i & \textrm{in }B_1\cap \Omega\\
u_i&=&0&\textrm{in }B_1\setminus\Omega,
\end{array}\right.\]
satisfying
\[u_i\geq0\quad\mbox{in}\quad \R^n \qquad{\mbox{and}}\qquad
 \int_{\R^n}\frac{u_i(x)}{1+|x|^{n+2s}}\,dx=1.\]
Then, there exists $\alpha_\circ\in(0,1)$ such that
\[ \left\|\frac{u_1}{u_2}\right\|_{C^{0,\alpha_\circ}(\overline\Omega\cap B_{1/2})}\leq C.\]
The constants $\alpha_\circ$ and $C$ depend only on $n$, $s$, $\Omega$, $\lambda$, $\Lambda$.
\end{thm}

\begin{proof}
The proof is exactly the same as the one given in \cite[Theorems 1.1 and 1.2]{RS-bdryH} for the case of symmetric operators.
The only extra ingredients are the half Harnack inequalities for nonsymmetric operators, which we proved in Theorems \ref{half-Harnack-sub} and \ref{half-Harnack-sup} above.
\end{proof}

\subsection{1D Liouville theorem in the half line}

In order to classify 1D solutions, we also need the following explicit computation.

\begin{prop}\label{1D-solution}
Let $L$ be any operator of the form \eqref{operator-L}-\eqref{ellipt-const} in dimension $n=1$ with $s\neq\frac12$ and with kernel\footnote{No confusion
should arise here with the vector~$b$ in~\eqref{operator-L} in the case~$s=\frac12$.}
\[K(y)=\frac{a+b\sign y}{|y|^{1+2s}}, \]
with $a>b>0$, and let $u\in C(\R)$ be defined by
\[\qquad \qquad \qquad u(x):= (x_+)^\beta,\qquad \beta\in (0,2s).\]

Then, 
\[\qquad \qquad Lu = \kappa_{\beta,L}(x_+)^{\beta-2s}\quad \textrm{in}\quad \R_+,\]
where 
\[\kappa_{\beta,L} = \frac{2\pi\,\Gamma(-2s)}{\Gamma(-\beta)\Gamma(1-2s+\beta)}\,
\frac{\cos(\pi s) \cos(\pi(\beta-s))}{\sin(\pi\beta)\sin(\pi(2s-\beta))}\,
\big(a\tan(\pi(s-\beta)) - b\tan(\pi s)\big).\]

\vspace{2mm}

\noindent If we define
\[\gamma_L:=s-\frac{1}{\pi}\arctan\left(\frac{b}{a}\,\tan(\pi s)\right),\]
then
\begin{equation}\label{1D-a}
\gamma_L\in (0,2s)\cap(2s-1,1)\quad \mbox{and}\quad 
\left\{\begin{array}{ll}
\kappa_{\beta,L}>0  & \quad \textrm{for}\quad \beta\in (0,\gamma_L), \\
\kappa_{\beta,L}=0& \quad \textrm{for}\quad \beta=\gamma_L,\\
\kappa_{\beta,L}<0& \quad \textrm{for}\quad \beta\in(\gamma_L,2s).
\end{array}\right.\end{equation}

\vspace{2mm}

\noindent In particular, $u_\circ(x):=(x_+)^{\gamma_L}$ satisfies
\begin{equation}\label{1D-b}
\left\{ \begin{array}{rcll}
L u_\circ &=&0&\textrm{in }\ \R_+ \\
u_\circ&=&0&\textrm{in }\ \R_-.
\end{array}\right.
\end{equation}
\end{prop}

We provide a proof of this result in Appendix~\ref{secA};
see also \cite{Jus,Bertoin,Grubb}.

When $s=\frac12$ the corresponding result reads as follows; see \cite[Proposition~2.4]{FR}.

\begin{prop}\label{1D-solution-1/2}
Let $L$ be any operator of the form \eqref{operator-L}-\eqref{ellipt-const} in dimension $n=1$, with $s=\frac12$, that is,
\[L=a\,(-\Delta)^{1/2}+b\cdot \nabla, \]
with $a>0$ and $b\in \R$, and let $u\in C(\R)$ be defined by
\[\qquad \qquad \qquad u(x):= (x_+)^\beta,\qquad \beta\in (0,2s).\]

Then, 
\[\qquad \qquad Lu = \kappa_{\beta,L}(x_+)^{\beta-1}\quad \textrm{in}\quad \R_+,\]
where 
\[\kappa_{\beta,L} = \beta \,
\big(a\cos(\pi\beta) - b\sin(\pi \beta)\big).\]

\vspace{2mm}

\noindent If we define
\[\gamma_L:=\frac12+\frac{1}{\pi}\arctan\left(\frac{b}{a}\right),\]
then $\gamma_L\in (0,1)$ and \eqref{1D-a} holds.
In particular, $u_\circ(x):=(x_+)^{\gamma_L}$ satisfies \eqref{1D-b}.
\end{prop}

Using Propositions \ref{1D-solution} and~\ref{1D-solution-1/2},
we next show the following complete classification of 1D solutions.

\begin{prop}\label{1D-half-thm}
Let $L$ and $\gamma_L$ be either as in Proposition \ref{1D-solution} or as
in Proposition~\ref{1D-solution-1/2}.
Let $u\in C(\R)$ be any solution of
\begin{equation}\label{1D-eqn-halfspace}
\left\{ \begin{array}{rcll}
L u &=&0&\textrm{in }\R_+ \\
u&=&0&\textrm{in }\R_-,
\end{array}\right.
\end{equation}
satisfying
\[|u(x)|\leq C(1+|x|^{\beta})\quad \textrm{in}\ \R,\qquad \beta<2s.\]

Then, $u(x)\equiv k_0(x_+)^{\gamma_L}$ for some $k_0\in\R$.
\end{prop}

\begin{proof}
We divide the proof into two steps.

\vspace{3mm}

\noindent\emph{Step 1.} 
We claim that for any solution $w$ of $L w =f$ in $(0,2)$
and~$w\equiv0$ in $(-\infty,0]$, we have
\begin{equation}\label{1D-Step1}
\big\|w/d^{\gamma_L}\big\|_{C^{\alpha_\circ}([0,1])} \leq C\big(\|f\|_{L^\infty((0,2))}+\|w\|_{L^\infty(\R)}\big),
\end{equation}
where $d(x):=x_+$.

For this, notice first that, thanks to Propositions \ref{1D-solution} and \ref{1D-solution-1/2}, a standard barrier argument yields that~$|w|\leq Cd^{\gamma_L}$ in $(0,2)$.
Thus, the function $w+Cd^{\gamma_L}$ is positive in all of $\R$, provided that $C$ is large enough.
Also, given any $\delta>0$, by rescaling the function~$w$ (i.e., taking $w(r_\circ x)$ instead of $w(x)$, with a fixed $r_\circ>0$) we may assume that $\|f\|_{L^\infty(0,2)}\leq \delta$.

Then we notice that, thanks to the fact that we are in dimension $n=1$, our kernel $K$ satisfies
\[0<\frac{\lambda}{|y|^{n+2s}} \leq K(y) \leq \frac{\Lambda}{|y|^{n+2s}}.\]
Therefore, by Theorem \ref{thm-bdryH}, applied with $n=1$, $\Omega=(0,2)$, $u_1=w+Cd^{\gamma_L}$ and $u_2=d^{\gamma_L}$, we have that $\|u_1/u_2\|_{C^{\alpha_\circ}([0,1])} \leq C$, and this yields our claim in~\eqref{1D-Step1}.

\vspace{3mm}

\noindent\emph{Step 2.} 
Applying \eqref{1D-Step1} to our solution $u$, we find that $u/d^{\gamma_L}\in C^{\alpha_\circ}([0,1])$.
In particular, the limit 
\[c_\circ:=\lim_{x\downarrow 0^+} \frac{u(x)}{d^{\gamma_L}(x)}\]
 exists.
Moreover, if we define 
\[v:= u-c_\circ d^{\gamma_L},\] 
then we have that $|v|\leq Cd^{\gamma_L+\alpha_\circ}$ in $[0,1]$.
Combining this with interior estimates (see Corollary \ref{cor-interior-growth}),
we obtain that
\[|v'| \leq Cd^{\gamma_L+\alpha_\circ-1}\quad \textrm{in}\quad [0,1].\]
Here we are using the fact that, because we are in dimension $n=1$, then the kernel of the operator is smooth outside the origin.

Recall also that we have $|v|\leq Cd^\beta$ for $x\geq1$, with $\beta<2s$.
This, combined with interior estimates from Corollary \ref{cor-interior-growth}, yields
\[|v'| \leq Cd^{\beta-1}\quad \textrm{in}\quad [1,\infty].\]
Therefore, combining the previous two bounds, we deduce that
\[|v'| \leq C\big(d^{\gamma_L+\alpha_\circ-1}+d^{\beta-1}\big)\quad \textrm{in}\quad \R.\]

Now, notice that the functions $d^{\gamma_L}$ and $d^{\gamma_L-1}$ both solve \eqref{1D-eqn-halfspace} pointwise.
In particular, the function
\[U(x):= A\big(d^{\gamma_L}+d^{\gamma_L-1}\big)\]
satisfies \eqref{1D-eqn-halfspace}, for any $A>0$.
Moreover, $U>>v'$ at the origin (since $d^{\gamma_L-1}>>d^{\gamma_L+\alpha_\circ-1}$) and $U>>v'$ at infinity (since $d^{\gamma_L}>>d^{2s-1}>>d^{\beta-1}$ by \eqref{1D-a}).
Therefore, if $A>0$ is large enough, then we have that~$U\geq v'$ in $\R$.

We claim now that we must have $v\leq0$ in $\R$.
Indeed, if not, let~$A^*>0$ be the lowest number for which $U\geq v'$.
Then, since $U>>v'$ both at the origin and at infinity, and $U\equiv v'$ in $\R_-$, then there must exist $x_\circ>0$ for which $U(x_\circ)=v'(x_\circ)$.
But then, 
$$0=LU(x_\circ)\leq L(v')(x_\circ)=0,$$
which yields that~$LU(x_\circ)=L(v')(x_\circ)$, which gives a contradiction.
This means that $A^*=0$, and $v\leq 0$ in $\R$.

Repeating the same argument with $-v$ instead of $v$, we deduce that $v\equiv0$ in $\R$, and therefore $u\equiv c_\circ d^{\gamma_L}$ in $\R$, as claimed.
\end{proof}

As a consequence, we have the following.

\begin{cor}\label{1D-half-thm-cor}
Let $L$ be any operator of the form \eqref{operator-L}-\eqref{ellipt-const},
let~$\nu\in S^{n-1}$
and let 
\begin{equation}\label{gamma-L-nu}
\gamma_{L,\nu}:=s+\frac{1}{\pi}\arctan\left(\frac{\mathcal B(\nu)}{\mathcal A(\nu)}\right),
\end{equation}
where $\mathcal A+i\mathcal B$ is the Fourier symbol of $L$, given by \eqref{Fourier-A}-\eqref{Fourier-B}.
Let 
\[u(x):=U_0(x\cdot \nu),\]
 with $U_0\in C(\R)$, be any 1D solution of
\[
\left\{ \begin{array}{rcll}
L u &=&0&\textrm{in }\ \{x\cdot \nu>0\} \\
u&=&0&\textrm{in }\ \{x\cdot \nu\leq 0\},
\end{array}\right.
\]
with $U_0$ satisfying 
\[  |U_0(z)| \leq C(1+|z|^{\beta})\quad \textrm{for all }\ z\in\R,\quad \mbox{for some } \beta<2s.\]

Then, $u(x)\equiv k_\circ(x\cdot \nu)_+^{\gamma_{L,\nu}}$ for some $k_\circ\in\R$.
\end{cor}

\begin{proof}
The result follows from Proposition \ref{1D-half-thm} and Lemma \ref{polar-coordinates}.
\end{proof}

\subsection{Barriers and H\"older continuity up to the boundary}

First, as a consequence of Propositions~\ref{1D-solution}
and~\ref{1D-solution-1/2}, we have the following.

\begin{cor}\label{1D-solution-cor}
Let $L$ be any
operator of the form \eqref{operator-L}-\eqref{ellipt-const},
let~$\nu\in S^{n-1}$ and let $\gamma_{L,\nu}$ be given by~\eqref{gamma-L-nu}.

Let $u\in C(\R^n)$ be defined by
\[u(x):= (x\cdot \nu)_+^\beta.\]
where $\beta\in (0,2s)$.
Then, 
\[Lu = \kappa_{\beta,L,\nu}(x_+)^{\beta-2s}\quad \textrm{in}\quad \{x\cdot \nu>0\},\]
where $\kappa_{\beta,L,\nu}$ satisfies 
\[
\left\{\begin{array}{ll}
\kappa_{\beta,L,\nu}>0  & \quad \textrm{for}\quad \beta\in (0,\gamma_{L,\nu}), \\
\kappa_{\beta,L,\nu}=0& \quad \textrm{for}\quad \beta=\gamma_{L,\nu},\\
\kappa_{\beta,L,\nu}<0& \quad \textrm{for}\quad \beta\in(\gamma_{L,\nu},2s).
\end{array}\right.\]
Moreover, $|\kappa_{\beta,L,\nu}|\leq C|\beta-\gamma_{L,\nu}|$, for a constant $C$ depending only on $s$ and the ellipticity constants.

In particular, let us define
\[u_0(x):=(x_+)^{\gamma_{L,\nu}}.\]
Then, $u_0$ satisfies
\[
\left\{ \begin{array}{rcll}
L u &=&0&\textrm{in }\ \{x\cdot \nu>0\} \\
u&=&0&\textrm{in }\ \{x\cdot \nu\leq 0\}.
\end{array}\right.
\]
\end{cor}

\begin{proof}
It follows from Propositions \ref{1D-solution} and~\ref{1D-solution-1/2}
and the polar coordinate representation from Lemma~\ref{polar-coordinates}.
The bound on $\kappa_{\beta,L,\nu}$ follows from its explicit expression.
\end{proof}

We next show the following.

\begin{prop}[Supersolution]\label{prop-supersolution}
Let $\Omega\subset \R^n$ be any bounded $C^{1,\alpha}$ domain
and let $d(x)={\rm dist}(x,\Omega^c)$.
Let $\rho$ be a regularized distance function, satisfying
\[C^{-1}d\leq \rho\leq Cd,\qquad \|\rho\|_{C^{1,\alpha}(\overline\Omega)}\leq C \qquad {\mbox{and}}\qquad |D^2\rho|\leq Cd^{\alpha-1},\]
where $C$ depends only on $\Omega$.

Let $L$ be any operator of the form \eqref{operator-L}-\eqref{ellipt-const},
let~$\nu$ be the inward
unit normal to~$\partial\Omega$,
let $\gamma_{L,\nu}$ be given by \eqref{gamma-L-nu}
and let $\gamma^*_L:=\inf_\nu \gamma_{L,\nu}>0$.

Then, for any $\sigma<\gamma^*_L$, the function $\phi:=\rho^\sigma$ satisfies 
\[L\phi \geq c_0d^{\sigma-2s}>0\quad \textrm{in} \ \{0<d(x)\leq \delta\}.\]
The constants $c_0$ and $\delta$ depend only on $\Omega$, $\sigma$, and the 
ellipticity constants.
\end{prop}

\begin{proof}
The proof uses some ideas from \cite[Section~2]{RS-C1alpha}, where the case of symmetric operators was considered.
Pick any $x_0\in \{0<d(x)\leq \delta\}$, and define 
\[\ell(x):= \bigl(\rho(x_0)+\nabla \rho(x_0)\cdot (x-x_0)\bigr)_+.\]
Notice that, since $\ell$ is a 1D function, then by Corollary \ref{1D-solution-cor} we have
\[L(\ell^\sigma)(x_0)=c_{\sigma,L,\Omega} \,\ell^{\sigma-2s}(x_0).\]
Moreover, we see that~$\ell(x_0)=\rho(x_0)$, $\nabla\ell (x_0)=\nabla \rho(x_0)$
and 
\[\big|\ell(x_0+y)-\rho(x_0+y)\big|\leq C|y|^{1+\alpha}.\]
Using this, if we denote by~$r:=\frac12 d(x_0)$, then one can show that
\[|\rho^\sigma-\ell^\sigma|(x_0+y)\leq 
\begin{dcases}
\displaystyle Cr^{\sigma+\alpha-2}|y|^2 \qquad& {\mbox{ for }}y\in B_r,\\
C|y|^{(1+\alpha)\sigma} \qquad& {\mbox{ for }}y\in B_1\setminus B_r,\\
C|y|^\sigma \qquad& {\mbox{ for }}y\in \R^n\setminus B_1.
\end{dcases}\]
Hence, it follows that 
\[\big|L\big(\rho^\sigma-\ell^\sigma\big)(x_0)\big|\leq Cr^{\alpha+\alpha\sigma-2s}.\]
Since 
\[L(\ell^\sigma)(x_0)=cr^{\sigma-2s}>0,\]
we deduce that 
\[L(\rho^\sigma)(x_0)\geq cr^{\sigma-2s} - Cr^{\alpha+\alpha\sigma-2s}>\frac{c}{2}r^{\sigma-2s},\]
provided that $r=\frac12d(x_0)<\delta$ is small enough, and this completes
the proof of Proposition~\ref{prop-supersolution}.
\end{proof}

Using the interior estimates and the supersolution constructed
in Proposition~\ref{prop-supersolution}, we find the following.

\begin{prop}\label{prop-Holder}
Let $\Omega\subset \R^n$ be any bounded $C^{1,\alpha}$ domain.
Let $L$ be an operator of the form \eqref{operator-L}-\eqref{ellipt-const},
let~$\nu$ be the inward unit normal to~$\partial\Omega$,
let $\gamma_{L,\nu}$ be given by \eqref{gamma-L-nu},
let~$\gamma^*_L:=\inf_\nu \gamma_{L,\nu}>0$
and let~$\sigma<\gamma^*_L\leq s$.

Let $u\in L^\infty(\R^n)$ be any bounded solution of
\begin{equation}\label{eq-567}
\left\{ \begin{array}{rcll}
L u &=&f&\textrm{in }\Omega\cap B_1 \\
u&=&0&\textrm{in }B_1\setminus\Omega,
\end{array}\right.
\end{equation}
with $|f|\leq Cd^{\sigma-2s}$ in $\Omega$.

Then, for any $\epsilon>0$ we have
\[\|u\|_{C^\sigma(B_{1/2})}\leq C\left(\|d^{2s-\sigma}f\|_{L^\infty(\Omega\cap B_1)}+\sup_{R\geq1}\left\{R^{\epsilon-2s}\|u\|_{L^\infty(B_R)}\right\}\right).\]
The constant $C$ depends only on $\sigma$, $\Omega$, $\epsilon$, and the
ellipticity constants.
\end{prop}

\begin{proof}
The proof is standard (it follows from a simple scaling argument) once we have the interior estimates from Theorem \ref{thm-interior} (a), and the supersolution from Proposition \ref{prop-supersolution}; see for example~\cite[Proposition~3.1]{RS-C1alpha}.
\end{proof}

We will also need the following.

\begin{prop}\label{prop-approx-solution}
Let $\Omega\subset \R^n$ be any bounded $C^{1,\alpha}$ domain with $0\in \partial\Omega$ and with $\alpha<2s$ and let~$d(x)={\rm dist}(x,\Omega^c)$.

Let $\rho$ be a regularized distance function, satisfying
\[C_\Omega^{-1}d\leq \rho\leq C_\Omega d,\qquad \|\rho\|_{C^{1,\alpha}(\overline\Omega)}\leq C_\Omega,\qquad |D^2\rho|\leq C_\Omega d^{\alpha-1}\qquad {\mbox{and}}\qquad |D^3\rho|\leq C_\Omega d^{\alpha-2},\]
where $C_\Omega$ depends only on $\Omega$.

Let $L$ be any operator of the form \eqref{operator-L}-\eqref{ellipt-const},
let~$\nu$ be the inward unit normal to~$\partial\Omega$,
let $\gamma(L,\nu)$ be given by \eqref{gamma-L-nu}
and let $\gamma_\circ:=\gamma(L,\nu(0))>0$.
Let $\bar\Gamma$ be a function
that coincides with $\gamma(L,\nu(x))$ on $\partial\Omega$
and satisfies $|D^2\bar\Gamma(x)|\leq Cd^{\alpha-2}$ inside $\Omega$.
Assume in addition that $\bar \Gamma(x)\geq \gamma_\circ-\delta$ in $\Omega\cap B_1$.

Let $\phi$ be a function that coincides with~$\big(\rho(x)\big)^{\bar\Gamma(x)}$
in a neighborhood of $\partial\Omega$, and such that $\|\phi\|_{C^{\gamma_\circ-\delta}(\R^n)}\leq C$.

Then,  
\begin{equation}\label{primo00}
\big|L\phi(x)\big| \leq C\big(d(x)\big)^{\gamma_\circ+\alpha\gamma_\circ-2\delta-2s}\quad \textrm{for every } x\in \Omega,\end{equation}
as long as such exponent is negative.

If in addition we assume that the kernel $K$ of the operator $L$ satisfies $K(y)\leq \Lambda|y|^{-n-2s}$, then the bound in~\eqref{primo00}
can be improved to
\begin{equation}\label{primo11}\big|L\phi(x)\big| \leq C\big(d(x)\big)^{\gamma_\circ+\alpha-2\delta-2s}\quad \textrm{for every } x\in \Omega,
\end{equation}
as long as such exponent is negative.

The constants $C$ depend only on $\delta$, $C_\Omega$, and the ellipticity constants.
\end{prop}

\begin{proof}
Pick any $x_0\in \Omega$ with $d(x_0)=2r$ small enough, and define 
\[\ell(x):= \bigl(\rho(x_0)+\nabla \rho(x_0)\cdot (x-x_0)\bigr)_+.\]

Notice that, since $\nu$ is $C^\alpha$ on $\partial\Omega$, with $\alpha<2s$, and the Fourier symbol of $L$ is $C^{2s-\varepsilon}$ for all $\varepsilon>0$, then the function $\gamma(L,\nu)$ is $C^\alpha$ on $\partial\Omega$.
Thus, we can extend it to a function $\bar\Gamma(x)$ inside $\Omega$ satisfying $|D^2\bar \Gamma(x)|\leq Cd^{\alpha-2}$ in $\Omega$.

Now, since $\ell$ is a 1D function,
by Corollary \ref{1D-solution-cor} we have
\[L\left(\ell^{\gamma\left(L,\frac{\nabla \rho(x_0)}{|\nabla \rho(x_0)|}\right)}\right)(x_0)=0.\]
Moreover, since $\frac{\nabla \rho(x)}{|\nabla \rho(x)|}$ is a $C^\alpha$ function that coincides with $\nu$ on $\partial\Omega$, then it is not difficult to see that 
\[\left| \bar\Gamma(x_0) - \gamma\left(L,\frac{\nabla \rho(x_0)}{|\nabla \rho(x_0)|}\right)\right| \leq Cr^\alpha,\]
and therefore by Corollary \ref{1D-solution-cor} we deduce that
\[\big| L(\ell^{\bar\Gamma(x_0)})(x_0) \big| \leq Cr^\alpha \ell^{\bar\Gamma(x_0)-2s} \leq Cr^{\alpha+\gamma_\circ-\delta-2s}.\]

On the other hand, exactly as in Proposition \ref{prop-supersolution}, we have that
\[\ell(x_0)=\rho(x_0),\qquad \nabla\ell (x_0)=\nabla \rho(x_0)\qquad
{\mbox{and}}\qquad \big|\ell(x)-\rho(x)\big|\leq C|x-x_0|^{1+\alpha}.\]
Also, we have
 \[\|\bar\Gamma\|_{C^{0,\alpha}(\overline\Omega)}\leq C,\qquad |\nabla \bar\Gamma|\leq C d^{\alpha-1}\qquad {\mbox{and}}
 \qquad |D^2\bar\Gamma|\leq  d^{\alpha-2}.\]

Now, a direct (but somewhat lengthy) computation gives that
\[\begin{split}
\partial_{ij}\big(\rho^{\bar\Gamma(x)}-\ell^{\bar\Gamma(x_0)}\big) = &\  \bar\Gamma(x) \rho^{\bar \Gamma(x)-1} \partial_{ij}\rho \\
&+ \bar\Gamma(x)\big(\bar\Gamma(x)-1\big)\rho^{\bar\Gamma(x)-2}\partial_i\rho\,\partial_j \rho - \bar\Gamma(x_0)\big(\bar\Gamma(x_0)-1\big)\ell^{\bar\Gamma(x_0)-2}\partial_i \rho(x_0)\,\partial_j\rho(x_0)\\
&+\partial_{ij}\bar\Gamma(x)\,(\log\rho)\,\rho^{\bar\Gamma(x)} +\partial_i\bar\Gamma(x)\partial_j\bar\Gamma(x)(\log\rho)^2\rho^{\bar\Gamma(x)}
\\
&+ \partial_j \bar\Gamma(x)\partial_i\rho\,\rho^{\bar\Gamma(x)-1} + \partial_j\bar\Gamma(x)(\log\rho)\,\bar\Gamma(x)\partial_i\rho\,\rho^{\bar\Gamma(x)-1}.
\end{split}\]
This, in turn, leads to
\[\big|D^2\big(\rho^{\bar\Gamma(x)}-\ell^{\bar\Gamma(x_0)}\big)\big|\leq 
Cr^{\alpha+\bar\Gamma(x)-2}|\log r| \qquad {\mbox{ in }}B_r(x_0),\]
and thus
\[\big|\rho^{\bar\Gamma(x)}-\ell^{\bar\Gamma(x_0)}\big|\leq 
Cr^{\alpha+\gamma_\circ-2\delta-2}|x-x_0|^2 \qquad {\mbox{ in }}B_r(x_0).\]
Furthermore, since $|a^\gamma-b^\gamma|\leq C|a-b|^\gamma$ and $|a^p-a^q|\leq Ca^p|p-q|\cdot|\log a|$ for all $p\leq q$ and $a\leq 1$, then we also have 
\[\big|\rho^{\bar\Gamma(x)}-\ell^{\bar\Gamma(x_0)}\big|\leq C|x-x_0|^{(1+\alpha)\bar\Gamma(x_0)}+Cd^{\gamma_\circ-\delta}(x)|x-x_0|^{\alpha}\leq C|x-x_0|^{\gamma_\circ+\alpha\gamma_\circ-2\delta} \qquad {\mbox{in }} B_1(x_0)\setminus B_r(x_0),\]
and
\[\big|\rho^{\bar\Gamma(x)}-\ell^{\bar\Gamma(x_0)}\big|(x)\leq C|x-x_0|^{\gamma_\circ-\delta}\qquad {\mbox{ in }} \R^n\setminus B_1(x_0).\]

Using these bounds, we find 
\[\begin{split} \big|L\phi(x_0)\big| = \,& \big|L\big(\rho^{\bar\Gamma(\cdot)}-\ell^{\bar\Gamma(x_0)})\big)(x_0)\big| + \big|L\big(\ell^{\bar\Gamma(x_0)}\big)(x_0)\big|\\
 \leq \,& C\int_{B_r} r^{\alpha+\gamma_\circ-2\delta-2}|y|^2K(y)dy\, + C\int_{B_1\setminus B_r} |y|^{\gamma_\circ+\alpha\gamma_\circ-2\delta}K(y)dy \,+  C\int_{B_1^c}|y|^{\gamma_\circ-\delta}K(y)dy\\
 & + Cr^{\alpha+\gamma_\circ-\delta-2s} \\
\leq \,& Cr^{\gamma_\circ+\alpha-2\delta-2s}+Cr^{\gamma_\circ+\alpha\gamma_\circ-2\delta-2s}+C+ Cr^{\alpha+\gamma_\circ-\delta-2s}  \\
\leq \,& Cr^{\gamma_\circ+\alpha\gamma_\circ-2\delta-2s},
\end{split}\]
as claimed in~\eqref{primo00}.

Finally, when $K(y)\leq C|y|^{-n-2s}$, then the bound in $B_1\setminus B_r$ can be improved as follows
\[\big|\rho^{\bar\Gamma(x)}-\ell^{\bar\Gamma(x_0)}\big|\leq C|x-x_0|^{1+\alpha}\big(d^{\gamma_\circ-\delta-1}(x)+\ell^{\gamma_\circ-\delta-1}(x)\big)+Cd^{\gamma_\circ-\delta}(x)|x-x_0|^{\alpha} \qquad {\mbox{ in }} B_1(x_0)\setminus B_r(x_0),\]
and then Lemma 2.5 in \cite{RS-C1alpha} ---applied twice--- yields that
\[\int_{B_1\setminus B_r} \big|\rho^{\bar\Gamma(\cdot)}-\ell^{\bar\Gamma(x_0)}\big|(x_0+y)\,K(y)dy\leq Cr^{\gamma_\circ+\alpha-\delta-2s},\]
which gives the improved bound in~\eqref{primo11}.
\end{proof}

\subsection{Liouville theorem in the half space}

We next show the following.

\begin{thm}\label{Liouv-half}
Let $L$ be any operator of the form \eqref{operator-L}-\eqref{ellipt-const},
let~$\nu\in S^{n-1}$
and let $\gamma=\gamma(L,\nu)$ be given by \eqref{gamma-L-nu}.

Let $u$ be any weak solution of
\begin{equation}\label{eq-I-flat}
\left\{ \begin{array}{rcll}
L u &=&0&\textrm{in }\{x\cdot \nu>0\} \\
u&=&0&\textrm{in }\{x\cdot \nu\leq 0\}.
\end{array}\right.
\end{equation}
Assume that, for some $\beta<2s$, $u$ satisfies the growth control
\[\|u\|_{L^\infty(B_R)}\leq C_0R^{\beta}\quad \textrm{for all}\ R\geq1.\]
Then,
\[u(x)=k_0(x\cdot \nu)_+^{\gamma}\]
for some constant $k_0\in \R$.
\end{thm}

\begin{proof}
After a rotation, we may assume that~$\nu=e_n$.
We will first prove that the function $u$ must be 1D, following some ideas from \cite[Theorem 4.1]{RS-elliptic}.

Let us consider, for $\rho\geq1$, the function $v_\rho(x):=\rho^{-\beta}u(\rho x)$, which satisfies the same growth condition as $u$,
\begin{equation}\label{growth-v-rho}
\|v_\rho\|_{L^\infty(B_R)}\leq C_0R^{\beta}\quad \textrm{for all}\ R\geq1.
\end{equation}
Moreover,  $Lv_\rho=0$ in $\R^n_+=\{x_n>0\}$ and $v_\rho=0$ in $\R^n_-=\{x_n<0\}$.

In particular, applying the H\"older estimates from Proposition \ref{prop-Holder} to the function $v_\rho\chi_{B_2}$, we find that
\[\|v_\rho\|_{C^\sigma(B_{1/2})}\leq CC_0,\]
for some $\sigma>0$.
Therefore, we deduce that
\[[u]_{C^\sigma(B_{\rho/2})}=\rho^{\beta-\sigma}[v_\rho]_{C^\sigma(B_{1/2})}\leq CC_0\rho^{\beta-s}.\]
In other words, we have that
\[[u]_{C^\sigma(B_R)}\leq CR^{\beta-\sigma}\qquad \textrm{for all}\ R\geq1.\]

Now, given $\tau\in S^{n-1}$ such that $\tau_n=0$
and given $h>0$, consider
\[w(x):=\frac{u(x+h\tau)-u(x)}{h^\sigma}.\]
By the previous considerations, we have that
\[\|w\|_{L^\infty(B_R)}\leq CR^{\beta-\sigma}\qquad \textrm{for all}\ R\geq1.\]
Moreover, we clearly have that~$Lw=0$ in $\R^n_+$ and $w=0$ in $\R^n_-$.
Therefore, we can repeat the previous argument (applied to $w$ instead of $u$), to find that
\[[w]_{C^\sigma(B_R)}\leq CR^{\beta-2\sigma}\qquad \textrm{for all}\ R\geq1.\]

Iterating this argument, we find in a finite number of steps that 
\[[w_m]_{C^\sigma(B_R)}\leq CR^{\beta-m\sigma}\qquad \textrm{for all}\ R\geq1,\]
where $\beta-m\sigma<0$ and where $w_m$ is an incremental quotient of $u$ in the first $(n-1)$ variables.
Letting $R\to\infty$, we deduce that $w_m$ is constant, and hence $w_m\equiv0$ in $\R^n$.
It is then not difficult to see (exactly as in \cite[Theorem 3.10]{AR}) that $u(x',x_n)$ must be a polynomial in $x'$ for every fixed $x_n$.

However, the growth condition on $u$ implies that such polynomial must be constant if $s\leq\frac12$, or linear if $s>\frac12$, i.e., either
\[u(x)= k_0 U_0(x_n)\]
for some constant $k_0$, or $s>\frac12$ and 
\[u(x)= k_0 U_0(x_n) + (q\cdot x')V_0(x_n)\]
for some $k_0\in \R$ and $q\in \R^{n-1}$, $q\neq0$.

In the first case, thanks to Corollary \ref{1D-half-thm-cor} we deduce that $u(x)=k_0(x_n)_+^{\gamma_{L,\nu}}$.

In the second case, applying Corollary \ref{1D-half-thm-cor} to the function $u(x'+q,x_n)-u(x',x_n)$, we deduce that $V_0(x_n)=c_0(x_n)_+^{\gamma_{L,\nu}}$.
However, the growth condition on $u$ then implies that $V_0\equiv0$, and therefore  $u(x)=k_0(x_n)_+^{\gamma_{L,\nu}}$, as claimed.
\end{proof}

\subsection{Boundary regularity}

In this section we prove Theorem \ref{thm-bdry}.
To this end, we start by obtaining a suitable H\"older bound
at the boundary:

\begin{prop}\label{proponbdryreg}
Let $L$ be any operator of the form \eqref{operator-L}-\eqref{ellipt-const}, and let $\gamma(L,\nu)$ be given by \eqref{gamma-L-nu}.
Assume $\partial\Omega$ is a $C^{1,\alpha}$ graph with norm bounded by 1, with $0\in \partial\Omega$, and such that its unit normal at the origin is $e_n$.

Let $\rho$ be a regularized distance function, as in Proposition \ref{prop-approx-solution}.

Let $f\in L^\infty(\Omega\cap B_1)$, and assume that $u\in L^\infty(\R^n)$ is a solution of
\[\left\{\begin{array}{rcl}
L u &= & f \quad \mbox{in } \Omega\cap B_1\\
u  &=&0\quad \mbox{in }B_1\setminus\Omega.
\end{array}\right.\]
Denote 
\[C_0:=\|u\|_{L^\infty(\R^n)}+\|f\|_{L^\infty(\Omega\cap B_1)}.\]

Then, for all $z\in \partial\Omega\cap \overline{B_{1/2}}$ there
exists a constant $Q(z)$ with $|Q(z)|\leq CC_0$ for which
\[\left|u(x)-Q(z)\rho^{\gamma(z)}\right|\leq CC_0|x-z|^{\gamma(z)+\varepsilon_\circ}\qquad \textrm{for all}\ x\in B_1,\]
where $\nu(z)$ is the unit unit normal to $\partial\Omega$ at $z$, and 
\[\gamma(z):=\gamma(L,\nu(z)).\]
The constants $C$ and $\varepsilon_\circ$ depend only on $n$, $s$ and the ellipticity constants in~\eqref{ellipt-const}.

Furthermore, if in addition we assume that the kernel of the operator $L$ satisfies $K(y)\leq C|y|^{-n-2s}$, then the improved estimate
\[\left|u(x)-Q(z)\rho^{\gamma(z)}\right|\leq CC_0|x-z|^{\min\{\gamma(z)+\alpha,\,2s\}-2\delta}\qquad \textrm{for all}\ x\in B_1\]
holds.
\end{prop}

\begin{proof}
We prove the result for $z=0$.

Assume that there are sequences $\Omega_k$, $f_k$, $u_k$
and $L_k$ that satisfy the assumptions of Proposition~\ref{proponbdryreg}, but suppose for a contradiction that the conclusion does not hold.

That is, for all $C>0$, there exists $k$ for which
no constant $Q\in \R$ satisfies
\begin{equation*}
\big| u_k(x) - Q \rho_k^{\gamma_k}(x) \big| \le C|x|^{\gamma_k+\varepsilon_\circ}\quad \mbox{for all }x\in B_1,
\end{equation*}
where $\gamma_k:=\gamma_k(0)$.

Now, let $\phi_k$ be given by Proposition \ref{prop-approx-solution}.
Then, since $|\rho_k^{\gamma_k}-\phi_k|\leq C|x|^{\gamma_k+\alpha-\delta}\leq C|x|^{\gamma_k+\varepsilon_\circ}$,
we have that
\begin{equation*}
\big| u_k(x) - Q \phi_k(x) \big| \le C|x|^{\gamma_k+\varepsilon_\circ}\quad \mbox{for all }x\in B_1,
\end{equation*}
Notice that, up to a subsequence, we may assume that~$|\gamma_k-\gamma_\circ| < \varepsilon_\circ/4$ and $\gamma_k\to\gamma_\circ$ for some fixed $\gamma_\circ$.

Using \cite[Lemma 3.3]{RS-C1alpha}, we deduce that
\begin{equation*}
\sup_k \sup_{r>0} \ r^{-\gamma_k-\varepsilon_\circ}\left\| u_k- \Psi_{k,r} \right\|_{L^\infty(B_{r})} = \infty,
\end{equation*}
where
\begin{equation*}
\Psi_{k,r}(x) := Q_{k}(r)\,\phi_k(x)
\end{equation*}
and
\[
Q_{k}(r):= \ {\rm arg\,min}_{Q\in \R}  \int_{B_r}  \big| u_k(x) -Q\phi_k\big|^2 \,dx= \frac{\displaystyle \int_{B_r}u_k\phi_k dx}{\displaystyle \int_{B_r}\phi_k^2 dx}.\]

Next define the monotone in $r$ quantity
\[\theta(r):=\sup_k \sup_{r'>r}  \ (r')^{-\gamma_k-\varepsilon_\circ} \bigl\|u_k-\Psi_{k,r'}\bigr\|_{L^\infty\left(B_{r'}\right)}.\]
We have that~$\theta(r)<\infty$ for $r>0$ and $\theta(r)\nearrow \infty$ as $r\searrow0$.
Clearly, there exist sequences $r_m\searrow 0$ and $k_m$ for which
\begin{equation}\label{nondeg2-3}
(r_m)^{-\gamma_k-\varepsilon_\circ}\left\|u_{k_m}-\Psi_{k_m,r_m}\right\|_{L^\infty(B_{r_m})}\ge \theta(r_{m})/2.
\end{equation}
From now on in this proof we denote $\Psi_m = \Psi_{k_m,r_m}$ and $\gamma_m=\gamma_{k_m}$.

In this situation we consider
\[ v_m(x) := \frac{u_{k_m}(r_m x)-\Psi_{m}(r_m x)}{(r_m)^{\gamma_k+\varepsilon_\circ}\theta(r_m)}.\]
Note that, for all $m\ge 1$,
\begin{equation}\label{2-3}
\int_{B_1} v_m(x)\phi_{k_m}(r_m x) \,dx =0.
\end{equation}
This is the optimality condition for least squares.

Note also that \eqref{nondeg2-3} is equivalent to
\begin{equation}\label{nondeg35-3}
\|v_m\|_{L^\infty(B_1)} \ge \frac12,
\end{equation}
which holds for all $m\geq1$.

\noindent {\bf Claim 1. }{\em The functions $v_m$ satisfy the growth control
\begin{equation}\label{growthc0-3}
\|v_{m}\|_{L^\infty(B_R)} \leq CR^{\gamma_m+\varepsilon_\circ}\quad \textrm{for all}\ \,R\ge 1,
\end{equation}
with $C$ independent of $m$.
}
\vspace{5pt}

For all $k$ and $r$ we have
\begin{equation}\label{si38b47556}
|Q_{k}(2r)-Q_{k}(r)|\le r^{\varepsilon_\circ}\theta(r).\end{equation}
Indeed,
\[\begin{split}
|Q_{k}(2r)-Q_{k}(r)|r^{\gamma_k}&=\|\phi_{k,2r}-\phi_{k,r}\|_{L^\infty(B_r)} 
\\&\leq \|\phi_{k,2r}-u\|_{L^\infty(B_{2r})}+\|u-\phi_{k,r}\|_{L^\infty(B_r)} \\
&\leq (2r)^{\gamma_k+\varepsilon_\circ} \theta(r)+r^{\gamma_k+\varepsilon_\circ}\theta(r)=Cr^{\gamma_k+\varepsilon_\circ} \theta(r),
\end{split}\]
which is~\eqref{si38b47556}.

Accordingly, for $R=2^N$ we have
\[\begin{split}
\frac{r^{-\varepsilon_\circ}|Q_{k}(rR)- Q_{k}(r)|}{\theta(r)}&\le \sum_{j=0}^{N-1} 2^{j\varepsilon_\circ}\frac{ (2^jr)^{-\varepsilon_\circ} |Q_{k}(2^{j+1}r)- Q_{k}(2^jr)|}{\theta(r)}\\
&\le C\sum_{j=0}^{N-1} 2^{j\varepsilon_\circ}\,\frac{\theta(2^jr)}{\theta(r)} \le C 2^{N\varepsilon_\circ} = CR^{\varepsilon_\circ}.
\end{split}\]
Then it follows that 
\[|Q_{k}(rR)- Q_{k}(r)|\leq (rR)^{\varepsilon_\circ}\theta(r)\]
for every $r>0$, $R\geq1$.

Therefore,
\[\begin{split}
\|v_{m}\|_{L^\infty(B_R)} &= \frac{1}{\theta({r_m})(r_m)^{\gamma_m+\varepsilon_\circ} } \bigl\|u_{k_m}-Q_{k_m}(r_m) \phi_m\bigr\|_{L^\infty\left(B_{r_m R}\right)}
\\
&\le \frac{R^{\gamma_m+\varepsilon_\circ}}{\theta({r_m}) (r_m R)^{\gamma_m+\varepsilon_\circ} }\bigl\|u_{k_m}- Q_{k_m}(r_mR) \phi_m\bigr\|_{L^\infty\left(B_{r_m R}\right)} 
\\
& \qquad\qquad+  \frac{1}{\theta({r_m}) (r_m)^{\gamma_m+\varepsilon_\circ} } |Q_{k_m}(r_mR)-Q_{k_m}(r_m)|\,(r_mR)^{\gamma_m}
\\
&\le \frac{R^{\gamma_m+\varepsilon_\circ}\theta({r_m}R)}{\theta({r_m})} + CR^{\gamma_m+\varepsilon_\circ}\leq CR^{\gamma_m+\varepsilon_\circ},
\end{split}\]
and hence \eqref{growthc0-3} follows. 

\vspace{5pt}
\noindent {\bf Claim 2. }{\em The functions $v_m$ satisfy the uniform H\"older estimate
\[
\|v_{m}\|_{C^\delta(B_R)} \leq C(R)\quad \textrm{for all}\ \,R\ge 1,
\]
with $C(R)$ independent of $m$.
}
\vspace{5pt}

To prove this, notice that $v_m$ solves
\[
L_{k_m} v_m(x) =  \frac{(r_m)^{2s}}{(r_m)^{\gamma_m+\varepsilon_\circ}\theta(r_m)}\left\{f_{k_m}(r_m x)-\big(L_{k_m}\Psi_m\big)(r_m x)\right\} \quad \mbox{ in }(r_m^{-1}\Omega_m)\cap B_{1/r_m}.
\]
Thus, thanks to Proposition \ref{prop-approx-solution}, we have that
\[\big|L_{k_m} v_m\big|\leq \frac{C}{\theta(r_m)}(r_m)^{2s-\gamma_m+\varepsilon_\circ}d_{k_m}^{\gamma_m+\alpha\gamma_m-2\delta-2s}(r_m x).\]
Denoting by~$\tilde d_m(x):={\rm dist}(x,r_m^{-1}\Omega_m)=r_m^{-1} d_{k_m}(r_m x)$, this is equivalent to
\begin{equation}\label{eqvm0}
\big|L_{k_m} v_m\big|\leq \frac{C}{\theta(r_m)}(r_m)^{\alpha\gamma_m-2\delta-\varepsilon_\circ}{\tilde d_m}^{\gamma_m+\alpha\gamma_m-2\delta-2s}.
\end{equation}
Choosing $\varepsilon_\circ>0$ small (depending only on $\alpha$ and the
ellipticity constants), we
have that~$\alpha\gamma_m-2\delta-\varepsilon_\circ\geq 0$.
Since $\theta(r_m)\to\infty$, then by Proposition \ref{prop-Holder} we have
\[\|v_{m}\|_{C^\sigma(B_M)} \leq C(M)\quad \textrm{for any}\ \,M\ge 1,\]
provided that $m$ is large enough.
Thus, Claim 2 is proved.

\vspace{3mm}

Now, by the Arzel\`a-Ascoli theorem, the functions $v_m$ converge  (up to a subsequence) locally uniformly to a function $v\in C(\R^n)$.
Moreover, notice that the domains $r_m^{-1}\Omega_m$ converge locally uniformly to $\{x_n>0\}$ as $m\to \infty$.

Therefore, since the right hand side in \eqref{eqvm0} converges to 0 locally uniformly in $\{x_n>0\}$, by Lemma \ref{lem-subseq} there exists
a limiting operator $L$ for which $v$ satisfies
\[\left\{\begin{array}{rcl}
L v &= & 0 \quad \mbox{in } \{x_n>0\}\\
v  &=&0\quad \mbox{in }\{x_n\leq0\}.
\end{array}\right.\]

Passing to the limit the growth control in~\eqref{growthc0-3} on $v_m$ we find
that 
\[\|v\|_{L^\infty(B_R)}\le R^{\gamma_\circ+\varepsilon_\circ}\] 
for all $R\ge1$.
Hence, by Theorem \ref{Liouv-half}, it must be
\[v(x)=k_0(x_n)_+^{\gamma_\circ}.\]

Finally, passing \eqref{2-3} to the limit, we find that
\[\int_{B_1} v(x) (x_n)_+^{\gamma_\circ} \,dx =0,\]
so that $k_0=0$.
However, passing \eqref{nondeg35-3} to the limit, we reach a contradiction.
Thus, Proposition~\ref{proponbdryreg} is proved.
\end{proof}

We can finally give the:

\begin{proof}[Proof of Theorem \ref{thm-bdry}]
The result follows from Proposition \ref{proponbdryreg}, and the fact that $\big|\rho^{\gamma(0)}(x)-(x_n)_+^{\gamma(0)}\big| \leq C|x|^{\gamma(0)+\varepsilon_\circ}$ and $\big|\rho^{\gamma(0)}(x)-d^{\gamma(0)}(x)\big| \leq C|x|^{\gamma(0)+1}$.
\end{proof}

\section{Integration by parts identities: the flat case}
\label{sec5}

The aim of this section is to prove Proposition~\ref{flat-case}. 
For simplicity we will do the proof under the additional assumption that~$K\in C^\infty(S^{n-1})$
(this will be used in the subsequent Lemma~\ref{NOTr34E}).
However, once the result in
Proposition~\ref{flat-case} is established for smooth kernels, one can easily obtain the general case by approximation (see e.g.~\cite{RSV}).

The results of this section are independent from the rest of the paper, and our proof is based on the ideas of \cite{RS-Poh,RSV}.
In a personal communication, G. Grubb recently showed us how, using the complex methods from \cite{Grubb3}, one can also give a different
proof of such results.

\subsection{The square root operator}

We will need the following convenient properties of the square root of a nonsymmetric operator $L$:

\begin{lem}\label{NOTr34E}
Let $L$ be any operator of the form \eqref{operator-L}-\eqref{ellipt-const}, with $K\in C^\infty(S^{n-1})$, and let $\mathcal A(\xi)$, $\mathcal B(\xi)$ be given by \eqref{Fourier-A}-\eqref{Fourier-B}.
Then,
\begin{equation}\label{CA}\begin{split}&
{\mbox{the symbol of~$\sqrt{L}$
is~$\mathcal{A}_\sharp(\xi)+i\mathcal{B}_\sharp(\xi)$, \ with}}\\
&\qquad\mathcal{A}_\sharp
:=\sqrt{\frac{\mathcal{A}+\sqrt{\mathcal{A}^2+\mathcal{B}^2}}{2}}\qquad
{\mbox{ and }}\qquad
\mathcal{B}_\sharp:=
\frac{\mathcal{B}}{ \sqrt{2}\;\sqrt{\mathcal{A}+\sqrt{\mathcal{A}^2+\mathcal{B}^2}}}\,.\end{split}\end{equation}
Moreover,
\begin{equation}\label{CA2}
{\mbox{$\mathcal{A}_\sharp$ and $\mathcal{B}_\sharp$
are positively homogeneous of degree~$s$,}}\end{equation}
\begin{equation}\label{FO}
{\mbox{${\mathcal{A}}_\sharp$ is even and~${\mathcal{B}}_\sharp$ is odd.}}\end{equation}
Furthermore, we can write
\begin{equation}\label{KSHAR}
\begin{split}&
\sqrt{L} u(x)=\int_{\R^n}\big(u(x)-u(x+y)\big)\,K_\sharp(y)\,dy,\\
&{\mbox{with $K_\sharp:\R^n\setminus\{0\}\to\R$ which is positively homogeneous of degree~$-n-s$.}} 
\end{split}
\end{equation}
In addition,
if we set
\begin{equation}\label{Espressioni0}
K_e^\sharp(x):=\frac{K_\sharp(x)+K_\sharp(-x)}{2}\qquad{\mbox{and}}\qquad
K_o^\sharp(x):=\frac{K_\sharp(x)-K_\sharp(-x)}{2},
\end{equation}
we have that
\begin{equation}\label{Espressioni}
\begin{split}&\mathcal{A}_\sharp(\xi)=
|\Gamma(-s)|\,
\cos\left(\frac{\pi s}{2}\right)\,
\int_{\partial B_1}|\theta\cdot\xi|^{s}
K_e^\sharp(\theta)d\theta
\\{\mbox{and }}\qquad&\mathcal{B}_\sharp(\xi)=-|\Gamma(-s)|\,
\sin\left(\frac{\pi s}{2}\right)\,
\int_{\partial B_1} |\theta\cdot\xi|^{s}\,\sign(\theta\cdot\xi)\,K_o^\sharp(\theta)
d\theta.
\end{split}
\end{equation}
\end{lem}

\begin{proof} {F}rom~\eqref{FS}, we know that
the symbol of~$\sqrt{L}$
is~$\sqrt{\mathcal{A}(\xi)+i\mathcal{B}(\xi)}$, which we write as~$\mathcal{A}_\sharp(\xi)+i\mathcal{B}_\sharp(\xi)$,
with the convention that~$\mathcal{A}_\sharp(\xi)>0$, and this proves~\eqref{CA}.
In addition, \eqref{FO} is a consequence of~\eqref{EO}.

Also, we have that~\eqref{KSHAR}
follows from
Lemma~4.1
in~\cite{RSV} (and~$K_\sharp$
is obtained by Fourier transforming the
symbol~${\mathcal{A}_\sharp+i\mathcal{B}_\sharp}$).

Finally, the proof of~\eqref{Espressioni} is the same as in Lemma \ref{AeBFOU}.
\end{proof}

\subsection{Expansions}

In this part of the paper, we compute the action of the square root operator
on a one-dimensional function multiplied by a cutoff.
In particular, we need a careful expansion of this object, with
an explicit estimate on the reminder, as stated in the following result:

\begin{lem}\label{LE.31}
Let~$s\in(0,1)$
and~$ \gamma\in(0,2s)$, with~$\gamma\ne s$.
Let~$\eta\in C^\infty_c(\R^n)$
and
\[u(x):=(x_n)_+^\gamma\,\eta(x).\]
Then\footnote{We remark that when~$\gamma=s$ one can adapt the methods
of this proof and of~\cite{RSV} to show that, in this case, one can replace~\eqref{EXP}
with
$$ \sqrt{L}u(x)=
\big(A_1 (\log|x_n|)_-+A_2 \,\chi_{(0,+\infty)}(x_n)\big)\eta(x',0)+h(x).$$
Nevertheless, considering explicitly the case~$\gamma=s$ here
would lead to a number of technical complications. Instead,
to keep the relevant arguments as simple as possible, we adopted
a different strategy, focusing in Lemma~\ref{LE.31}
on the case~$\gamma\ne s$. Later on,
we will extend any useful information also to the case~$\gamma=s$ simply by taking limits.}
\begin{equation}\label{EXP}
\sqrt{L}u(x)=
\big(A_1 (x_n)_+^{\gamma-s}+A_2 (x_n)_-^{\gamma-s}\big)\eta(x',0)+h(x),
\end{equation}
where~$A_1$, $A_2\in\R$, and,
for any~$x\in\R^n$ with~$|x_n|<1$,
\begin{equation}\label{LIA}
|\partial_n h(x',x_n)|\le C\,|x_n|^{\beta-1},
\end{equation}
for some~$C>0$ and~$\beta\in\left(\max\{0,\gamma-s\},\,1\right)$.

The precise values of~$A_1$ and~$A_2$ are given by
\begin{equation}\label{VALUE A}
\begin{split}&
A_1:= \frac{1}{2\cos\left(\frac{\pi s}{2}\right) }\,
\left( \frac{\Gamma(s-\gamma)}{\Gamma(-\gamma)}
-\frac{\Gamma(1+\gamma)}{\Gamma(1-s+\gamma)}\right)
\left[
\tan\left(\pi\left(\gamma-\frac{s}{2}\right)\right)
\mathcal{A}_\sharp(e_n)-\mathcal{B}_\sharp(e_n)\right]\\
{\mbox{and }}\qquad&A_2:= -\frac{s\,\Gamma(1+\gamma)\,\Gamma(s-\gamma)}{2\cos\left(\frac{\pi s}{2}\right) \,\Gamma(1-s)\,\Gamma(1+s)}\,\left[
\tan\left(\frac{\pi s}{2}\right)
\mathcal{A}_\sharp(e_n)-\mathcal{B}_\sharp(e_n)\right].
\end{split}\end{equation}
\end{lem}

\begin{proof} Let us first assume that
\begin{equation}\label{CASO1}
\gamma\in(0,s).\end{equation} 
In this case, using the notation in~\eqref{KSHAR}, we have that
\begin{equation}\label{SCEA}
\begin{split}&
\sqrt{L} \big( (x_n)_+^\gamma\big)=A_1 (x_n)_+^{\gamma-s}+A_2 (x_n)_-^{\gamma-s},\\
&{\mbox{with}}\qquad A_1:=\int_{\R^n}\big(
1 - (1+y_n)_+^\gamma\big)\,K_\sharp(y)\,dy\qquad {\mbox{and}}\qquad A_2:=-\int_{\R^n}
(y_n-1)_+^\gamma\,K_\sharp(y)\,dy.
\end{split}
\end{equation}
To check this, we change variable~$y\mapsto |x_n|\,y$ and we recall~\eqref{KSHAR},
thus obtaining that
\begin{equation}\label{1394}\begin{split}
\sqrt{L} \big( (x_n)_+^\gamma\big)\,&=\int_{\R^n}\big(
(x_n)_+^\gamma - (x_n+y_n)_+^\gamma\big)\,K_\sharp(y)\,dy\\&=|x_n|^{n}\,
\int_{\R^n}\big(
(x_n)_+^\gamma - (x_n+|x_n|\,y_n)_+^\gamma\big)\,K_\sharp(|x_n|\,y)\,dy\\&=
|x_n|^{-s}\,\int_{\R^n}\big(
(x_n)_+^\gamma - (x_n+|x_n|\,y_n)_+^\gamma\big)\,K_\sharp(y)\,dy.
\end{split}\end{equation}
Now, we distinguish whether~$x_n>0$ or~$x_n<0$. In the first case, we deduce from~\eqref{1394}
that
$$ \sqrt{L} \big( (x_n)_+^\gamma\big)=
x_n^{-s}\,\int_{\R^n}\big(
(x_n)^\gamma - (x_n+x_n\,y_n)_+^\gamma\big)\,K_\sharp(y)\,dy=
x_n^{\gamma-s}\,\int_{\R^n}\big(
1 - (1+y_n)_+^\gamma\big)\,K_\sharp(y)\,dy
.$$
This proves~\eqref{SCEA} in this case, and we now suppose that~$x_n<0$.
Then, from~\eqref{1394},
$$ \sqrt{L} \big( (x_n)_+^\gamma\big)=
-(x_n)_-^{-s}\,\int_{\R^n}
(-(x_n)_-+(x_n)_-\,y_n)_+^\gamma\,K_\sharp(y)\,dy=
-(x_n)_-^{\gamma-s}\,\int_{\R^n}
(y_n-1)_+^\gamma\,K_\sharp(y)\,dy,$$
hence completing the proof of~\eqref{SCEA}.

Now we claim that if, for some~$R>0$,
\begin{equation}\label{INS-0}
\begin{split}&
{\mbox{$\psi\in C^\infty(\R^n)\cap W^{1,\infty}(\R^n)\cap C^\gamma(\R^n)$
such that~$\psi(x',0)=0$,}}\\&{\mbox{$\psi(x)=0$ if~$|x'|>R$
and~$\partial_n\psi\in C^\infty_c(\R^n)$,}}\\&{\mbox{and setting~$
\varphi(x):=(x_n)_+^{\gamma}\psi(x)$,}}\end{split}\end{equation} we have that,
for every~$x\in\R^{n-1}\times[-1,1]$,
\begin{equation}\label{DERHAM}
|\partial_n \sqrt{L}\varphi(x)|\le C'\,|x_n|^{\gamma-s}
\end{equation}
for some~$C'>0$. To check this, 
we first observe that
\begin{equation}\label{SDRer}
\partial_n \sqrt{L}\varphi(x)= \partial_n \int_{\R^n}
\big(\varphi(x)-\varphi(x+y)\big)\,K_\sharp(y)\,dy =\int_{\R^n}
\big(\partial_n \varphi(x)-\partial_n \varphi(x+y)\big)\,K_\sharp(y)\,dy .
\end{equation}
If~$x_n<0$, we have that~$\partial_n \varphi(x)=0$
and, more generally~$\partial_n \varphi(x+y)=0$ if~$x_n+y_n<0$.
Using the change of variable~$z:= -y/x_n=y/|x_n|$
and the homogeneity of the kernel in~\eqref{KSHAR},
the above observation and~\eqref{SDRer} give that, if~$x_n<0$,
\begin{eqnarray*}&&
|\partial_n \sqrt{L}\varphi(x)|\\&= &\left|
\int_{\R^n\cap\{ x_n+y_n\ge0\}}
\partial_n \varphi(x+y)\,K_\sharp(y)\,dy \right|\\&\le&
\int_{\R^n\cap\{ x_n+y_n\ge0\}}
\Big| 
\gamma (x_n+y_n)^{\gamma-1}\psi(x+y)
+ (x_n+y_n)^{\gamma}\partial_n\psi(x+y)
\Big|\,K_\sharp(y)\,dy\\&\le&
\int_{\R^n\cap\{ x_n+y_n\ge0\}} \Big(
\gamma (x_n+y_n)^{\gamma-1}\big| \psi(x+y)-\psi(x'+y',0)\big|
+ (x_n+y_n)^{\gamma}\big|\partial_n\psi(x+y)\big|\Big)\,K_\sharp(y)\,dy
\\ &\le& (\gamma+1)\|\partial_n\psi\|_{L^\infty(\R^n)}
\int_{\R^n\cap\{ x_n+y_n\ge0\}} (x_n+y_n)^{\gamma}\,K_\sharp(y)\,dy\\
&=& |x_n|^n
(\gamma+1)\|\partial_n\psi\|_{L^\infty(\R^n)}
\int_{\R^n\cap\{ x_n(1-z_n)\ge0\}} (x_n(1-z_n))^{\gamma}\,K_\sharp(|x_n| z)\,dz\\
&=& |x_n|^{\gamma-s}
(\gamma+1)\|\partial_n\psi\|_{L^\infty(\R^n)}
\int_{\R^n\cap\{ z_n\ge1\}} (z_n-1)^{\gamma}\,K_\sharp( z)\,dz.
\end{eqnarray*}
This proves~\eqref{DERHAM} when~$x_n<0$.

Hence, we now focus on the proof of~\eqref{DERHAM} when~$x_n>0$.
To this end, 
we exploit
the substitution~$z:=y/x_n$ and
the homogeneity of the kernel in~\eqref{KSHAR}, and we
observe that
\begin{equation}\label{po0-11}
\begin{split}&
\int_{ B_{x_n/2}}
\big| \partial_n\varphi(x)-\partial_n\varphi(x+y)\big|\,K_\sharp(y)\,dy=
x_n^n\int_{ B_{1/2}}
\big| \partial_n\varphi(x)-\partial_n\varphi(x+x_nz)\big|\,K_\sharp(x_nz)\,dz\\
&\qquad=
x_n^{-s}\int_{ B_{1/2}}
\big| \partial_n\varphi(x)-\partial_n\varphi(x+x_nz)\big|\,K_\sharp(z)\,dz
\\&\qquad\le
x_n^{-s}\int_{ B_{1/2}}\Big(\gamma
\Big| (x_n)_+^{\gamma-1}\psi(x)-(x_n+x_nz_n)_+^{\gamma-1}\psi(x+x_nz)
\Big|\\&\qquad\qquad\quad+\Big|(x_n)_+^{\gamma}\partial_n\psi(x)-
(x_n+x_nz_n)_+^{\gamma}\partial_n\psi(x+x_nz)
\Big|\Big)\,K_\sharp(z)\,dz
\\&\qquad=\gamma
x_n^{\gamma-1-s}\int_{ B_{1/2}}
\Big| \psi(x)-(1+z_n)^{\gamma-1}\psi(x+x_nz)
\Big|\,K_\sharp(z)\,dz\\&\qquad\qquad\quad+
x_n^{\gamma-s}\int_{ B_{1/2}}
\Big|\partial_n\psi(x)-
(1+z_n)^{\gamma}\partial_n\psi(x+x_nz)
\Big| \,K_\sharp(z)\,dz.
\end{split}
\end{equation}
We also remark that, if~$z\in B_{1/2}$,
\begin{equation}\label{po0-12}
\begin{split}&
\Big| \psi(x)-(1+z_n)^{\gamma-1}\psi(x+x_nz)
\Big|\\
\le\;&|\psi(x)|\;
| 1-(1+z_n)^{\gamma-1}|
+(1+z_n)^{\gamma-1}
| \psi(x)-\psi(x+x_nz)|
\\  \le\;&C\,\Big(
|\psi(x)|\;
|z_n|
+| \psi(x)-\psi(x+x_nz)|\Big)\\
=\;&C\,\Big(
|\psi(x)-\psi(x',0)|\;
|z_n|
+| \psi(x)-\psi(x+x_nz)|\Big)\\ \le\;& Cx_n\,|z|,
\end{split}\end{equation}
for some~$C>0$, possibly varying from line to line.

Furthermore, if~$x\in B_1$ and~$z\in B_{1/2}$,
\begin{eqnarray*}&&
\Big|\partial_n\psi(x)-
(1+z_n)^{\gamma}\partial_n\psi(x+x_nz)
\Big|\\ &\le&
|\partial_n\psi(x)|\,
|1-
(1+z_n)^{\gamma}|+(1+z_n)^{\gamma}
|\partial_n\psi(x)-
\partial_n\psi(x+x_nz)|\\&\le& C\,|z|,
\end{eqnarray*}
up to renaming~$C>0$.

Then, we insert the latter inequality and~\eqref{po0-12} into~\eqref{po0-11},
and we find that
\begin{equation}\label{INS-1-6}
\int_{ B_{x_n/2}}
\big| \partial_n\varphi(x)-\partial_n\varphi(x+y)\big|\,K_\sharp(y)\,dy\le Cx_n^{\gamma-s}
\int_{ B_{1/2}}
|z| \,K_\sharp(z)\,dz\le Cx_n^{\gamma-s},
\end{equation}
up to renaming~$C$.

Now, we  let
$$ \Psi(x):=\int_0^1 \partial_n\psi(x',t x_n)\,dt$$
and we observe that
$$ \psi(x)=\psi(x)-\psi(x',0)=
\int_0^{x_n} \partial_n\psi(x',t)\,dt=x_n\Psi(x).$$
For~$\theta\in\R$, we also use the notation
$$ r_+^\theta:=\begin{cases}
r^\theta & {\mbox{ if }} r>0,\\
0 & {\mbox{ otherwise,}}
\end{cases}$$
and we
set
\begin{equation}\label{310} \Psi_1(x):= (x_n)_+^\gamma\partial_n\psi(x)\qquad{\mbox{ and }}\qquad
\Psi_2(x):=(x_n)_+^{\gamma-1}\psi(x).\end{equation}
Notice that
\begin{equation}\label{308}
\partial_n\varphi=\Psi_1+\gamma\Psi_2
\end{equation}
and
\begin{equation}\label{312}
\Psi_2(x)=(x_n)_+^\gamma \Psi(x).
\end{equation}
Also, by~\eqref{INS-0}, we have that
\begin{equation}\label{309}
\Psi_1\in C^\gamma(\R^n)
\end{equation}
and, using also~\eqref{312},
\begin{equation}\label{314}
\Psi_2\in C^\gamma( \{ |x_n|\le2^{1/\gamma}\}).
\end{equation}
Now we claim that, for all~$p$, $q\in\R^n$,
\begin{equation}\label{313}
|\Psi_2(p)-\Psi_2(q)|\le C|p-q|^\gamma,
\end{equation}
for some~$C>0$. To check this, we distinguish different cases.
To start with, we point out that if~$|p_n|$, $|q_n|\le2^{1/\gamma}$, then~\eqref{313}
follows from~\eqref{314}. Hence, up to exchanging~$p$ and~$q$,
we can assume that~$|p_n|\ge2^{1/\gamma}$.

Moreover, if~$|q_n|\le1$, we use~\eqref{312} to see that
\begin{equation}\label{315}
|\Psi_2(p)-\Psi_2(q)|\le |\Psi_2(p)|+|\Psi_2(q)|\le
(p_n)_+^\gamma |\Psi(p)|+|\Psi(q)|\le
\|\Psi\|_{L^\infty(\R^n)} (|p_n|^\gamma+1).
\end{equation}
Furthermore,
$$ |p_n|^\gamma-|q_n|^\gamma
= \frac{|p_n|^\gamma-1}4 +\frac{|p_n|^\gamma +1}{4}
+\frac{|p_n|^\gamma}2
-|q_n|^\gamma\ge
\frac{2-1}4 +\frac{|p_n|^\gamma +1}{4}
+1
-1>\frac{|p_n|^\gamma +1}{4}.
$$
This and~\eqref{315} give that
$$ |\Psi_2(p)-\Psi_2(q)|\le 4\|\Psi\|_{L^\infty(\R^n)}
(|p_n|^\gamma-|q_n|^\gamma),$$
which in turn gives~\eqref{313} in this case.

Hence, to complete the proof of~\eqref{313},
we can suppose that~$|p_n|\ge2^{1/\gamma}$
and~$|q_n|\ge1$, and therefore~$|p_n|$, $|q_n|\ge1$.

We observe that if~$p_n$, $q_n\le0$ then~\eqref{313} is obvious,
thanks to~\eqref{312}.
If instead~$p_n\ge0\ge q_n$, then, by~\eqref{312},
$$ |\Psi_2(p)-\Psi_2(q)|=|\Psi_2(p)|\le
\|\Psi\|_{L^\infty(\R^n)} p_n^\gamma
\le \|\Psi\|_{L^\infty(\R^n)} (p_n-q_n)^\gamma,
$$
which gives~\eqref{313} in this case (and a similar computation
would hold if~$q_n\ge0\ge p_n$).

Hence, from now on, we can assume that~$p_n$, $q_n\ge1$.
In this case, we observe that the map~$x\mapsto
(x_n)_+^{\gamma-1}$ belongs to~$C^\gamma(\{ x_n\ge1\})$.
Then, recalling the settings in~\eqref{INS-0}
and~\eqref{310}, we conclude that~$\Psi_2\in
C^\gamma(\{ x_n\ge1\})$, from which~\eqref{313} follows.

Now, in light of~\eqref{308}, \eqref{309}
and~\eqref{313}, we conclude that, for all~$p$, $q\in\R^n$,
$$ |\partial_n\varphi(p)-\partial_n\varphi(q)|\le C|p-q|^\gamma,$$
for a suitable~$C>0$.

Consequently,
for all~$x\in B_1$ with~$x_n>0$, we have that
\begin{equation}\label{INS-1}
\begin{split}&
\int_{\R^n\setminus B_{x_n/2}}
\big| \partial_n\varphi(x)-\partial_n\varphi(x+y)\big|\,K_\sharp(y)\,dy=x_n^n
\int_{\R^n\setminus B_{1/2}}
\big| \partial_n\varphi(x)-\partial_n\varphi(x+x_n z)\big|\,K_\sharp(x_n z)\,dz
\\&\qquad=x_n^{-s}
\int_{\R^n\setminus B_{1/2}}
\big| \partial_n\varphi(x)-\partial_n\varphi(x+x_n z)\big|\,K_\sharp(z)\,dz\\&\qquad
\le C\,x_n^{\gamma-s}
\int_{\R^n\setminus B_{1/2}}|z|^\gamma\,K_\sharp(z)\,dz\le C\,x_n^{\gamma-s},
\end{split}\end{equation}
for some~$C>0$,
thanks to the substitution~$z:=y/x_n$ and
the homogeneity of the kernel in~\eqref{KSHAR}.

Then, combining~\eqref{SDRer}, \eqref{INS-1-6}
and~\eqref{INS-1}, we obtain~\eqref{DERHAM}, as desired.

Now, we define
\begin{equation}\label{h1h2h}\begin{split}
h_1(x)\,&:=
\sqrt{L}\Big( (x_n)_+^\gamma\,\big(\eta(x)-\eta(x',0)\big)\Big),\\
h_2(x)\,&:=
\sqrt{L}\big( (x_n)_+^\gamma\,\eta(x',0)\big)-
\sqrt{L}\big( (x_n)_+^\gamma\big)\,\eta(x',0)\\
{\mbox{and }}\qquad h(x)\,&:=h_1(x)+h_2(x).
\end{split}\end{equation}
We observe that
\begin{eqnarray*}&&
\sqrt{L}u(x)=\sqrt{L}\big( (x_n)_+^\gamma\,\eta(x)\big)=
\sqrt{L}\big( (x_n)_+^\gamma\big)\,\eta(x',0)
+ h_1(x)+h_2(x)\\
&&\qquad=
\big(A_1(x_n)_+^{\gamma-s}+A_2(x_n)_-^{\gamma-s}\big)\,\eta(x',0)
+ h(x),
\end{eqnarray*}
thanks to~\eqref{SCEA}, and this gives~\eqref{EXP} in this case.

To prove~\eqref{LIA}, we point out that, if~$|x_n|\le1$,
\begin{equation}\label{DU1}
|\partial_n h_1(x',x_n)|\le C\,|x_n|^{\gamma-s},\end{equation}
for some~$C>0$,
thanks to~\eqref{DERHAM}, used here with~$\psi(x):=\eta(x)-\eta(x',0)$.

To estimate~$h_2$, we exploit~\eqref{KSHAR} and we observe that,
given two functions~$f$ and~$g$,
\begin{equation}\label{BEG}\begin{split}
\sqrt{L} (fg)(x)\,&=\int_{\R^n}\big(f(x)g(x)-f(x+y)g(x+y)\big)\,K_\sharp(y)\,dy
\\ &=
\int_{\R^n}\Big(f(x)\big(g(x)-g(x+y)\big)+g(x+y)\big(
f(x)-f(x+y)\big)\Big)\,K_\sharp(y)\,dy\\
&= f(x)\;\sqrt{L} g(x)+
\int_{\R^n} g(x+y)\big(
f(x)-f(x+y)\big)\,K_\sharp(y)\,dy\\
&= f(x)\;\sqrt{L} g(x)+
\int_{\R^n} g(x)\big(
f(x)-f(x+y)\big)\,K_\sharp(y)\,dy\\&\qquad+
\int_{\R^n} \big(g(x+y)-g(x)\big)\big(
f(x)-f(x+y)\big)\,K_\sharp(y)\,dy\\
&= f(x)\;\sqrt{L} g(x)+g(x)\;\sqrt{L} f(x)-
\int_{\R^n} \big(g(x)-g(x+y)\big)\big(
f(x)-f(x+y)\big)\,K_\sharp(y)\,dy.
\end{split}\end{equation}
Hence, by~\eqref{h1h2h},
\begin{equation}\label{s698} h_2(x)= (x_n)_+^\gamma\,\sqrt{L}\big(\eta(x',0)\big)
-\int_{\R^n} \big((x_n)_+^\gamma-(x_n+y_n)_+^\gamma\big)\big(
\eta(x',0)-\eta(x'+y',0)\big)\,K_\sharp(y)\,dy.
\end{equation}
We notice that, if~$x_n\in[-1,0]$,
\begin{equation}\label{nondi}
\begin{split}&
\int_{\R^n} \big|(x_n)_+^{\gamma-1}-(x_n+y_n)_+^{\gamma-1}\big|\;\big|
\eta(x',0)-\eta(x'+y',0)\big|\,K_\sharp(y)\,dy\\
=\;& \int_{\R^n} (y_n-|x_n|)_+^{\gamma-1}\,\big|
\eta(x',0)-\eta(x'+y',0)\big|\,K_\sharp(y)\,dy\\
=\;& |x_n|^n \int_{\R^n} (|x_n|z_n-|x_n|)_+^{\gamma-1}\,\big|
\eta(x',0)-\eta(x'+|x_n|z',0)\big|\,K_\sharp(|x_n|z)\,dz\\
=\;& |x_n|^{\gamma-1-s} \int_{\R^n\cap\{z_n\ge1\}} (z_n-1)^{\gamma-1}\,\big|
\eta(x',0)-\eta(x'+|x_n|z',0)\big|\,K_\sharp(z)\,dz\\
\le\;&
|x_n|^{\gamma-s}\, \|\nabla\eta\|_{L^\infty(\R^n)}
\int_{\R^n\cap\{z_n\ge1\}} (z_n-1)^{\gamma-1}\,|z|\,K_\sharp(z)\,dz\\ \le\;& C\;|x_n|^{\gamma-s}
\end{split}
\end{equation}
for some~$C>0$,
thanks to
the homogeneity of the kernel in~\eqref{KSHAR}.

We also remark that, if~$x_n\in[0,1]$,
\begin{eqnarray*}&&
\int_{\R^n} \big|(x_n)_+^{\gamma-1}-(x_n+y_n)_+^{\gamma-1}\big|\;\big|
\eta(x',0)-\eta(x'+y',0)\big|\,K_\sharp(y)\,dy\\
&=&x_n^n
\int_{\R^n} \big| x_n^{\gamma-1}-(x_n+x_n z_n)_+^{\gamma-1}\big|\;\big|
\eta(x',0)-\eta(x'+x_n z',0)\big|\,K_\sharp(x_n z)\,dz\\
&=&x_n^{\gamma-1-s}
\int_{\R^n} \big| 1-(1+ z_n)_+^{\gamma-1}\big|\;\big|
\eta(x',0)-\eta(x'+x_n z',0)\big|\,K_\sharp( z)\,dz\\
&\le& 2\|\eta\|_{C^1(\R^n)}\;
x_n^{\gamma-1-s}
\int_{\R^n} \big| 1-(1+ z_n)_+^{\gamma-1}\big|\;\min\{1,
x_n |z'|\}\,K_\sharp( z)\,dz\\
&\le& C
x_n^{\gamma-1-s}\Bigg(x_n
\int_{B_{1/2}} |z|\,K_\sharp( z)\,dz+
x_n
\int_{B_2\setminus B_{1/2}} (1+(1+z_n)^{\gamma-1})|z|\,K_\sharp( z)\,dz\\
&&\qquad+
\int_{\R^n\setminus B_{2}} \min\{1,
x_n |z'|\}\,K_\sharp( z)\,dz\Bigg)\\
&\le& C x_n^{\gamma-s}+C
x_n^{\gamma-1-s}
\int_{\R^n\setminus B_{1/x_n}} K_\sharp( z)\,dz\\
&\le& C\,(x_n^{\gamma-s}+x_n^{\gamma-1})\\
&\le& C x_n^{\gamma-1}.
\end{eqnarray*}
with~$C>0$ varying from line to line,
thanks to
the homogeneity of the kernel in~\eqref{KSHAR}. {F}rom this and~\eqref{nondi},
we infer that, for all~$x_n\in[-1,1]$,
$$\int_{\R^n} \big|(x_n)_+^{\gamma-1}-(x_n+y_n)_+^{\gamma-1}\big|\;\big|
\eta(x',0)-\eta(x'+y',0)\big|\,K_\sharp(y)\,dy\le C\,|x_n|^{\gamma-1}.$$
Thus, recalling~\eqref{s698},
we conclude that~$|\partial_n h_2(x',x_n)|\le C\,|x_n|^{\gamma-1}$.
This and~\eqref{DU1} give that~$|\partial_n h(x',x_n)|\le C\,|x_n|^{\gamma-1}$,
for all~$x_n\in[-1,1]$,
up to renaming~$C$, which establishes~\eqref{LIA} in this case,
with~$\beta:=\gamma$.
\medskip

It remains to check~\eqref{VALUE A}. For this, we 
define~$\tilde K_1$ and~$\tilde K_2$ to be as~$K_1$ and~$K_2$
in~\eqref{KAPPI}
with~$s$ replaced by~$s/2$, namely
\begin{equation*}
\tilde K_1:=\int_\R\big( 1-(1+y)_+^\gamma\big)\,\frac{dy}{|y|^{1+s}}\qquad
{\mbox{and }}\qquad\tilde K_2:=\int_\R
\big( 1-(1+y)_+^\gamma\big)\,\frac{\sign y\,dy}{|y|^{1+s}}.
\end{equation*}
Then, we exploit~\eqref{SCEA},
\eqref{Espressioni0}
and~\eqref{Espressioni}, and the symmetry~$(\rho,\theta)\longmapsto(-\rho,-\theta)$,
finding that
\begin{eqnarray*}
A_1&=&\iint_{(0,+\infty)\times(\partial B_1)}\frac{\big(
1 - (1+\rho\theta_n)_+^\gamma\big)\,K_\sharp(\theta)}{\rho^{1+s}}\,d(\rho,\theta)\\
&=&
\iint_{(0,+\infty)\times(\partial B_1)}\frac{\big(
1 - (1+\rho\theta_n)_+^\gamma\big)\,K_\sharp^e(\theta)}{\rho^{1+s}}\,d(\rho,\theta)+
\iint_{(0,+\infty)\times(\partial B_1)}\frac{\big(
1 - (1+\rho\theta_n)_+^\gamma\big)\,K_\sharp^o(\theta)}{\rho^{1+s}}\,d(\rho,\theta)\\
&=&\frac12\,\left[
\iint_{\R\times(\partial B_1)}\frac{\big(
1 - (1+\rho\theta_n)_+^\gamma\big)\,K_\sharp^e(\theta)}{|\rho|^{1+s}}\,d(\rho,\theta)+
\iint_{\R\times(\partial B_1)}\frac{\big(
1 - (1+\rho\theta_n)_+^\gamma\big)\,\sign(\rho)\,K_\sharp^o(\theta)}{|\rho|^{1+s}}\,d(\rho,\theta)\right]\\
&=&\frac12\,\left[
\iint_{\R\times(\partial B_1)}\frac{|\theta_n|^s\,\big(
1 - (1+t)_+^\gamma\big)\,K_\sharp^e(\theta)}{|t|^{1+s}}\,d(t,\theta)\right.\\&&\qquad\left.+
\iint_{\R\times(\partial B_1)}\frac{|\theta_n|^s\,\big(
1 - (1+t)_+^\gamma\big)\,\sign(t/\theta_n)\,K_\sharp^o(\theta)}{|t|^{1+s}}\,d(t,\theta)\right]\\
&=&\frac12\,\left[\tilde K_1\,
\int_{ \partial B_1}|\theta_n|^s\,K_\sharp^e(\theta)d\theta+\tilde K_2\,
\int_{\partial B_1} |\theta_n|^s\,\sign(\theta_n)\,K_\sharp^o(\theta)\,d\theta\right]\\
&=&\frac{\tilde K_2\,s}{2\cos\left(\frac{\pi s}{2}\right) \,\Gamma(1-s)}\,\left[\frac{\tilde K_1}{\tilde K_2}\,\tan\left(\frac{\pi s}{2}\right)
\mathcal{A}_\sharp(e_n)-\mathcal{B}_\sharp(e_n)\right]\\
&=&\frac{\tilde K_2\,s}{2\cos\left(\frac{\pi s}{2}\right) \,\Gamma(1-s)}\,\left[
\tan\left(\pi\left(\gamma-\frac{s}{2}\right)\right)
\mathcal{A}_\sharp(e_n)-\mathcal{B}_\sharp(e_n)\right]
,\end{eqnarray*}
where the latter identity follows from~\eqref{QUOZ} (used here with~$s/2$ in the place of~$s$).

Then, using~\eqref{COIN} with~$s/2$ instead of~$s$,
$$ A_1=
\frac{1}{2\cos\left(\frac{\pi s}{2}\right) }\,
\left( \frac{\Gamma(s-\gamma)}{\Gamma(-\gamma)}
-\frac{\Gamma(1+\gamma)}{\Gamma(1-s+\gamma)}\right)
\left[
\tan\left(\pi\left(\gamma-\frac{s}{2}\right)\right)
\mathcal{A}_\sharp(e_n)-\mathcal{B}_\sharp(e_n)\right]
.$$
Also, in a similar way, if
$$ P:=
\int_{1}^{+\infty}\frac{
(t-1)^\gamma}{t^{1+s}}\,dt=\frac{\Gamma(1+\gamma)\,\Gamma(s-\gamma)}{\Gamma(1+s)},$$
we find that
\begin{eqnarray*}
-A_2&=&
\iint_{(0,+\infty)\times(\partial B_1)}\frac{
(\rho\theta_n-1)_+^\gamma\,K_\sharp(\theta)}{\rho^{1+s}}\,d(\rho,\theta)\\
&=&\frac12\,\left[
\iint_{\R\times(\partial B_1)}\frac{
(\rho\theta_n-1)_+^\gamma\,K_\sharp^e(\theta)}{|\rho|^{1+s}}\,d(\rho,\theta)+
\iint_{\R\times(\partial B_1)}\frac{
(\rho\theta_n-1)_+^\gamma\,\sign(\rho)\,K_\sharp^o(\theta)}{|\rho|^{1+s}}\,d(\rho,\theta)\right]\\
&=&\frac12\,\left[
\iint_{\R\times(\partial B_1)}\frac{|\theta_n|^s\,
(t-1)_+^\gamma\,K_\sharp^e(\theta)}{|t|^{1+s}}\,d(t,\theta)+
\iint_{\R\times(\partial B_1)}\frac{|\theta_n|^s\,
(t-1)_+^\gamma\,\sign(t/\theta_n)\,K_\sharp^o(\theta)}{|t|^{1+s}}\,d(t,\theta)\right]\\
&=&\frac{P}2\,\left[
\int_{ \partial B_1} |\theta_n|^s\,K_\sharp^e(\theta)\,d\theta+
\int_{ \partial B_1} |\theta_n|^s\,
\sign(\theta_n)\,K_\sharp^o(\theta) \,d\theta\right]\\&=&
\frac{P\,s}{2\cos\left(\frac{\pi s}{2}\right) \,\Gamma(1-s)}\,\left[
\tan\left(\frac{\pi s}{2}\right)
\mathcal{A}_\sharp(e_n)-\mathcal{B}_\sharp(e_n)\right]\\&=&
\frac{s\,\Gamma(1+\gamma)\,\Gamma(s-\gamma)}{2\cos\left(\frac{\pi s}{2}\right) \,\Gamma(1-s)\,\Gamma(1+s)}\,\left[
\tan\left(\frac{\pi s}{2}\right)
\mathcal{A}_\sharp(e_n)-\mathcal{B}_\sharp(e_n)\right].
\end{eqnarray*}
These observations complete the proof of~\eqref{VALUE A}
when~$\gamma\in(0,s)$.
\medskip

The arguments above established the desired claim in Lemma~\ref{LE.31} when~$\gamma\in(0,s)$,
and we now focus on the case~$\gamma\in(s,2s)$.
In this case, the growth of the functions at infinity does not allow
a direct computation.
To circumvent this difficulty,
one could exploit a method introduced in~\cite{2016arXiv161004663D}
to effectively compute, up to additional constants, the fractional operators
with functions growing at infinity ``with one degree more'' than what the operator
permits, but for the sake of concreteness we follow here
a more explicit approach. Namely, 
we take~$\zeta\in C^\infty_c([-10,10])$ with~$\zeta=1$ in~$[-4,4]$.

We claim that, for all~$x_n\in[-1,0)\cup(0,1]$ and all~$k\in\N$ with~$k\ge1$,
\begin{equation}\label{Y6yARUsAa}
\big|\partial_n^k \sqrt{L}\big( (x_n)_+^\gamma \zeta(x_n)\big)\big|\le C\,|x_n|^{\gamma -s-k},
\end{equation}
for some~$C>0$, possibly depending on~$k$.

To check this, given~$j\in\N$, we denote by~$\zeta^{(j)}$ the $j$th
derivative of~$\zeta$.
Moreover, we remark that for all~$i\in\N$
with~$i\ge1$,
\begin{equation} \label{p0op0i30-39}\begin{split}&\big|\chi_{(0,+\infty)}(x_n)-({\rm sign}(x_n)+z_n)_+^{\gamma-i}\big|\\ \le\;&
C\,\Big[|z_n|\,\chi_{(-1/2,1/2)}(z_n)+ \Big(1+
(z_n+1)_+^{\gamma-i}+(z_n-1)_+^{\gamma-i}\Big)\chi_{\R\setminus(-1/2,1/2)}(z_n)\Big].\end{split}\end{equation}
Moreover, using~\eqref{KSHAR}
and the substitution~$z:=\frac{y}{|x_n|}$,
\begin{equation}\label{0097690345di}
\begin{split}&
\sqrt{L}\big( (x_n)_+^{\gamma-i} \zeta^{(j)}(x_n)\big)\\ =\;&
\int_{\R^n}\big((x_n)_+^{\gamma -i}\zeta^{(j)}(x_n)-(x_n+y_n)_+^{\gamma-i} \zeta^{(j)}(x_n+y_n)
\big)\,K_\sharp(y)\,dy\\
=\;&|x_n|^n
\int_{\R^n}\big((x_n)_+^{\gamma -i}\zeta^{(j)}(x_n)-(x_n+|x_n|z_n)_+^{\gamma-i} \zeta^{(j)}(x_n+|x_n|z_n)
\big)\,K_\sharp(|x_n|z)\,dz
\\=\;&|x_n|^{\gamma -s-i}
\int_{\R^n}\big(
\chi_{(0,+\infty)}(x_n)
\zeta^{(j)}(x_n)-({\rm sign}(x_n)+z_n)_+^{\gamma-i} \zeta^{(j)}(x_n+|x_n|z_n)
\big)\,K_\sharp(z)\,dz
.\end{split}
\end{equation}
Using this and~\eqref{p0op0i30-39}, we deduce that
\begin{equation}\label{s9jgsorosgfap}
\begin{split}&
\big| \sqrt{L}\big( (x_n)_+^{\gamma-i} \zeta^{(j)}(x_n)\big)\big|\\ \le\;&
|x_n|^{\gamma -s-i}\left(
\int_{\R^n}
\chi_{(0,+\infty)}(x_n)\,\big|
\zeta^{(j)}(x_n)- \zeta^{(j)}(x_n+|x_n|z_n)
\big|\,K_\sharp(z)\,dz\right.\\&
\left.\qquad+
\int_{\R^n}\big|
\chi_{(0,+\infty)}(x_n)
-({\rm sign}(x_n)+z_n)_+^{\gamma-i}\big|\,| \zeta^{(j)}(x_n+|x_n|z_n)
|\,K_\sharp(z)\,dz
\right)\\ \le\;&C\,
|x_n|^{\gamma -s-i},
\end{split}
\end{equation}
up to renaming~$C>0$, possibly in dependence of~$i$ and~$j$.

Besides, we observe that~$\zeta'=\dots=\zeta^{(k)}=0$ in~$(-4,4)$ and therefore,
for all~$i$, $j\in\N$ with~$j\ge1$,
\begin{equation}\label{coho2}
\begin{split}&
\big| \sqrt{L}\big( (x_n)_+^{\gamma-i} \zeta^{(j)}(x_n)\big)\big|=
\left|
\int_{\R^{n}\cap\{|y_n|\in[2,11]\}} (x_n+y_n)_+^{\gamma-i} \,\zeta^{(j)}(x_n+y_n)\,K_\sharp(y)\,dy
\right|\\ &\qquad\qquad \le C\,
\int_{\R^n\setminus B_1} K_\sharp(y)\,dy\le C.
\end{split}\end{equation}
Combining this with~\eqref{s9jgsorosgfap}, we obtain that
if~$i$, $j\in\N$ and~$i+j=k\ge1$, then
$$ \big| \sqrt{L}\big( (x_n)_+^{\gamma-i} \zeta^{(j)}(x_n)\big)\big|\le
C\,|x_n|^{\gamma -s-i}\le C\,|x_n|^{\gamma -s-k},$$
and this gives~\eqref{Y6yARUsAa}, as desired.

Now, we focus on the case~$x_n>0$, since the case~$x_n<0$ is similar.
We claim that the following limit exists:
\begin{equation}\label{mology}
A_1:=\frac{1}{\gamma-s}\,\lim_{x_n\searrow0}
\frac{\partial_n\sqrt{L}\big( (x_n)_+^\gamma \zeta(x_n)\big)}{x_n^{\gamma-s-1}}
.\end{equation}
To check this, we observe that, by~\eqref{0097690345di},
when~$x_n\in(0,1]$,
\begin{equation}\label{coho}
\sqrt{L}\big( (x_n)_+^{\gamma-1} \zeta(x_n)\big)=
x_n^{\gamma -s-1}
\int_{\R^n}\big(1
-(1+z_n)_+^{\gamma-1} \zeta(x_n+x_nz_n)
\big)\,K_\sharp(z)\,dz.\end{equation}
We also remark that, by~\eqref{p0op0i30-39},
\begin{eqnarray*}
&& \big|1
-(1+z_n)_+^{\gamma-1} \zeta(x_n+x_nz_n)
\big|\,K_\sharp(z)\\
&& \big|\zeta(x_n)
-(1+z_n)_+^{\gamma-1} \zeta(x_n+x_nz_n)
\big|\,K_\sharp(z)\\
&\le&\Big(\big|\zeta(x_n)
- \zeta(x_n+x_nz_n)
\big|
+|\zeta(x_n+x_nz_n)|\;\big|1
-(1+z_n)_+^{\gamma-1} 
\big|\Big)\,K_\sharp(z)\\&\le&
C\,\Big[|z_n|\,\chi_{(-1/2,1/2)}(z_n)+ \Big(1+
(z_n+1)_+^{\gamma-i}+(z_n-1)_+^{\gamma-i}\Big)
\chi_{\R\setminus(-1/2,1/2)}(z_n)\Big]\,K_\sharp(z),
\end{eqnarray*}
and the latter function belongs to~$L^1(\R)$.

Hence, recalling~\eqref{coho} and the Dominated Convergence Theorem,
\begin{equation}\label{coho3}
\begin{split}
\lim_{x_n\searrow0}
\frac{\sqrt{L}\big( (x_n)_+^{\gamma-1} \zeta(x_n)\big)}{
x_n^{\gamma -s-1}}\,&=\lim_{x_n\searrow0}
\int_{\R^n}\big(1
-(1+z_n)_+^{\gamma-1} \zeta(x_n+x_nz_n)
\big)\,K_\sharp(z)\,dz\\&=
\int_{\R^n}\big(1
-(1+z_n)_+^{\gamma-1} 
\big)\,K_\sharp(z)\,dz.\end{split}
\end{equation}
Furthermore, in  view of~\eqref{coho2},
\[ \big| \sqrt{L}\big( (x_n)_+^{\gamma} \zeta'(x_n)\big)\big|\le C,
\]
whence
\[ \lim_{x_n\searrow0}\frac{ \sqrt{L}\big( (x_n)_+^{\gamma} \zeta'(x_n)\big)}{x_n^{\gamma-s-1}}=0.\]
This and~\eqref{coho3} yield that
\begin{equation}\label{mologyy}
\lim_{x_n\searrow0}\frac{\partial_n\sqrt{L}\big( (x_n)_+^\gamma \zeta(x_n)\big)}{x_n^{\gamma-s-1}}
=\gamma\,\int_{\R^n}\big(1
-(1+z_n)_+^{\gamma-1} 
\big)\,K_\sharp(z)\,dz,\end{equation}
and this establishes~\eqref{mology}.

We define
$$\Theta(x_n):=\sqrt{L}\big( (x_n)_+^\gamma \zeta(x_n)\big),$$
and we observe that~$\Theta$ is continuous at the origin, since~$\gamma>s$,
and, for this reason, we can set~$ A_0:=\Theta(0)$.

We also let
\begin{eqnarray*}\tilde h(x_n):=\Theta(x_n)-A_0-A_1 x_n^{\gamma-s}\qquad
{\mbox{and }}\qquad \Psi(x_n):=
\frac{\Theta'(x_n)}{x_n^{\gamma-s-1}}.
\end{eqnarray*}
We claim that, for all~$x_n\in(0,1)$,
\begin{equation}\label{7u08-dfg-92ekd}
|\Psi'(x_n)|\le \frac{C}{x_n^{\gamma-s}}.
\end{equation}
To check this, we observe that
$$ \Theta'(x_n)=\gamma\sqrt{L}\big( (x_n)_+^{\gamma-1} \zeta(x_n)\big)+
\sqrt{L}\big( (x_n)_+^\gamma \zeta'(x_n)\big),$$
and thus
$ \Psi(x_n)=\gamma\Psi_1(x_n)+\Psi_2(x_n)$,
with
\begin{eqnarray*}&&
\Psi_1(x_n):=x_n^{1+s-\gamma}\sqrt{L}\big( (x_n)_+^{\gamma-1} \zeta(x_n)\big)
\\{\mbox{and }}&&\Psi_2(x_n):=x_n^{1+s-\gamma}
\sqrt{L}\big( (x_n)_+^\gamma \zeta'(x_n)\big).
\end{eqnarray*}
Using~\eqref{0097690345di} with~$i:=1$ and~$j:=0$, we see that
$$ \Psi_1(x_n)=\int_{\R^n}\big(
1-(1+z_n)_+^{\gamma-1} \zeta(x_n(1+z_n))
\big)\,K_\sharp(z)\,dz,$$
and, as a result,
\begin{equation}\label{ejcomed9k-2ol}
\begin{split} |\Psi_1'(x_n)|\,&\le \int_{\R^n} 
(1+z_n)_+^{\gamma} |\zeta'(x_n(1+z_n))|
\,K_\sharp(z)\,dz\\&\le
C\, \left(1+\int_{\R^{n-1}\times [2,C/x_n]} 
z_n^{\gamma} 
\,K_\sharp(z)\,dz\right)\\&\le
C\, \left(1+\frac{1}{x_n^{\gamma-s}}\right)\\&\le\frac{C}{x_n^{\gamma-s}}.
\end{split}\end{equation}
Moreover, by~\eqref{coho2},
\begin{eqnarray*}
|\Psi_2'(x_n)|&\le& C\,\Big[
x_n^{s-\gamma}
\big|\sqrt{L}\big( (x_n)_+^\gamma \zeta'(x_n)\big)\big|+
x_n^{1+s-\gamma}\Big(\big|
\sqrt{L}\big( (x_n)_+^{\gamma-1} \zeta'(x_n)\big)\big|+
\big|
\sqrt{L}\big( (x_n)_+^\gamma \zeta''(x_n)\big)\big|
\Big)\Big]\\&\le&\frac{C}{{x_n^{\gamma-s}}}.
\end{eqnarray*}
Combining this with~\eqref{ejcomed9k-2ol}
we obtain~\eqref{7u08-dfg-92ekd}, as desired.

As a consequence of~\eqref{mology} and~\eqref{7u08-dfg-92ekd},
we have that, if~$x_n\in(0,1)$,
\begin{equation*}
\begin{split}&
\left|\Psi(x_n)-(\gamma-s)A_1\right|=\lim_{t\searrow0}
\left|\Psi(x_n)-\frac{\Theta'(t)}{t^{\gamma-s-1}}\right|=
\lim_{t\searrow0}
\left|\Psi(x_n)-\Psi(t)\right|\le\int_0^{x_n}|\Psi'(\tau)|\,d\tau\\
&\qquad\le C\,\int_0^{x_n}\tau^{s-\gamma}\,d\tau=
Cx_n^{1+s-\gamma},
\end{split}\end{equation*}
up to renaming~$C>0$.

Therefore, for all~$x_n\in(0,1)$,
\begin{equation}\label{htilde2}
|\tilde h'(x_n)|
=\big|\Theta'(x_n)-(\gamma-s)A_1 x_n^{\gamma-s-1}\big|
=x_n^{\gamma-s-1}\left|\Psi(x_n)-(\gamma-s)A_1\right|\le C.
\end{equation}
Hence, for all~$x_n\in(0,1)$,
\begin{equation}\label{htilde3}
|\tilde h(x_n)|=|\tilde h(x_n)-\tilde h(0)|\le Cx_n.
\end{equation}

Now, using the short notation
\begin{equation}\label{7520j4} B_{f,g}(x):=
\int_{\R^n} \big(g(x)-g(x+y)\big)\big(
f(x)-f(x+y)\big)\,K_\sharp(y)\,dy,\end{equation}
we see that
\begin{equation}\label{Saw}
\begin{split}
\sqrt{L}\big(
(x_n)_+^\gamma\,\eta(x)\big)\,&=
\sqrt{L}\big(
(x_n)_+^\gamma \zeta(x_n)\,\eta(x)\big)+\mu(x)\\
&=\sqrt{L}\big(
(x_n)_+^\gamma \zeta(x_n)\big)\eta(x)
+(x_n)_+^\gamma \zeta(x_n)\,
\sqrt{L}\eta(x)-B_{(x_n)_+^\gamma \zeta,\eta}(x)
+\mu(x)\\
&=\Theta(x_n)\eta(x)
+(x_n)_+^\gamma \zeta(x_n)\,
\sqrt{L}\eta(x)-B_{(x_n)_+^\gamma \zeta,\eta}(x)
+\mu(x)\\
&=
\big(\tilde h(x_n)+A_0+A_1 x_n^{\gamma-s}\big)
\eta(x)
+(x_n)_+^\gamma \zeta(x_n)\,
\sqrt{L}\eta(x)-B_{(x_n)_+^\gamma \zeta,\eta}(x)
+\mu(x)\\
&= \big(A_0+A_1 x_n^{\gamma-s}\big)\eta(x',0)+h(x),
\end{split}\end{equation}
with
\begin{eqnarray*}&&
\mu(x):=\sqrt{L}\Big(
(x_n)_+^\gamma \big(1-\zeta(x_n)\big)\,\eta(x)\Big)
\end{eqnarray*}
and 
\begin{eqnarray*}
&& h(x):=\tilde h(x_n)\eta(x)+
\big(A_0+A_1 x_n^{\gamma-s}\big)\eta(x)-
A_1 x_n^{\gamma-s}\eta(x',0)\\&&\qquad\qquad\qquad
+(x_n)_+^\gamma \zeta(x_n)\,
\sqrt{L}\eta(x)-B_{(x_n)_+^\gamma \zeta,\eta}(x)
+\mu(x).
\end{eqnarray*}
We remark that~\eqref{Saw} establishes~\eqref{EXP} in this case
(i.e., when~$x_n>0$, being the case~$x_n<0$ similar).

It remains to prove~\eqref{LIA}. 
For this, we first observe that~$1-\zeta=0$ in a neighborhood of the origin,
and therefore, for all~$x_n\in(0,1)$,
\begin{equation}\label{loLos-2}
\big|
\partial_n \mu(x)\big|\le C,
\end{equation}
for some~$C>0$.

Furthermore,
\begin{equation}\label{75955df552384-1}
\Big|\partial_n \Big( (x_n)_+^\gamma \zeta(x_n)\,
\sqrt{L}\eta(x)\Big)\Big|\le Cx_n^{\gamma-1}.
\end{equation}
In addition,
\begin{equation}\label{75955df552384}
\begin{split}&
\Big| \partial_n\Big(\big(A_0+A_1 x_n^{\gamma-s}\big)\big(\eta(x)-\eta(x',0)\big)\Big)\Big|
\le
C\Big( x_n
\big| \partial_n\big(A_0+A_1 x_n^{\gamma-s}\big)\big|
+ 
\big| \partial_n\eta(x)\big|\Big)\\&\qquad\qquad\le C\big(x_n^{\gamma-s}+1\big)\le C.
\end{split}\end{equation}
Moreover, recalling~\eqref{htilde2}
and~\eqref{htilde3},
\begin{equation}\label{75955df552384-78}
\Big| \partial_n\big( \tilde h(x_n)\eta(x)\big)\Big|\le
C,
\end{equation}
up to renaming~$C$.

Now,
we define~$\theta(x_n):=(x_n)_+^\gamma\zeta(x_n)$.
In light of~\eqref{7520j4},
we see that
\begin{eqnarray*}
\partial_n B_{(x_n)_+^\gamma \zeta,\eta}(x)&=&
\partial_n\left(
\int_{\R^n} \big(\theta(x_n)-\theta(x_n+y_n)\big)\big(
\eta(x)-\eta(x+y)\big)\,K_\sharp(y)\,dy
\right)\\&=&
\int_{\R^n} \big(\partial_n\theta(x_n)-\partial_n\theta(x_n+y_n)\big)\big(
\eta(x)-\eta(x+y)\big)\,K_\sharp(y)\,dy\\&&\quad
+\int_{\R^n} \big(\theta(x_n)-\theta(x_n+y_n)\big)\big(
\partial_n\eta(x)-\partial_n\eta(x+y)\big)\,K_\sharp(y)\,dy.
\end{eqnarray*}
Since
$$|\partial_n\theta(x_n)|=\big|\gamma (x_n)_+^{\gamma-1}\zeta(x_n)+
(x_n)_+^\gamma\partial_n\zeta(x_n)\big|\le C (x_n)_+^{\gamma-1},$$
we thereby conclude that
\begin{eqnarray*}|
\partial_n B_{(x_n)_+^\gamma \zeta,\eta}(x)|&\le&C\left(
\int_{\R^n} (x_n+y_n)^{\gamma-1}_+\,\min\{1,|y|\}\,K_\sharp(y)\,dy
+(1+x_n^{\gamma-1})\int_{\R^n} \min\{1,|y|\}\,K_\sharp(y)\,dy\right)\\
&\le&C\left(
\int_{\R^n} (x_n+|y_n|)^{(\gamma-1)_+}\,\min\{1,|y|\}\,K_\sharp(y)\,dy
+1+x_n^{\gamma-1}\right)\\
&\le& C(1+x_n^{\gamma-1}).
\end{eqnarray*}
{F}rom this, \eqref{loLos-2}, \eqref{75955df552384-1}, \eqref{75955df552384}
and~\eqref{75955df552384-78},
we obtain~\eqref{LIA}, as desired.

The proof of~\eqref{VALUE A} in this case is similar to that in the case~$\gamma\in(0,s)$,
e.g. using here~\eqref{mology} and~\eqref{mologyy} to write~$A_1$,
and the details are therefore omitted.
\end{proof}

As a byproduct of Lemma~\ref{LE.31}, we also obtain the following
decay for the function~$h$:

\begin{cor}
In the notation of Lemma~\ref{LE.31},
we have that, for all~$x\in \R^{n-1}\times[-2,2]$,
\begin{equation}\label{BV33}
|h(x)|\le \frac{C }{|x_n|^{(s-\gamma)_+}\;(1+|x|^{n+s})},
\end{equation}
for some~$C>0$.
\end{cor}

\begin{proof} If~$u$ and~$\eta$ are as in Lemma~\ref{LE.31},
we let~$R\ge4$ be such that~$\eta=0$ outside~$B_R$. We claim that,
for all~$x\in [-2R,2R]^{n-1}\times[-2,2]$,
\begin{equation}\label{BV831}
|\sqrt{L}u(x)|\le\frac{C }{|x_n|^{(s-\gamma)_+}},
\end{equation}
for some~$C>0$.
Indeed, if~$\gamma\in(0,s)$, we see that
\begin{eqnarray*}&&
\int_{\R^n}\frac{| (x_n)_+^\gamma\eta(x)-(x_n+y_n)_+^\gamma\eta(x+y)|}{|y|^{n+s}}\,dy\\&\le&
\|\eta\|_{L^\infty(\R^n)}
\int_{\R^n}\frac{| (x_n)_+^\gamma-(x_n+y_n)_+^\gamma|}{|y|^{n+s}}\,dy
+\|\eta\|_{C^1(\R^n)}\int_{\R^n}\frac{(x_n+y_n)_+^\gamma\min\{1,|y|\}}{|y|^{n+s}}\,dy\\
&\le&
\|\eta\|_{L^\infty(\R^n)}\,|x_n|^{\gamma-s}
\int_{\R^n}\frac{| ({\rm sign}(x_n))_+^\gamma-({\rm sign}(x_n)+z_n)_+^\gamma|}{ |z|^{n+s}}\,dz
+\|\eta\|_{C^1(\R^n)}\int_{\R^n}\frac{(2+|y|)^\gamma\min\{1,|y|\}}{|y|^{n+s}}\,dy\\
&\le& C\,(|x_n|^{\gamma-s}+1)\\
&\le& C\,|x_n|^{\gamma-s},
\end{eqnarray*}
up to renaming~$C$ line after line, and this gives~\eqref{BV831}
in this case.

If instead~$\gamma\in(s,2s)$,
\begin{eqnarray*}&&
\int_{\R^n}\frac{| (x_n)_+^\gamma\eta(x)-(x_n+y_n)_+^\gamma\eta(x+y)|}{|y|^{n+s}}\,dy\\
&\le& (x_n)_+^\gamma
\int_{\R^n}\frac{| \eta(x)-\eta(x+y)|}{|y|^{n+s}}\,dy
+
\int_{\R^n}\frac{| (x_n)_+^\gamma-(x_n+y_n)_+^\gamma|\;|\eta(x+y)|}{|y|^{n+s}}\,dy\\&\le&
2^\gamma\|\eta\|_{C^1(\R^n)}
\int_{\R^n}\frac{\min\{1,|y|\}}{|y|^{n+s}}\,dy
+\|\eta\|_{L^\infty(\R^n)}
\int_{B_{CR}}\frac{| (x_n)_+^\gamma-(x_n+y_n)_+^\gamma|}{|y|^{n+s}}\,dy\\&\le&
C\,\left(1+
\int_{B_{CR}}\frac{|y|^\gamma}{|y|^{n+s}}\,dy\right)\\&\le&C.
\end{eqnarray*}
This completes the proof of~\eqref{BV831}.

Now, we claim that, for all~$x\in \R^{n-1}\times[-2,2]$,
\begin{equation}\label{BV832}
|\sqrt{L}u(x)|\le\frac{C }{|x_n|^{(s-\gamma)_+}\;(1+|x|^{n+s})},
\end{equation}
for some~$C>0$. Indeed, if~$x'\in [-2R,2R]^{n-1}$, then~\eqref{BV832}
is a consequence of~\eqref{BV831}. If instead~$x'\in\R^{n-1}\setminus
[-2R,2R]^{n-1}$, we see that
\begin{eqnarray*}&&
\int_{\R^n}\frac{| (x_n)_+^\gamma\eta(x)-(x_n+y_n)_+^\gamma\eta(x+y)|}{|y|^{n+s}}\,dy
=
\int_{B_R(-x)}\frac{| (x_n+y_n)_+^\gamma\eta(x+y)|}{|y|^{n+s}}\,dy\\&&\qquad\le
R^\gamma \|\eta\|_{L^\infty(\R^n)}\int_{B_R(-x)}\frac{dy}{|y|^{n+s}}\le
R^\gamma \|\eta\|_{L^\infty(\R^n)}\int_{B_R(-x)}\frac{dy}{(|x|/2)^{n+s}}=\frac{C}{|x|^{n+s}},
\end{eqnarray*}
which yields~\eqref{BV832}, as desired.

As a result, exploiting~\eqref{EXP} and~\eqref{BV832}, for all~$x\in \R^{n-1}\times[-2,2]$,
\begin{eqnarray*}
|h(x)|\le \frac{C\,|\eta(x',0)|}{|x_n|^{(s-\gamma)_+}}+|\sqrt{L}u(x)|\le
\frac{C\,\chi_{B_R}(x)}{|x_n|^{(s-\gamma)_+}}+
\frac{C }{|x_n|^{(s-\gamma)_+}\;(1+|x|^{n+s})},
\end{eqnarray*}
from which~\eqref{BV33} plainly follows.
\end{proof}

\subsection{Proof of Proposition~\ref{flat-case}}

With the previous preliminary work, we are now in the position
of giving the

\begin{proof}[Proof of   Proposition~\ref{flat-case}] Let us focus on the case~$\gamma\ne s$ (the case $\gamma=s$ can be covered after simply taking the limit ~$\gamma\to s$).
Given~$\lambda\in(0,1)$, we introduce the notation
\begin{eqnarray*}
&& u_\lambda(x):=u(x+\lambda e_n), \qquad 
v_\lambda(x):=v(x+\lambda e_n)
\\ {\mbox{and }}&&
E_\lambda(u,v):=\int_{\R^n_+} \big( u_\lambda(x)\,L^* v(x)+Lu(x)\,v_\lambda(x)\big)\,dx.
\end{eqnarray*}
Since~$u=O(|x_n|^\gamma)$ near~$\R^{n-1}\times\{0\}$, we have that~$\sqrt{L}u(x)=O(|x_n|^{\gamma-s})$,
which is locally integrable. Similarly, $\sqrt{L^*}v(x)=O(|x_n|^{\gamma^*-s})$
which is locally integrable. Accordingly, we can write that
\begin{equation}\label{LAE}
E_\lambda(u,v)=
\int_{\R^n_+} \big( \sqrt{L} u_\lambda(x)\,\sqrt{L^*} v(x)+\sqrt{L}u(x)\,\sqrt{L^*}v_\lambda(x)\big)\,dx
=\int_{\R^n_+} \big( w_\lambda(x)\,w^*(x)+w(x)\,w^*_\lambda(x)\big)\,dx,
\end{equation}
where
\begin{equation}\label{54398703}
w:=\sqrt{L} u\qquad{\mbox{ and }}\qquad w^*:=\sqrt{L^*} v.
\end{equation}
Using the change of variable~$x\mapsto x-\frac\lambda2\,e_n$,
we can write~\eqref{LAE} as
\begin{equation}\label{LAE2}
E_\lambda(u,v)=I_\lambda(w,w^*),
\end{equation}
where
\begin{equation*}
I_\lambda(w,w^*):=\int_{\R^n_+} \big( w_{\lambda/2}(x)\,w^*_{-\lambda/2}(x)
+w_{-\lambda/2}(x)\,w^*_{\lambda/2}(x)\big)\,dx.
\end{equation*}
We claim that
\begin{equation}\label{M:SX}
\lim_{\lambda\searrow0}\frac{E_\lambda(u,v)-E_0(u,v)}\lambda=
\int_{\R^n_+} \big(
\partial_n u(x)\,L^* v(x)+Lu(x)\,\partial_n v(x)\big)\,dx.
\end{equation}
For this, we let~$R>0$ be such that the support of~$\eta$
is contained in~$B_R$, and
we observe that, for every~$x\in\R^{n-1}\times(0,R]$,
\begin{eqnarray*}
&& \left| \frac{u_\lambda(x)-u(x)}{\lambda}\right|=
\left| \frac{(x_n+\lambda)^\gamma\,\eta(x+\lambda e_n)-x_n^\gamma\,\eta(x)}{\lambda}\right|
\\ &&\qquad\le
\left| \frac{(x_n+\lambda)^\gamma\,\eta(x+\lambda e_n)-(x_n+\lambda)^\gamma\,\eta(x)}{\lambda}\right|
+
\left| \frac{(x_n+\lambda)^\gamma\,\eta(x)-x_n^\gamma\,\eta(x)}{\lambda}\right|
\\ &&\qquad\le (R+\lambda)^\gamma \|\eta\|_{C^1(\R^n)}
+\frac{\gamma\,\|\eta\|_{L^\infty(\R^n)}}{\lambda}\,
\int_0^\lambda (x_n+t)^{\gamma-1}\,dt\\
&&\qquad\le
2^\gamma \,\|\eta\|_{C^1(\R^n)}
+\gamma\,\|\eta\|_{L^\infty(\R^n)}\,x_n^{\gamma-1}.
\end{eqnarray*}
In fact, since~$u(x)=u_\lambda(x)=0$ whenever~$x_n>R$,
we have that the estimate above holds true for all~$x\in\R^n_+$.
As a consequence, recalling that~$L^*v\in L^\infty(\R^n_+)\cap L^1(\R^n_+)$,
we see that, in~$\R^n_+$,
$$ \left| \frac{u_\lambda(x)-u(x)}{\lambda}\;L^*v(x)\right|\le
C\,(1+x_n)^{\gamma-1}\,|L^*v(x)|\in L^1(\R^{n-1}\times(0,1)),$$
for some~$C>0$, and hence, by Dominated Convergence Theorem,
\begin{equation}\label{Lvsu}
\lim_{\lambda\searrow0} \int_{\R^n_+}
\frac{u_\lambda(x)-u(x)}{\lambda}\;L^*v(x)\,dx=
\int_{\R^{n-1}\times (0,1)}
\partial_nu(x)\,L^*v(x)\,dx.
\end{equation}
Similarly, one can prove that
\begin{equation*}
\lim_{\lambda\searrow0} \int_{\R^n_+}
\frac{v_\lambda(x)-v(x)}{\lambda}\;Lu(x)\,dx=
\int_{\R^{n-1}\times (0,1)}
\partial_nv(x)\,Lu(x)\,dx.
\end{equation*}
This and~\eqref{Lvsu} prove~\eqref{M:SX}, as desired.

Now we claim that
\begin{equation}\label{M:DX}
\lim_{\lambda\searrow0}\frac{I_\lambda(w,w^*)-I_0(w,w^*)}\lambda=
c\,\int_{\R^{n-1}\times\{0\}}\eta(x',0)\,\tau(x',0)\,dx',
\end{equation}
for some~$c>0$. To this end, we let
\begin{eqnarray*} \Psi(x,\lambda):=
w_{\lambda/2}(x)\,w^*_{-\lambda/2}(x)
= \sqrt{L}u \left(x+\frac\lambda2 e_n\right)\sqrt{L^*}v
\left(x-\frac\lambda2 e_n\right)
\end{eqnarray*}
and we notice that, if~$\lambda\in\left(-\frac1{10},\frac1{10}\right)$,
$$ \sup_{|x_n|\ge1}  \left|
\frac{\partial\Psi}{\partial\lambda}(x,\lambda)\right|+
\left|
\frac{\partial^2\Psi}{\partial\lambda^2}(x,\lambda)\right|\le C,$$
for some~$C>0$.

As a consequence,
if~$|x_n|\ge1$ and~$\lambda\in\left(-\frac1{10},\frac1{10}\right)$,
\begin{eqnarray*}&&
\Big|
w_{\lambda/2}(x)\,w^*_{-\lambda/2}(x)+
w_{-\lambda/2}(x)\,w^*_{\lambda/2}(x)-2
w(x)\,w^*(x)\Big|\\
&=&\Big|\Phi(x,\lambda)+\Phi(x,-\lambda)-2\Phi(x,0)\Big|\\
&=& \left| \int_0^{\lambda} \left(\frac{\partial\Psi}{\partial\lambda}(x,\theta)
- \frac{\partial\Psi}{\partial\lambda}(x,-\theta)\right)\,d\theta
\right|\\&\le&
\sup_{{|x_n|\ge1}\atop{|\theta|\le1/10}}  \left|
\frac{\partial^2\Psi}{\partial\lambda^2}(x,\theta)\right|\;\lambda^2\\
&\le& C\lambda^2.
\end{eqnarray*}
{F}rom this, it follows that, if~$|x_n|\ge1$,
\begin{equation}\label{mur9}
\lim_{\lambda\searrow0}\frac{w_{\lambda/2}(x)\,w^*_{-\lambda/2}(x)+
w_{-\lambda/2}(x)\,w^*_{\lambda/2}(x)-2
w(x)\,w^*(x)}{\lambda}=0.
\end{equation}
Furthermore, since, if~$|x_n|\ge1/2$, 
\begin{equation*}
\sum_{{\alpha\in\N^n}\atop{|\alpha|\le1}}|D^\alpha\sqrt{L}u(x)|+|D^\alpha\sqrt{L^*}v(x)|\le \frac{C}{|x|^{n+s}},\end{equation*}
up to renaming~$C$, we see that, for all~$|x_n|\ge1$ and~$|\lambda|\le\frac1{10}$,
\begin{eqnarray*}&&
|w_\lambda(x)w^*_{-\lambda}(x)-w(x)w^*(x)|
\\&\le& |w_\lambda(x)|\;|w^*_{-\lambda}(x)-w^*(x)|+ |w^*(x)|\;|w_{\lambda}(x)-w(x)|\\
&=& \big| \sqrt{L}u(x+\lambda e_n)\big|\;
\big| \sqrt{L^*}v(x-\lambda e_n)-\sqrt{L^*}v(x)\big|\\&&\qquad+
\big| \sqrt{L^*}v(x)\big|\;
\big| \sqrt{L}u(x+\lambda e_n)-\sqrt{L}u(x)\big|\\
&\le&\frac{C\lambda}{|x|^{2(n+s)}}.
\end{eqnarray*}
On this account,
we deduce from~\eqref{mur9} and the Dominated Convergence Theorem that
$$ \lim_{\lambda\searrow0}\int_{\R^{n-1}\times((-\infty,-1]\cup[1,+\infty))}
\frac{w_{\lambda/2}(x)\,w^*_{-\lambda/2}(x)+
w_{-\lambda/2}(x)\,w^*_{\lambda/2}(x)-2
w(x)\,w^*(x)}{\lambda}\,dx=0.$$
For this reason, we have that
\begin{equation}\label{mie2}
\lim_{\lambda\searrow0}\frac{I_\lambda(w,w^*)-I_0(w,w^*)}\lambda=
 \lim_{\lambda\searrow0}\int_{\R^{n-1}\times(-1,1)}
\frac{w_{\lambda/2}(x)\,w^*_{-\lambda/2}(x)+
w_{-\lambda/2}(x)\,w^*_{\lambda/2}(x)-2
w(x)\,w^*(x)}{\lambda}\,dx.
\end{equation}
Now we exploit Lemma~\ref{LE.31}. For this, we assume that~$\gamma<s$,
and therefore~$\gamma^*>s$ (the case~$\gamma>s$, which leads to~$\gamma^*<s$
being perfectly symmetric). Then, Lemma~\ref{LE.31} gives us that
\begin{equation}\label{56:293}
\begin{split}
&w(x)=\big(A_1 (x_n)_+^{\gamma-s}+A_2 (x_n)_-^{\gamma-s}\big)\eta(x',0)+h(x)
\\{\mbox{and }}\quad &w^*(x)=
\big(A^*_1 (x_n)_+^{\gamma^*-s}+A^*_2 (x_n)_-^{\gamma^*-s}\big)\tau(x',0)+h^*(x)
,\end{split}\end{equation}
with~$A_1$, $A_2$, $A_1^*$, $A_2^*\in\R$ and, if~$|x_n|\le2$,
\begin{equation}\label{SET2} |\partial_n h(x)|\le C|x_n|^{\beta-1}\qquad{\mbox{ and }}\qquad
|\partial_n h^*(x)|\le C|x_n|^{\beta^*-1}\end{equation}
for some~$C>0$, $\beta\in(0,1)$ and~$\beta^*\in( \gamma^*-s,1)$.
In addition, by~\eqref{BV33}, we know that
\begin{equation}\label{SET2N}
|h(x)|\le\frac{C |x_n|^{\gamma-s}}{1+|x|^{n+s}}\qquad{\mbox{ and }}\qquad
|h^*(x)|\le \frac{C }{1+|x|^{n+s}},\end{equation}
up to renaming~$C$.

Possibly replacing~$\beta$ with~$\min\left\{\beta,\frac{1-\gamma^*+s}{2(2-\gamma^*+s)}\right\}$
and~$\beta^*$ with~$\min\left\{\beta^*,\frac{1-\gamma^*+s}{2(2-\gamma^*+s)}+\frac{(3-\gamma^*+s)(\gamma^*-s)}{2(2-\gamma^*+s)}\right\}$,
we can also assume that
\begin{equation}\label{65930475}
\beta+\beta^*\le
\frac{1-\gamma^*+s}{2(2-\gamma^*+s)}+
\frac{1-\gamma^*+s}{2(2-\gamma^*+s)}+\frac{(3-\gamma^*+s)(\gamma^*-s)}{2(2-\gamma^*+s)}
=
\frac{1+\gamma^*-s}2
<1.
\end{equation}
As a consequence of~\eqref{SET2}, we have that
\begin{equation}\label{65:11}\begin{split}
W(x)\,&:=
w(x)w^*(x)\\&=
\big(A_1 (x_n)_+^{\gamma-s}+A_2 (x_n)_-^{\gamma-s}\big)
\big(A^*_1 (x_n)_+^{\gamma^*-s}+A^*_2 (x_n)_-^{\gamma^*-s}\big)
\eta(x',0)\tau(x',0)+{\mathcal{H}}(x)
\\&= B_0(x_n)\,\eta(x',0)\tau(x',0)+(x_n)_+^{\gamma-s}{\mathcal{G}}_1(x)+
(x_n)_-^{\gamma-s}{\mathcal{G}}_2(x)+{\mathcal{H}}(x),\end{split}\end{equation}
with
$$B_0(x_n):=A_1 A_1^*\,\chi_{(0,+\infty)}(x_n)+A_2 A_2^*\,\chi_{(-\infty,0)}(x_n)
,$$
and~${\mathcal{H}}$, ${\mathcal{G}}_1$ and~${\mathcal{G}}_2$ such that when~$|x_n|\le2$
we have that
\begin{equation}\label{FTR}
|\partial_n{\mathcal{H}}(x)|\le \frac{C\;|x_n|^{\bar\beta-1}}{1+|x|^{n+s}},\end{equation}
\begin{equation}\label{FTR2}
|{\mathcal{G}}_1(x)|+
|{\mathcal{G}}_2(x)|\le \frac{C}{1+|x|^{n+s}}\qquad{\mbox{ and }}\qquad
|\partial_n{\mathcal{G}}_1(x)|+
|\partial_n{\mathcal{G}}_2(x)|\le \frac{C\;|x_n|^{\bar\beta-1}}{1+|x|^{n+s}},\end{equation}
for some~$C$, $\bar\beta>0$. 

In light of~\eqref{FTR}, if~$G\subseteq\R^{n-1}\times(-1,1)$ and~$\lambda\in\left(0,\frac1{10}\right)$,
\begin{equation}\label{VIT}
\begin{split}&
\int_G \frac{|{\mathcal{H}}(x+\lambda e_n)-{\mathcal{H}}(x)|}\lambda\,dx
\le \int_0^1\left[\int_G |\partial_n {\mathcal{H}}(x+\lambda t e_n)|\,dx\right]\,dt\\&\qquad\le
\int_0^1\left[\int_G 
\frac{C\;|x_n+\lambda t |^{\bar\beta-1}}{1+|x+\lambda t e_n|^{n+s}}
\,dx\right]\,dt=
\int_0^1\left[\int_{G+\lambda t e_n} 
\frac{C\;|x_n |^{\bar\beta-1}}{1+|x|^{n+s}}
\,dx\right]\,dt,
\end{split}
\end{equation}
which produces an arbitrarily small result whenever either~$G$ has a sufficiently small
measure or it is contained in the complement of a sufficiently large ball.

Also, since, by construction, ${\mathcal{H}}\in C^1(\R^{n-1}\times(\R\setminus\{0\}))$,
we have that,
if~$x_n\in[-1,0)\cup(0,1]$,
\begin{equation}\label{65:00}
\begin{split}&
\lim_{\lambda\searrow0} \frac{ {\mathcal{H}}(x+\lambda e_n)+
{\mathcal{H}}(x-\lambda e_n)-2{\mathcal{H}}(x)}{\lambda}=
\lim_{\lambda\searrow0}
\partial_n {\mathcal{H}}(x+\lambda e_n)-\partial_n
{\mathcal{H}}(x-\lambda e_n)=0.
\end{split}
\end{equation}
Therefore, by~\eqref{VIT}, \eqref{65:00} and
the Vitali Convergence Theorem,
\begin{equation}\label{65:10}
\lim_{\lambda\searrow0} \int_{\R^{n-1}\times(-1,1)}
\frac{ {\mathcal{H}}(x+\lambda e_n)+
{\mathcal{H}}(x-\lambda e_n)-2{\mathcal{H}}(x)}{\lambda}\,dx=
0.
\end{equation}

Moreover, for small~$\lambda$,
\begin{eqnarray*}&&
\int_{-1}^{1} \frac{(x_n+\lambda)_+^{\gamma-s}+(x_n-\lambda)_+^{\gamma-s}
-2(x_n)_+^{\gamma-s}}{\lambda}\,dx_n =
\lambda^{\gamma-s}
\int_{-1/\lambda}^{1/\lambda} (t+1)_+^{\gamma-s}+(t-1)_+^{\gamma-s}
-2t_+^{\gamma-s}\,dt \\
&&\qquad=
\lambda^{\gamma-s}\left[
\int_{-1}^{1/\lambda} (t+1)^{\gamma-s}\,dt
+
\int_{1}^{1/\lambda} (t-1)^{\gamma-s}\,dt-
2\int_{0}^{1/\lambda} t^{\gamma-s}\,dt\right]
\\&&\qquad=
\lambda^{\gamma-s}\left[
\int_{0}^{\frac1\lambda+1} t^{\gamma-s}\,dt
+
\int_{0}^{\frac1\lambda-1} t^{\gamma-s}\,dt-
2\int_{0}^{1/\lambda} t^{\gamma-s}\,dt\right]\\
\\&&\qquad=
\lambda^{\gamma-s}\left[
\int_{\frac1\lambda-1}^{\frac1\lambda+1} t^{\gamma-s}\,dt
-
2\int_{\frac1\lambda-1}^{1/\lambda} t^{\gamma-s}\,dt\right]
\\&&\qquad=
\frac{\lambda^{\gamma-s}}{\gamma-s+1}\left[
\left( \frac1\lambda+1\right)^{\gamma-s+1}-
\left( \frac1\lambda-1\right)^{\gamma-s+1}
-2\left( \frac1\lambda\right)^{\gamma-s+1}+2\left( \frac1\lambda-1\right)^{\gamma-s+1}
\right]\\&&\qquad=\frac{1}{\lambda\,(\gamma-s+1)}
\left[
\left( 1+\lambda\right)^{\gamma-s+1}+
\left( 1-\lambda\right)^{\gamma-s+1}
-2
\right]\\&&\qquad=\frac{1}{\lambda\,(\gamma-s+1)}
\left[ 1+(\gamma-s+1)\lambda+O(\lambda^2)+
1-(\gamma-s+1)\lambda+O(\lambda^2)-2
\right]\\&&\qquad=O(\lambda),
\end{eqnarray*}
As a result, by~\eqref{FTR2},
\begin{equation}\label{65:30}
\lim_{\lambda\searrow0} \int_{\R^{n-1}\times(-1,1)}
\frac{(x_n+\lambda)_+^{\gamma-s}+(x_n-\lambda)_+^{\gamma-s}
-2(x_n)_+^{\gamma-s}}{\lambda}\,{\mathcal{G}}_1(x)\,dx=0.\end{equation}
Similarly,
\begin{equation}\label{65:40}
\lim_{\lambda\searrow0} \int_{\R^{n-1}\times(-1,1)}
\frac{(x_n+\lambda)_-^{\gamma-s}+(x_n-\lambda)_-^{\gamma-s}
-2(x_n)_-^{\gamma-s}}{\lambda}\,{\mathcal{G}}_2(x)\,dx=0.\end{equation}
Furthermore, if~$\lambda\in(0,1)$,
\begin{equation}\label{Quasi1}\begin{split}&
\int_{-1}^1 \frac{\chi_{(0,+\infty)}(x_n+\lambda)+\chi_{(0,+\infty)}(x_n-\lambda)-
2\chi_{(0,+\infty)}(x_n)}{\lambda}\,dx_n\\=\;&
\int_{-\lambda}^\lambda \big(\chi_{(0,+\infty)}(t+1)+\chi_{(0,+\infty)}(t-1)-
2\chi_{(0,+\infty)}(t)\big)\,dt\\=\;&
\int_{-\lambda}^0 \big(1+0-
0\big)\,dt+
\int_0^\lambda \big(1+0-
2\big)\,dt\\=\;&0,
\end{split}
\end{equation}
and, in the same way,
\begin{equation}\label{Quasi2}
\int_{-1}^1 \frac{\chi_{(-\infty,0)}(x_n+\lambda)+\chi_{(-\infty,0)}(x_n-\lambda)-
2\chi_{(-\infty,0)}(x_n)}{\lambda}\,dx_n=0.
\end{equation}
In light of~\eqref{65:11}, \eqref{65:10},
\eqref{65:30}, \eqref{65:40}, \eqref{Quasi1} and~\eqref{Quasi2}, we conclude that
\begin{equation}\label{65:22}
\lim_{\lambda\searrow0}
\int_{\R^{n-1}\times(-1,1)}\frac{W(x+\lambda e_n)+W(x-\lambda e_n)-2W(x)}{\lambda}
\,dx=0.\end{equation}

Now, we write that
\begin{eqnarray*}&&
w_{\lambda}(x)\,w^*_{-\lambda}(x)+
w_{-\lambda}(x)\,w^*_{\lambda}(x)-2
w(x)\,w^*(x)\\&=&
w(x+\lambda e_n)w^*(x-\lambda e_n)+
w(x-\lambda e_n)w^*(x+\lambda e_n)-2
w(x)\,w^*(x)\\
&=&
w(x+\lambda e_n)\Big( w^*(x-\lambda e_n)-
w^*(x+\lambda e_n)\Big)+
w(x-\lambda e_n)\Big(w^*(x+\lambda e_n)-
w^*(x-\lambda e_n)\Big)\\&&\quad+w(x+\lambda e_n)w^*(x+\lambda e_n)+
w(x-\lambda e_n)w^*(x-\lambda e_n)
-2
w(x)\,w^*(x)\\
&=&\Big(
w(x+\lambda e_n)-w(x-\lambda e_n)\Big)\Big( w^*(x-\lambda e_n)-
w^*(x+\lambda e_n)\Big)
\\&&\qquad+W(x+\lambda e_n)+W(x-\lambda e_n)-2W(x).
\end{eqnarray*}
On this account, recalling~\eqref{65:22}, we conclude that
\begin{equation}\label{PLAmsd}
\begin{split}&
\lim_{\lambda\searrow0}
\int_{\R^{n-1}\times(-1,1)}
\frac{w_{\lambda}(x)\,w^*_{-\lambda}(x)+
w_{-\lambda}(x)\,w^*_{\lambda}(x)-2
w(x)\,w^*(x)}{\lambda}\,dx\\=\;&
\lim_{\lambda\searrow0}\int_{\R^{n-1}\times(-1,1)}\frac{\big(
w(x+\lambda e_n)-w(x-\lambda e_n)\big)\big( w^*(x-\lambda e_n)-
w^*(x+\lambda e_n)\big)}\lambda\,dx.
\end{split}
\end{equation}
We also remark that, by~\eqref{56:293},
\begin{equation}\label{SET}
\begin{split}&
w(x+\lambda e_n)-w(x-\lambda e_n)\\&\quad=
\Big(A_1 \big(
(x_n+\lambda)_+^{\gamma-s}-(x_n-\lambda)_+^{\gamma-s}\big)
+A_2 \big( (x_n+\lambda)_-^{\gamma-s}-(x_n-\lambda)_-^{\gamma-s}\big)\Big)\eta(x',0)\\
&\quad\qquad+h(x+\lambda e_n)-h(x-\lambda e_n)
\\{\mbox{and }}\quad &w^*(x-\lambda e_n)-
w^*(x+\lambda e_n)\\&\quad=\Big(
A^*_1\big( (x_n-\lambda)_+^{\gamma^*-s}-
(x_n+\lambda)_+^{\gamma^*-s}\big)+A^*_2 \big(
(x_n-\lambda)_-^{\gamma^*-s}- (x_n+\lambda)_-^{\gamma^*-s}\big)\Big)\tau(x',0)\\&\quad\qquad+h^*(x-\lambda e_n)-h^*(x+\lambda e_n).
\end{split}\end{equation}
In addition,
\begin{eqnarray*}&&\lim_{\lambda\searrow0}
\int_{-1}^1
\frac{
\big(
(x_n+\lambda)_+^{\gamma-s}-(x_n-\lambda)_+^{\gamma-s}\big)
\big(
(x_n-\lambda)_+^{\gamma^*-s}-(x_n+\lambda)_+^{\gamma^*-s}\big)}\lambda\,dx_n\\&=&\lim_{\lambda\searrow0}
\int_{-1/\lambda}^{1/\lambda}
\big(
(t+1)_+^{\gamma-s}-(t-1)_+^{\gamma-s}\big)
\big(
(t-1)_+^{\gamma^*-s}-(t+1)_+^{\gamma^*-s}\big)\,dt\\&=&
\int_{\R}
\big(
(t-1)_+^{\gamma-s}-(t+1)_+^{\gamma-s}\big)
\big(
(t+1)_+^{\gamma^*-s}-(t-1)_+^{\gamma^*-s}\big)\,dt\\&=&
\int_1^{+\infty}\left[
\left(\frac{t-1}{t+1}\right)^{\gamma-s}+\left(\frac{t+1}{t-1}\right)^{\gamma-s}-2
\right]\,dt-2\\
&=&2c_1,
\end{eqnarray*}
where
\begin{equation}\label{RMS1} c_1:=\frac{\pi(s-\gamma)}{\tan\big(\pi(\gamma-s)\big)}
,\end{equation} and, similarly,
\begin{eqnarray*}&&\lim_{\lambda\searrow0}
\int_{-1}^1
\frac{
\big(
(x_n+\lambda)_+^{\gamma-s}-(x_n-\lambda)_+^{\gamma-s}\big)
\big(
(x_n-\lambda)_-^{\gamma^*-s}-(x_n+\lambda)_-^{\gamma^*-s}\big)}\lambda\,dx_n\\&=&
\int_{\R}
\big(
(t+1)_+^{\gamma-s}-(t-1)_+^{\gamma-s}\big)
\big(
(t-1)_-^{\gamma^*-s}-(t+1)_-^{\gamma^*-s}\big)\,dt\\&=&
\int_{-1}^1
\big(
(t-1)_+^{\gamma-s}-(t+1)_+^{\gamma-s}\big)
\big(
(t+1)_-^{\gamma^*-s}-(t-1)_-^{\gamma^*-s}\big)\,dt\\&=&
\int_{-1}^1 (t+1)^{\gamma-s}(1-t)^{s-\gamma}\,dt\\&=&
2c_2,\end{eqnarray*}
where
\begin{equation}\label{RMS2} c_2:=\Gamma(\gamma-s+1)\,\Gamma(s-\gamma+1).\end{equation}
Moreover, recalling~\eqref{SET2}, if~$2\le|t|\le\frac1\lambda$
and~$\lambda\in\left(0,\frac1{10}\right)$,
\begin{eqnarray*}&&
\big| h^*(x', \lambda t-\lambda )-h^*(x',\lambda t+\lambda)\big|
\le \lambda\sup_{{x'\in\R^{n-1}}\atop{|\tau|\le\lambda}}
|\partial_n h^*(x', \lambda t+\tau )|\\&&\qquad\le
C \lambda\sup_{{|\tau|\le\lambda}}
|\lambda t+\tau |^{\beta^*-1}=
C \lambda^{\beta^*}\sup_{{|\theta|\le1}}
| t+\theta|^{\beta^*-1}\le C \lambda^{\beta^*}
| t|^{\beta^*-1}
\end{eqnarray*}
and, if~$|t|\le2$,
\begin{eqnarray*}&&
\big| h^*(x', \lambda t-\lambda )-h^*(x',\lambda t+\lambda)\big|
\le C\int_{-\lambda}^\lambda |\lambda t+\tau|^{\beta^*-1}\,d\tau
= C\lambda^{\beta^*}\int_{-1}^1 | t+\theta|^{\beta^*-1}\,d\theta\le
C\lambda^{\beta^*},
\end{eqnarray*}
and accordingly, for all~$|t|\le\frac1\lambda$,
\begin{equation}\label{657959} \big| h^*(x', \lambda t-\lambda )-h^*(x',\lambda t+\lambda)\big|
\le\frac{C\lambda^{\beta^*}}{1+|t|^{1-\beta^*}},\end{equation}
up to renaming~$C$ line after line.

As a consequence,
\begin{eqnarray*}&&\lim_{\lambda\searrow0}\left|
\int_{-1}^1
\frac{
\big(
(x_n+\lambda)_+^{\gamma-s}-(x_n-\lambda)_+^{\gamma-s}\big)
\big( h^*(x-\lambda e_n)-h^*(x+\lambda e_n)\big)}\lambda\,dx_n\right|\\&=&
\lim_{\lambda\searrow0}\lambda^{\gamma-s}\left|
\int_{-1/\lambda}^{1/\lambda}
\big(
(t+1)_+^{\gamma-s}-(t-1)_+^{\gamma-s}\big)
\big( h^*(x', \lambda t-\lambda )-h^*(x',\lambda t+\lambda)\big)\,dt\right|
\\&\le&
C\lim_{\lambda\searrow0}\lambda^{\beta^*+\gamma-s}
\int_{-1/\lambda}^{1/\lambda}\frac{
\big|
(t+1)_+^{\gamma-s}-(t-1)_+^{\gamma-s}\big|}{1+|t|^{1-\beta*}}
\,dt\\ &\le&
C\lim_{\lambda\searrow0}\lambda^{\beta^*+\gamma-s}
\left(1+
\int_{1}^{+\infty}\frac{dt}{t^{2+s-\gamma-\beta*}}\right)
\\ &\le&
C\lim_{\lambda\searrow0}\lambda^{\beta^*+\gamma-s}\\
&=&0.
\end{eqnarray*}
Similarly,
$$
\lim_{\lambda\searrow0}\left|
\int_{-1}^1
\frac{
\big(
(x_n+\lambda)_+^{\gamma^*-s}-(x_n-\lambda)_+^{\gamma^*-s}\big)
\big( h(x-\lambda e_n)-h(x+\lambda e_n)\big)}\lambda\,dx_n\right|=0.$$
Moreover, recalling~\eqref{65930475}
and~\eqref{657959} (as well as a similar estimate for~$h$),
\begin{eqnarray*}&&\lim_{\lambda\searrow0}\left|
\int_{-1}^1
\frac{\big( h(x-\lambda e_n)-h(x+\lambda e_n)\big)
\big( h^*(x-\lambda e_n)-h^*(x+\lambda e_n)\big)}\lambda\,dx_n\right|\\&=&
\lim_{\lambda\searrow0}\left|
\int_{-1/\lambda}^{1/\lambda}
\frac{\big( h(x',\lambda t-\lambda )-h(x',\lambda t+\lambda )\big)
\big( h^*(x',\lambda t-\lambda)-h^*(x',\lambda t+\lambda )\big)}\lambda\,dx_n\right|\\&\le&
\lim_{\lambda\searrow0}
C\lambda^{\beta+\beta^*}
\int_\R
\frac{dt}{(1+|t|^{1-\beta})(1+|t|^{1-\beta^*})}\\
&\le&
\lim_{\lambda\searrow0}
C\lambda^{\beta+\beta^*}\left(1+
\int_\R
\frac{dt}{1+t^{2-(\beta+\beta^*)}}\right)\\&=&
\lim_{\lambda\searrow0}
C\lambda^{\beta+\beta^*}=0.
\end{eqnarray*}
{F}rom these observations and~\eqref{SET}, 
we find that
\begin{eqnarray*}&& \lim_{\lambda\searrow0}\int_{\R^{n-1}\times(-1,1)}\frac{\big(
w(x+\lambda e_n)-w(x-\lambda e_n)\big)\big( w^*(x-\lambda e_n)-
w^*(x+\lambda e_n)\big)}\lambda\,dx\\&=&
2c\,\int_{\R^{n-1}\times\{0\}}\eta(x',0)\,\tau(x',0)\,dx'
,\end{eqnarray*}
with
\begin{equation}\label{VAL:cc}
c:= c_1\big(A_1 A_1^*+A_2A_2^*\big)+c_2\big(A_1 A_2^*+A_2A_1^*\big)
.\end{equation}
On this account, recalling~\eqref{56:293} and~\eqref{PLAmsd},
we obtain that \begin{eqnarray*}&&
\lim_{\lambda\searrow0}
\int_{\R^{n-1}\times(-1,1)}
\frac{w_{\lambda}(x)\,w^*_{-\lambda}(x)+
w_{-\lambda}(x)\,w^*_{\lambda}(x)-2
w(x)\,w^*(x)}{\lambda}\,dx\\&=&2c\,\int_{\R^{n-1}\times\{0\}}\eta(x',0)\,\tau(x',0)\,dx'.\end{eqnarray*}
As a result, recalling~\eqref{mie2}, we conclude that~\eqref{M:DX}
holds true.

Then, the desired claim in~\eqref{M:X} now follows by combining~\eqref{LAE2},
\eqref{M:SX} and~\eqref{M:DX}.

Now we prove~\eqref{VAL:c}. For this, we recall~\eqref{VALUE A},
and thus we write that
\begin{eqnarray*}
A_1&=& \frac{1}{2\cos\left(\frac{\pi s}{2}\right) }\,
\left( \frac{\Gamma(s-\gamma)}{\Gamma(-\gamma)}
-\frac{\Gamma(1+\gamma)}{\Gamma(1-s+\gamma)}\right)
\left[
\tan\left(\pi\left(\gamma-\frac{s}{2}\right)\right)
\mathcal{A}_\sharp(e_n)-\mathcal{B}_\sharp(e_n)\right]\\
&=& k_1\,\mathcal{A}_\sharp(e_n)+k_2\,\mathcal{B}_\sharp(e_n)\end{eqnarray*}
and
\begin{eqnarray*}
A_2&=& -\frac{s\,\Gamma(1+\gamma)\,\Gamma(s-\gamma)}{2\cos\left(\frac{\pi s}{2}\right)\,\Gamma(1-s)\,\Gamma(1+s)}\,\left[
\tan\left(\frac{\pi s}{2}\right)
\mathcal{A}_\sharp(e_n)-\mathcal{B}_\sharp(e_n)\right]
\\
&=& p_1\,\mathcal{A}_\sharp(e_n)+p_2\,\mathcal{B}_\sharp(e_n),\end{eqnarray*}
with
\begin{eqnarray*}
&& k_1:=
\frac{1}{2\cos\left(\frac{\pi s}{2}\right) }\,
\left( \frac{\Gamma(s-\gamma)}{\Gamma(-\gamma)}
-\frac{\Gamma(1+\gamma)}{\Gamma(1-s+\gamma)}\right)\,
\tan\left(\pi\left(\gamma-\frac{s}{2}\right)\right)
,\\
&& k_2:=-
\frac{1}{2\cos\left(\frac{\pi s}{2}\right) }\,
\left( \frac{\Gamma(s-\gamma)}{\Gamma(-\gamma)}
-\frac{\Gamma(1+\gamma)}{\Gamma(1-s+\gamma)}\right)
,\\
&& p_1:=
-\frac{s\,\Gamma(1+\gamma)\,\Gamma(s-\gamma)}{2\cos\left(\frac{\pi s}{2}\right) \,\Gamma(1-s)\,\Gamma(1+s)}\,
\tan\left(\frac{\pi s}{2}\right)
\\{\mbox{and }}&&p_2:=
\frac{s\,\Gamma(1+\gamma)\,\Gamma(s-\gamma)}{2\cos\left(\frac{\pi s}{2}\right) \,\Gamma(1-s)\,\Gamma(1+s)}.\end{eqnarray*}
As for~$A_1^*$ and~$A_2^*$, we recall that they
are also obtained from~\eqref{VALUE A},
but in this case the operator~$\sqrt{L}$ is replaced by~$\sqrt{L^*}$
and the function~$u$ by the function~$v$ (recall~\eqref{54398703}).

Hence, since the dual operation exchanges the sign of~$\mathcal{B}_\sharp$
(recall~\eqref{CA}) and the function replacement
exchanges~$\gamma$ with~$\gamma^*=2s-\gamma$ (recall~\eqref{uev-9731}),
we can obtain the values of~$A_1^*$ and~$A_2^*$
from those of~$A_1$ and~$A_2$, with these two structural changes.
In this way, we see that
\begin{eqnarray*}
A_1^*=k_1^*\,\mathcal{A}_\sharp(e_n)+k_2^*\,\mathcal{B}_\sharp(e_n)\qquad
{\mbox{and}}\qquad A_2^*=
p_1^*\,\mathcal{A}_\sharp(e_n)+p_2^*\,\mathcal{B}_\sharp(e_n),
\end{eqnarray*}
with
\begin{eqnarray*}
&& k_1^*:=
\frac{1}{2\cos\left(\frac{\pi s}{2}\right) }\,
\left( \frac{\Gamma(\gamma-s)}{\Gamma(\gamma-2s)}
-\frac{\Gamma(1+2s-\gamma)}{\Gamma(1+s-\gamma)}\right)\,
\tan\left(\pi\left(\frac{3s}{2}-\gamma\right)\right)
,\\
&& k_2^*:=
\frac{1}{2\cos\left(\frac{\pi s}{2}\right) }\,
\left( \frac{\Gamma(\gamma-s)}{\Gamma(\gamma-2s)}
-\frac{\Gamma(1+2s-\gamma)}{\Gamma(1+s-\gamma)}\right)
,\\
&& p_1^*:=
-\frac{s\,\Gamma(1+2s-\gamma)\,\Gamma(\gamma-s)}{2\cos\left(\frac{\pi s}{2}\right)\,\Gamma(1-s)\,\Gamma(1+s)}\,
\tan\left(\frac{\pi s}{2}\right)
\\{\mbox{and }}&&p_2^*:=-
\frac{s\,\Gamma(1+2s-\gamma)\,\Gamma(\gamma-s)}{2\cos\left(\frac{\pi s}{2}\right)\,\Gamma(1-s)\,\Gamma(1+s)}.\end{eqnarray*}
It is now convenient to look at suitable ratios of these quantities
which only deal with trigonometric functions, removing any explicit
dependence on the Euler's Gamma function. For this, 
using the ``trigonometric sum-to-product identity'',
we observe that
\begin{eqnarray*}
&&\frac{k_1}{k_2}=-\tan\left(\pi\left(\gamma-\frac{s}{2}\right)\right),
\\&&\frac{k_1^*}{k_2^*}=\tan\left(\pi\left(\frac{3s}{2}-\gamma\right)\right),
\\&&\frac{p_1}{p_2}=-\tan\left(\frac{\pi s}{2}\right),
\\&&\frac{p_1^*}{p_2^*}=\tan\left(\frac{\pi s}{2}\right),
\\&&\frac{k_1}{p_1}=-
\frac{\Gamma(1-s)\,\Gamma(1+s)\,\big(\Gamma(s-\gamma)\,\Gamma(1-s+\gamma)
-\Gamma(-\gamma)\,\Gamma(1+\gamma)
\big)
}{s\,\Gamma(1+\gamma)\,\Gamma(s-\gamma)\,\Gamma(-\gamma)\,\Gamma(1-s+\gamma)}\times
\frac{ \sin(\pi\gamma)-\sin(\pi(s-\gamma))}{
\sin(\pi\gamma)+\sin(\pi (s-\gamma)) }\\
&&\qquad\qquad=-\frac{\csc(\pi s) \big(
\csc(\pi \gamma) -\csc(\pi (\gamma - s))\big)}{
\csc(\pi\gamma) \csc(\pi (\gamma - s))
}
\times\frac{ \sin(\pi\gamma)-\sin(\pi(s-\gamma))}{
\sin(\pi\gamma)+\sin(\pi (s-\gamma)) }\\&&\qquad\qquad=
\frac{\sin(\pi\gamma)+\sin(\pi(\gamma-s))}{\sin(\pi s)}\\
&&\qquad\qquad=
\frac{\sin\left(\pi\left(\gamma-\frac{ s}{2}\right)\right)
}{\sin\left(\frac{\pi s}{2}\right)}
,
\\&&\frac{k_1^*}{p_1^*}=\frac{\sin\left(\pi\left(\frac{3 s}{2}-\gamma\right)\right)
}{\sin\left(\frac{\pi s}{2}\right)}
\\&&\frac{k_2}{p_1}=
\frac{k_2}{k_1}\times\frac{k_1}{p_1}=
-\frac{1}{\tan\left(\pi\left(\gamma-\frac{s}{2}\right)\right)}\times
\frac{\sin\left(\pi\left(\gamma-\frac{ s}{2}\right)\right)
}{\sin\left(\frac{\pi s}{2}\right)}=
-\frac{\cos\left(\pi\left(\gamma-\frac{ s}{2}\right)\right)
}{\sin\left(\frac{\pi s}{2}\right)}
\\{\mbox{and }}&&\frac{k_2^*}{p_1^*}=
\frac{\cos\left(\pi\left(\frac{3 s}{2}-\gamma\right)\right)
}{\sin\left(\frac{\pi s}{2}\right)}
.\end{eqnarray*}
Moreover, by~\eqref{RMS1} and~\eqref{RMS2},
$$\frac{c_1}{c_2}=\frac{\pi(s-\gamma)}{
\Gamma(\gamma-s+1)\,\Gamma(s-\gamma+1)
\tan\big(\pi(\gamma-s)\big)}=-\cos(\pi(\gamma-s)).$$
We also define
\begin{equation} \label{do-def}
d_0:= c_2\,p_1\,p_1^*.
\end{equation}
In this setting~\eqref{VAL:cc}, becomes
\begin{equation}\label{PL78:10}
\begin{split}
c\,&= c_1\Big(
\big( k_1\,\mathcal{A}_\sharp(e_n)+k_2\,\mathcal{B}_\sharp(e_n)\big)
\big( k_1^*\,\mathcal{A}_\sharp(e_n)+k_2^*\,\mathcal{B}_\sharp(e_n)\big)
+
\big( p_1\,\mathcal{A}_\sharp(e_n)+p_2\,\mathcal{B}_\sharp(e_n)\big)
\big( p_1^*\,\mathcal{A}_\sharp(e_n)+p_2^*\,\mathcal{B}_\sharp(e_n)\big)\Big)\\&
+c_2\Big(
\big( k_1\,\mathcal{A}_\sharp(e_n)+k_2\,\mathcal{B}_\sharp(e_n)\big)
\big( p_1^*\,\mathcal{A}_\sharp(e_n)+p_2^*\,\mathcal{B}_\sharp(e_n)\big)+\big( p_1\,\mathcal{A}_\sharp(e_n)+p_2\,\mathcal{B}_\sharp(e_n)\big) \big( k_1^*\,\mathcal{A}_\sharp(e_n)+k_2^*\,\mathcal{B}_\sharp(e_n)\big)\Big)\\
&= \alpha \big( \mathcal{A}_\sharp(e_n)\big)^2+
\beta \big( \mathcal{B}_\sharp(e_n)\big)^2+\delta\,
\mathcal{A}_\sharp(e_n)\,\mathcal{B}_\sharp(e_n),
\end{split}\end{equation}
where
\begin{eqnarray*}
&&\alpha:= c_1(k_1\,k_1^*+p_1\,p_1^*)+
c_2(k_1\,p_1^*+k_1^*\,p_1^*),\\
&&\beta:= c_1(k_2\,k_2^*+p_2\,p_2^*)+
c_2(k_2\,p_2^*+k_2^*\,p_2 )\\
{\mbox{and }}
&&\delta:= c_1(k_1\,k_2^*+k_1^*\,k_2+p_1\,p_2^*+p_1^*\,p_2)+
c_2(k_1\,p_2^*+k_1^*\,p_2+p_1\,k_2^*+p_1^*\,k_2).
\end{eqnarray*}
Using this notation, and recalling the trigonometric identity
$$ \sin(a-b)\,\sin(3b-a)-\sin^2 b=-\sin^2(a-2b),$$
to be exploited here with~$a:=\pi\gamma$ and~$b:=\frac{\pi s}2$,
we see that
\begin{eqnarray*}
\frac{\alpha}{d_0}&=&
\frac{c_1}{c_2}\left(\frac{k_1}{p_1}\times\frac{k_1^*}{p_1^*}+1\right)+
\left(\frac{k_1}{p_1}+\frac{k_1^*}{p_1^*}\right)\\&=&
-\cos(\pi(\gamma-s))\left(
\frac{\sin\left(\pi\left(\gamma-\frac{ s}{2}\right)\right)
}{\sin\left(\frac{\pi s}{2}\right)}
\times\frac{\sin\left(\pi\left(\frac{3 s}{2}-\gamma\right)\right)
}{\sin\left(\frac{\pi s}{2}\right)}+1
\right)\\&&\quad+
\frac{\sin\left(\pi\left(\gamma-\frac{ s}{2}\right)\right)
}{\sin\left(\frac{\pi s}{2}\right)}
+\frac{\sin\left(\pi\left(\frac{3 s}{2}-\gamma\right)\right)
}{\sin\left(\frac{\pi s}{2}\right)}
\\
&=&-\frac{\cos(\pi(\gamma-s))}{\sin^2\left(\frac{\pi s}{2}\right) }\left(
\sin\left(\pi\left(\gamma-\frac{ s}{2}\right)\right)
\sin\left(\pi\left(\frac{3 s}{2}-\gamma\right)\right)
+\sin^2\left(\frac{\pi s}{2}\right)
\right)\\&&\quad+ 2\cos\left(\pi (\gamma-s)\right)\\&=&
-\frac{\cos(\pi(\gamma-s))}{\sin^2\left(\frac{\pi s}{2}\right) }\left(
\sin\left(\pi\left(\gamma-\frac{ s}{2}\right)\right)
\sin\left(\pi\left(\frac{3 s}{2}-\gamma\right)\right)
-\sin^2\left(\frac{\pi s}{2}\right)
\right)\\
&=&\frac{\cos(\pi(\gamma-s))\sin^2(\pi(\gamma-s))}{\sin^2\left(\frac{\pi s}{2}\right) }
.\end{eqnarray*}
In addition, utilizing that
$$ \cos(a-b)\cos(3b-a)=\cos^2 b
-\sin^2(a -2 b),$$
we obtain that
\begin{eqnarray*}
\frac{\beta}{d_0}&=&\frac{c_1}{c_2}\left(\frac{k_2}{p_1}\times\frac{
k_2^*}{p_1^*}+\frac{p_2}{p_1}\times\frac{p_2^*}{p_1*}\right)+
\left(\frac{k_2}{p_1}\times\frac{p_2^*}{p_1^*}+\frac{k_2^*}{p_1^*}\times\frac{p_2}{p_1}\right)\\
&=&- \cos(\pi(\gamma-s)) 
\left( 
-\frac{\cos\left(\pi\left(\gamma-\frac{ s}{2}\right)\right)
}{\sin\left(\frac{\pi s}{2}\right)}\times
\frac{\cos\left(\pi\left(\frac{3 s}{2}-\gamma\right)\right)
}{\sin\left(\frac{\pi s}{2}\right)}
-\frac1{\tan\left(\frac{\pi s}{2}\right)}\times\frac1{\tan\left(\frac{\pi s}{2}\right)}\right)\\&&\quad+\left(
-\frac{\cos\left(\pi\left(\gamma-\frac{ s}{2}\right)\right)
}{\sin\left(\frac{\pi s}{2}\right)}\times
\frac1{\tan\left(\frac{\pi s}{2}\right)}-
\frac{\cos\left(\pi\left(\frac{3 s}{2}-\gamma\right)\right)
}{\sin\left(\frac{\pi s}{2}\right)}\times
\frac1{\tan\left(\frac{\pi s}{2}\right)}\right)\\
&=&
\frac{\cos(\pi(\gamma-s)) }{\sin^2\left(\frac{\pi s}{2}\right)}
\left( 
\cos\left(\pi\left(\gamma-\frac{ s}{2}\right)\right)
\cos\left(\pi\left(\frac{3 s}{2}-\gamma\right)\right)
+\cos^2\left(\frac{\pi s}{2}\right)\right)\\&&\quad-
\frac{\cos\left(\frac{\pi s}{2}\right) }{\sin^2\left(\frac{\pi s}{2}\right)}
\left(
\cos\left(\pi\left(\gamma-\frac{ s}{2}\right)\right)
+
\cos\left(\pi\left(\frac{3 s}{2}-\gamma\right)\right)
\right)\\
&=&
\frac{\cos(\pi(\gamma-s)) }{\sin^2\left(\frac{\pi s}{2}\right)}
\left( 
2\cos^2\left(\frac{\pi s}{2}\right)-\sin^2(\pi(\gamma-s))\right)-
\frac{2\cos^2\left(\frac{\pi s}{2}\right) \cos(\pi(\gamma-s))}{\sin^2\left(\frac{\pi s}{2}\right)}
\\
&=&-\frac{\cos(\pi(\gamma-s))\sin^2(\pi(\gamma-s))}{\sin^2\left(\frac{\pi s}{2}\right) }
.\end{eqnarray*}
We can already observe the nice structural property that~$\frac{\alpha}{d_0}=-\frac{\beta}{d_0}$.
Also, we observe that~$\frac{p_1^*}{p_2^*}=-\frac{p_1}{p_2}$, and 
then
\begin{eqnarray*}&&
\frac{k_1}{p_1}\times\frac{k_2^*}{p_1^*}+\frac{k_1^*}{p_1^*}\times\frac{k_2}{p_1}+
\frac{p_2^*}{p_1^*}+\frac{p_2}{p_1}\\&=&
\frac{k_1}{p_1}\times\frac{k_2^*}{p_1^*}+\frac{k_1^*}{p_1^*}\times\frac{k_2}{p_1}\\&=&
\frac{\sin\left(\pi\left(\gamma-\frac{ s}{2}\right)\right)
}{\sin\left(\frac{\pi s}{2}\right)}
\times
\frac{\cos\left(\pi\left(\frac{3 s}{2}-\gamma\right)\right)
}{\sin\left(\frac{\pi s}{2}\right)}-\frac{\sin\left(\pi\left(\frac{3 s}{2}-\gamma\right)\right)
}{\sin\left(\frac{\pi s}{2}\right)}\times\frac{\cos\left(\pi\left(\gamma-\frac{ s}{2}\right)\right)
}{\sin\left(\frac{\pi s}{2}\right)}
\\&=&\frac{\sin(2 \pi(\gamma -s))}{\sin^2\left(\frac{\pi s}{2}\right)}\\&=&
\frac{2\sin(\pi(\gamma -s))\cos(\pi(\gamma -s))}{\sin^2\left(\frac{\pi s}{2}\right)}
.\end{eqnarray*}
Moreover,
\begin{eqnarray*}
&&\frac{k_1}{p_1}\times\frac{p_2^*}{p_1^*}+\frac{k_1^*}{p_1^*}\times\frac{p_2}{p_1}+
\frac{k_2^*}{p_1^*}+\frac{k_2}{p_1}\\&=&
\frac{\sin\left(\pi\left(\gamma-\frac{ s}{2}\right)\right)
}{\sin\left(\frac{\pi s}{2}\right)} \times \frac1{\tan\left(\frac{\pi s}{2}\right)}-
\frac{\sin\left(\pi\left(\frac{3 s}{2}-\gamma\right)\right)
}{\sin\left(\frac{\pi s}{2}\right)}\times \frac1{\tan\left(\frac{\pi s}{2}\right)}\\&&\quad+
\frac{\cos\left(\pi\left(\frac{3 s}{2}-\gamma\right)\right)
}{\sin\left(\frac{\pi s}{2}\right)}-\frac{\cos\left(\pi\left(\gamma-\frac{ s}{2}\right)\right)
}{\sin\left(\frac{\pi s}{2}\right)}\\&=&\frac{
2 \cos\left(\frac{\pi s}{2}\right) \sin(\pi(\gamma-s))}{\sin\left(\frac{\pi s}{2}\right)\tan\left(\frac{\pi s}{2}\right)}+
2 \sin(\pi(\gamma-s))\\&=&\frac{2 \sin(\pi(\gamma-s))}{
\sin^2\left(\frac{\pi s}{2}\right)
}.
\end{eqnarray*}
Therefore,
\begin{eqnarray*}
\frac{\delta}{d_0}&=&\frac{
c_1}{c_2}\left(\frac{k_1}{p_1}\times\frac{k_2^*}{p_1^*}+\frac{k_1^*}{p_1^*}\times\frac{k_2}{p_1}+
\frac{p_2^*}{p_1^*}+\frac{p_2}{p_1}\right)+
\left(\frac{k_1}{p_1}\times\frac{p_2^*}{p_1^*}+\frac{k_1^*}{p_1^*}\times\frac{p_2}{p_1}+
\frac{k_2^*}{p_1^*}+\frac{k_2}{p_1}\right)
\\&=&
-\cos(\pi(\gamma-s))\times
\frac{2\sin(\pi(\gamma -s))\cos(\pi(\gamma -s))}{\sin^2\left(\frac{\pi s}{2}\right)}+
\frac{2 \sin(\pi(\gamma-s))}{
\sin^2\left(\frac{\pi s}{2}\right)
}\\&=&
\frac{2\sin^3(\pi(\gamma -s))}{\sin^2\left(\frac{\pi s}{2}\right)}
.\end{eqnarray*}
We plug these pieces of information into~\eqref{PL78:10}, and we find that
\begin{equation}\label{Lij810} \begin{split}\frac{c}{d_0}
\,&=\frac{\alpha}{d_0} \big( \mathcal{A}_\sharp(e_n)\big)^2+
\frac{\beta}{d_0} \big( \mathcal{B}_\sharp(e_n)\big)^2+\frac{\delta}{d_0}\,
\mathcal{A}_\sharp(e_n)\,\mathcal{B}_\sharp(e_n)\\&=
-\frac{\cos(\pi(\gamma-s)) \sin^2(\pi(\gamma-s))}{
\sin^2\left(\frac{\pi s}{2}\right)
}\Big( \big( \mathcal{A}_\sharp(e_n)\big)^2-\big( \mathcal{B}_\sharp(e_n)\big)^2\Big)+\frac{2\sin^3(\pi(\gamma -s))}{\sin^2\left(\frac{\pi s}{2}\right)}\,
\mathcal{A}_\sharp(e_n)\,\mathcal{B}_\sharp(e_n)\\&=
-\frac{\cos(\pi(\gamma-s)) \sin^2(\pi(\gamma-s))}{
\sin^2\left(\frac{\pi s}{2}\right)}\left(
\big( \mathcal{A}_\sharp(e_n)\big)^2-\big( \mathcal{B}_\sharp(e_n)\big)^2-2
\tan(\pi(\gamma -s))\,
\mathcal{A}_\sharp(e_n)\,\mathcal{B}_\sharp(e_n)\right).\end{split}\end{equation}
Now, recalling~\eqref{do-def}, we see that
\begin{eqnarray*} d_0&=& \Gamma(\gamma-s+1)\,\Gamma(s-\gamma+1)\\
&&\quad\times
\frac{s\,\Gamma(1+\gamma)\,\Gamma(s-\gamma)}{2\cos\left(\frac{\pi s}{2}\right) \,\Gamma(1-s)\,\Gamma(1+s)}\,
\tan\left(\frac{\pi s}{2}\right)
\times\frac{s\,\Gamma(1+2s-\gamma)\,\Gamma(\gamma-s)}{2\cos\left(\frac{\pi s}{2}\right)\,\Gamma(1-s)\,\Gamma(1+s)}\,
\tan\left(\frac{\pi s}{2}\right)
\\&=&-\frac{
\Gamma(\gamma + 1)\, \Gamma(-\gamma + 2 s + 1)\,\sin^2(\pi s)
}{4\,\cos^2\left(\frac{\pi s}{2}\right)
\sin^2(\pi (\gamma- s))
}\tan^2\left(\frac{\pi s}{2}\right)\\&=&
-\frac{
\Gamma(\gamma + 1)\, \Gamma(-\gamma + 2 s + 1)\,
\sin^2\left(\frac{\pi s}{2}\right)
}{ 
\sin^2(\pi (\gamma- s))
}
.\end{eqnarray*}
In light of this and~\eqref{Lij810}, we conclude that
\begin{equation}\label{cla-c-1}
c=c(\gamma,s)\,
\Big( \big( \mathcal{A}_\sharp(e_n)\big)^2-\big( \mathcal{B}_\sharp(e_n)\big)^2-2
\tan(\pi(\gamma -s))\,
\mathcal{A}_\sharp(e_n)\,\mathcal{B}_\sharp(e_n)\Big).
\end{equation}
with
$$ c(\gamma,s):=
\Gamma(\gamma + 1)\, \Gamma(-\gamma + 2 s + 1)\,
\cos(\pi(\gamma-s)).$$
{F}rom~\eqref{CA}, we know that
$$ \big(\mathcal{A}_\sharp(e_n)\big)^2-\big( \mathcal{B}_\sharp(e_n)\big)^2=
\mathcal{A}(e_n)\qquad{\mbox{and}}\qquad
2\,\mathcal{A}_\sharp(e_n)\, \mathcal{B}_\sharp(e_n)={
\mathcal{B}(e_n)},$$
where~$\mathcal{A}+i\mathcal{B}$ is the Fourier symbol of~$L$
according to Lemma~\ref{AeBFOU},
and hence~\eqref{cla-c-1} gives that
$$ c=\Gamma(\gamma + 1)\, \Gamma(\gamma^* + 1)\,\Big(\cos(\pi(\gamma -s)\mathcal{A} (e_n)-\sin(\pi(\gamma -s))\,
\mathcal{B}(e_n)\Big).$$
{F}rom this, \eqref{GAMMAGIUST}, and the trigonometric 
identities
$$ \cos(\arctan\theta)=\frac{1}{\sqrt{\theta^2 + 1}}\qquad{\mbox{ and }}\qquad
\sin(\arctan\theta)=\frac{\theta}{\sqrt{\theta^2 + 1}},$$
we obtain~\eqref{VAL:c}, as desired.
\end{proof}

\section{Integration by parts identities: bounded domains}
\label{sec6}

The aim of this section is to give the:

\begin{proof}[Proof of Theorem \ref{thm-Poh}]
First, notice that by a simple approximation argument we may assume that $f,g\in C^\infty(\overline\Omega)$ and $K\in C^\infty(S^{n-1})$, so that $u$ is smooth inside $\Omega$.

We split the proof in several steps.

\vspace{3mm}

\noindent \textbf{Step 1}. We prove that $u$ and $\nabla u$ can be approximated by a 1D solution at any boundary point $z\in\partial\Omega$.

For this, notice that thanks to Corollary \ref{cor-bdry} we have
\[\|u/d^{\bar \gamma}\|_{C^{\varepsilon_\circ}(\overline\Omega)} + \|v/d^{\bar \gamma^*}\|_{C^{\varepsilon_\circ}(\overline\Omega)} \leq C.\]
Combining this with the interior estimates from Theorem \ref{thm-interior}, we deduce that
\[\big|\nabla(u/d^{\bar \gamma})\big| + \big|\nabla(v/d^{\bar \gamma^*})\big| \leq Cd^{\varepsilon_\circ-1}.\]
This, in turn, yields that
\[\big| \nabla u-u\,d^{-\bar\gamma}\,\nabla(d^{\bar\gamma})\big| \leq Cd^{\bar\gamma+\varepsilon_\circ-1},\]
and hence
\[\big| \nabla u-c_z\nabla(d^{\bar\gamma})\big| \leq C|x-z|^{\varepsilon_\circ} d^{\bar\gamma-1},\]
where $c_z$ is the constant in Theorem \ref{thm-bdry}.

Furthermore, since $\Omega$ is a $C^{1,\alpha}$ domain, then it is not difficult to see that 
\[\big|\nabla(d^{\bar\gamma})-\nabla(\ell_z^\gamma)\big| \leq C|x-z|^{\varepsilon_\circ} \big(d^{\bar\gamma-1}+\ell_z^{\gamma-1}\big),\]
where
$\ell_z(x):= \big((x-z)\cdot\nu\big)_+$.

Therefore, if we denote by
\[U_z(x):= \ell_z^\gamma(x) = \big((x-z)\cdot\nu\big)_+^\gamma,\]
we have that 
\begin{equation}\label{werh}
\big| \nabla u-c_z\nabla U_z \big| \leq C|x-z|^{\varepsilon_\circ} \big(d^{\bar\gamma-1}+\ell_z^{\gamma-1}\big).
\end{equation}
By Theorem \ref{thm-bdry} we also have that
\begin{equation}\label{werh2}
\big| u-c_zU_z\big| \leq C|x-z|^{\gamma+\varepsilon_\circ},
\end{equation}
and analogous estimates hold for $v$ and $\gamma^*$, namely
\begin{equation}\label{werhv}
\big| \nabla v-c_z^*\nabla V_z \big| \leq C|x-z|^{\varepsilon_\circ} \big(d^{\bar\gamma^*-1}+\ell_z^{\gamma^*-1}\big)
\end{equation}
and
\begin{equation}\label{werh2v}
\big| v-c_z^*V_z\big| \leq C|x-z|^{\gamma^*+\varepsilon_\circ}.
\end{equation}

\vspace{3mm}

\noindent \textbf{Step 2}. For any small $r>0$, define $\Omega_r := \{x\in \Omega \, : {\rm dist}(x, \partial\Omega)>r\}$. 
Let us consider a family of points $x_{i,r}\in \partial \Omega$, with~$i\in \mathcal I_r$, such that the balls $B_{r/8}(x_{i,r})$ are a maximal disjoint cover of $\partial \Omega$. 
It is easy to see that then $B_{r/4}(x_{i,r})$ cover $\partial \Omega$, and $\{B_{r/2}(x_{i,r}) \ : \ i\in \mathcal I_r\}\cup \{\Omega_{r/2}\}$  is an open cover all of $\Omega$.
In addition, since $\partial \Omega$ is bounded and of class $C^{1,\alpha}$ the number of balls $|\mathcal I_r|$ can be taken comparable to $r^{1-n}$.

Fix $\xi \in C^\infty_c(B_2)$ that satisfies $\chi_{B_{1/2}}\le\xi \le \chi_{B_1}$ and define the convolution
\[
\tilde \eta_{0,r}(x) : =  \int_{\R^n} \chi_{\Omega_{r/4}}(x-y)  {\textstyle  \big(\frac{r}{8}\big)^{-n}\xi\big(\frac{8}{r} y\big)}\,dy.
\]
For $i\in \mathcal I_r$ define 
\[
\tilde \eta_{i,r}(x) : =  \eta( x_{i,r} + rx).
\]
Note that
$\chi_{\Omega_{r/2}} \tilde \eta_{0,r} \le \chi_{\Omega_{r/8}}$ and  $ \chi_{B_{r/2}(x_{i,r})} \le \tilde \eta_{i,r} \le \chi_{B_r(x_{i,r})}$.  Also, for all $i \in \{0\}\cup \mathcal I_r$ we have that $|D^m \tilde \eta_{i,r} | \le C(n,m) r^{-m}$.

Since by construction $B_{r/8}(x_{i,r})$ are disjoint, every point $x\in \Omega$ belongs at most to $C(n)$ balls~$B_{r}(x_{i,r})$. 
Hence, 
\[
S_r : =  \eta_{0,r} +\sum_{i\in\mathcal I_r} \tilde \eta_{r,i}  \quad \mbox{satisfies} \quad  1\le S_r \le 1+C(n) \quad\mbox{and}\quad |D^m S_r | \le C(n,m) r^{-m}.
\]
We can thus define, for all  $i \in \{0\}\cup \mathcal I_r$, 
\[
\eta_{i,r} : = \tilde \eta_{i,r}(x)/S_r.
\]
This  is a partition of unity (adapted to our covering) satisfying 
\[
\eta_{i,r}\ge 0 \qquad \mbox{and}\qquad  \sum \eta_i =1
\]
and 
\begin{equation}\label{ansiofhwoihwioh}
\eta_{0,r} \le \chi_{\Omega_{r/8}},\quad   \eta_{i,r} \le \chi_{B_r(x_{i,r})}  \quad \mbox{for $i\neq 0$} \quad \mbox{and} \quad |D^m  \eta_{i,r} | \le C(n,m) r^{-m}.
\end{equation}

We now define, for $i \in \{0\}\cup \mathcal I_r$,
\[u_{i,r}:= u\,\eta_{i,r} \qquad \textrm{and} \qquad v_{i,r}:= v\, \eta_{i,r}\]
and we notice that 
\begin{equation}\label{star64576}
u = \sum_{i\in \mathcal I_r} u_{i,r} \qquad \textrm{and} \qquad v = \sum_{i\in \mathcal I_r} v_{i,r}.\end{equation}

We claim that if  both $i$ and $j$ are different from $0$ and $|x_{i,r}-x_{j,r}|\geq 2r$, or one of the two is zero, then 
\begin{equation}\label{wniowhiohw}
\int_\Omega \big( \partial_e u_{i,r} L^* v_{j,r} + \partial_e v_{j,r} L u_{i,r} \big) = 0.
\end{equation}
Indeed, if the supports of $u_{i,r}$ and $v_{j,r}$ are disjoint, or if one of the two functions is supported away from the boundary $\partial\Omega$, then we can simply integrate by parts, to get
\[\int_\Omega \partial_e u_{i,r}\, L^* v_{j,r} = - \int_\Omega u_{i,r}\, L^* \partial_e v_{j,r} =  - \int_\Omega L u_{i,r} \, \partial_e v_{j,r},\]
which is~\eqref{wniowhiohw}.

\vspace{3mm}

\noindent \textbf{Step 3}.  We are left to proving that, if $i, j\in \mathcal I_r$ and $|x_{i,r}-x_{j,r}|<2r$, then 
\begin{equation}\label{final} 
\int_\Omega \big( \partial_e u_{i,r} L^* v_{j,r} + \partial_e v_{j,r} L u_{i,r} \big) 
= \int_{\partial\Omega}\,\frac{u_{i,r}}{d^{\gamma}}\,\frac{v_{j,r}}{d^{\gamma^*}}\, \Gamma(\gamma + 1)\, \Gamma(\gamma^* + 1)\,\sqrt{\mathcal A^2(\nu)+\mathcal B^2(\nu)}\,dz + O(r^{n-1+\varepsilon_\circ}).
\end{equation}

For this, recall that~$x_{i,r},x_{j,r} \in \partial \Omega$ and  take $z\in\partial \Omega$ (here we may assume that $r$ is sufficienly small) such that 
\[B_{4r}(z) \supset B_r(x_{i,r})\cup B_r(x_{j,r}).\]

Next,  we apply Proposition \ref{flat-case} to the hyperplane $H_z:=\{(x-z)\cdot \nu(z) >0\}$ and the functions $U_z\eta_{i,r}$ and $V_z\eta_{j,r}$ ($U_z$ and $V_z$ defined in Step 1), to get 
\[\int_{H_z}\big( \partial_e (U_z\eta_{i,r}) L^* (V_z\eta_{j,r}) + \partial_e(V_z\eta_{j,r}) L (U_z\eta_{i,r}) \big) 
= \int_{\partial H_z}\,\eta_{i,r} \eta_{j,r}\, \Gamma(\gamma + 1)\, \Gamma(\gamma^* + 1)\,\sqrt{\mathcal A^2(\nu)+\mathcal B^2(\nu)}\,dz.\]

Now, using that $Lu,L^*v\in L^\infty$, the estimates \eqref{werh}, \eqref{werh2}
and~\eqref{ansiofhwoihwioh},
it is not difficult to see that
\[\big| \nabla \big( u_{i,r}-c_zU_z\eta_{i,r}\big) \big| \leq Cr^{\varepsilon_\circ} \big(d^{\bar\gamma-1}+\ell_z^{\gamma-1}\big) + Cr^{\gamma+\varepsilon_\circ-1} \qquad \textrm{in}\quad B_{2r}(z),\]
\[\big| L \big( u_{i,r}-c_zU_z\eta_{i,r}\big) \big| \leq Cr^{\gamma+\varepsilon_\circ-2s} \qquad \textrm{in}\quad B_{2r}(z),\]
\[\big| \nabla u_{i,r} \big| \leq Cd^{\bar\gamma-1} + Cr^{\gamma-1} \qquad \textrm{in}\quad B_{2r}(z)\]
and
\[|L(u_{i,r})| \leq Cr^{\gamma-2s} \qquad \textrm{in}\quad B_{2r}(z).\]
Moreover, analogous estimates hold for $v$ and $\gamma^*$.
Integrating over $B_{2r}(z)$, and using that $\gamma+\gamma^*=2s$, we deduce that 
\[\int_{H_z}\big( \partial_e (c_zU_z\eta_{i,r}) L^* (c_z^*V_z\eta_{j,r}) + \partial_e(c_z^*V_z\eta_{j,r}) L (c_zU_z\eta_{i,r}) \big) = 
\int_{\Omega}\big( \partial_e u_{i,r} L^* v_{j,r} + \partial_e v_{j,r} L u_{i,r} \big) + O(r^{n-1+\varepsilon_\circ}).\]

Similarly, using that 
\[\big| u/d^\gamma-c_z\big| \leq Cr^{\varepsilon_\circ}\qquad \textrm{and} \qquad \big| v/d^{\gamma^*}-c_z^*\big| \leq Cr^{\varepsilon_\circ}\qquad \textrm{on}\quad \partial\Omega\cap B_{2r}(z),
\]
that $\gamma,\gamma^*$ are H\"older continuous and that $\partial\Omega$ is $C^{1,\alpha}$, it follows that 
\[\begin{split}
\int_{\partial H_z}\,c_z \,c_z^*\,\eta_{i,r}\, \eta_{j,r}\,  \Gamma(\gamma + 1)\,& \Gamma(\gamma^* + 1)\,\sqrt{\mathcal A^2(\nu)+\mathcal B^2(\nu)}\,dz
\\
& = \int_{\partial\Omega}\,\frac{u_{i,r}}{d^{\gamma}}\,\frac{v_{J,r}}{d^{\gamma^*}}\, \Gamma(\gamma + 1)\, \Gamma(\gamma^* + 1)\,\sqrt{\mathcal A^2(\nu)+\mathcal B^2(\nu)}\,dz + O(r^{n-1+\varepsilon_\circ}).\end{split}
\]
Thus, \eqref{final} follows.

\vspace{3mm}

\noindent \textbf{Step 4}.  
To finish the proof, we sum \eqref{final} --- or \eqref{wniowhiohw}--- over all $i,j\in \{0\}\cup\mathcal I_r$, to deduce that 
\[
\int_\Omega \big( \partial_e u \, L^* v + \partial_e v \, L u \big) 
= \int_{\partial\Omega}\,\frac{u}{d^{\gamma}}\,\frac{v}{d^{\gamma^*}}\, \Gamma(\gamma + 1)\, \Gamma(\gamma^* + 1)\,\sqrt{\mathcal A^2(\nu)+\mathcal B^2(\nu)}\,dz + O(r^{\varepsilon_\circ}),
\]
where we used that $|\mathcal I_r|\leq Cr^{1-n}$ and~\eqref{star64576}.

Letting $r\to0$, \eqref{Poh} follows.
\end{proof}

\appendix

\section{The one-dimensional solution}
\label{secA}

The goal of this Appendix is to prove Proposition \ref{1D-solution}.
Such result is known; see for example \cite[Eq. (10)]{Jus} or \cite[Lemmas VII.11 and VIII.1]{Bertoin}, see also~\cite[Section~2]{Grubb}.
Still, for completeness, we provide here a proof by direct computation.

It is interesting to notice that these one-dimensional operators have an associated extension problem; see \cite[pp. 30-31]{Kw}.
The same computation below could be probably done by using such extension technique.

\begin{proof}[Proof of Proposition \ref{1D-solution}]
When $b=0$ this is well known, so we assume $b\neq0$.

We first deal with the case~$s\in\left(0,\frac12\right)$.
In this case we have 
\[
Lu(x) = aK_1+bK_2,\]
where
\begin{equation}\label{KAPPI}\begin{split}&
K_1:=\int_\R\big( 1-(1+y)_+^\beta\big)\,\frac{dy}{|y|^{1+2s}}\\
{\mbox{and }}\quad& K_2:=\int_\R
\big( 1-(1+y)_+^\beta\big)\,\frac{\sign y\,dy}{|y|^{1+2s}}.
\end{split}\end{equation}
Moreover, we have that
\begin{equation}\label{COIN}
\begin{split}
& K_1=\frac{\Gamma(1-2s)}{2s}\,\left( \frac{\Gamma(2s-\beta)}{\Gamma(-\beta)}
+\frac{\Gamma(1+\beta)}{\Gamma(1-2s+\beta)}\right)\\
{\mbox{and }}\;\;& K_2=\frac{\Gamma(1-2s)}{2s}\,\left( \frac{\Gamma(2s-\beta)}{\Gamma(-\beta)}
-\frac{\Gamma(1+\beta)}{\Gamma(1-2s+\beta)}\right),\end{split}
\end{equation}
and so, using Euler's Reflection Formula,
\begin{eqnarray*}
&& K_1=\frac{\pi\Gamma(1-2s)}{2s\,\Gamma(-\beta)\,\Gamma(1-2s+\beta)}\,
\left( \frac{1}{\sin(\pi(2s-\beta))}-\frac{1}{\sin(\pi\beta)}\right)\\
{\mbox{and }}&& K_2=\frac{\pi\Gamma(1-2s)}{2s\,\Gamma(-\beta)\,\Gamma(1-2s+\beta)}\,
\left( \frac{1}{\sin(\pi(2s-\beta))}+\frac{1}{\sin(\pi\beta)}\right).
\end{eqnarray*}
We thereby deduce that
\begin{equation}\label{QUOZ}  \frac{K_1}{K_2}=
\frac{\frac{1}{\sin(\pi(2s-\beta))}-\frac{1}{\sin(\pi\beta)}}{
\frac{1}{\sin(\pi(2s-\beta))}+\frac{1}{\sin(\pi\beta)} }=
\frac{ \sin(\pi\beta)-\sin(\pi(2s-\beta)) }{\sin(\pi\beta)
+\sin(\pi(2s-\beta))}=-\frac{ \tan(\pi (s - \beta))}{
\tan(\pi s)},
\end{equation}
and hence
\[\begin{split}
\kappa_{\beta,L} & = aK_1+bK_2 = K_2\left(a\,\frac{K_1}{K_2}+b\right)\\
& =\frac{\pi\Gamma(1-2s)}{2s\,\Gamma(-\beta)\,\Gamma(1-2s+\beta)}\,
\left( \frac{1}{\sin(\pi(2s-\beta))}+\frac{1}{\sin(\pi\beta)}\right) \left(-a\,\frac{\tan(\pi (s - \beta))}{\tan(\pi s)}+b\right)\\
& = \frac{\pi\Gamma(-2s)}{\Gamma(-\beta)\,\Gamma(1-2s+\beta)}\,
\frac{\sin(\pi\beta)+\sin(\pi(2s-\beta))}{\sin(\pi\beta)\,\sin(\pi(2s-\beta))} \left(a\,\frac{ \tan(\pi (s - \beta))}{
\tan(\pi s)}-b\right)\\
& = \frac{2\pi\Gamma(-2s)}{\Gamma(-\beta)\,\Gamma(1-2s+\beta)}\,
\frac{\sin(\pi s)\cos(\pi(\beta-s))}{\sin(\pi\beta)\,\sin(\pi(2s-\beta))}\left(a\,\frac{ \tan(\pi (s - \beta))}{\tan(\pi s)}-b\right)\\
& = \frac{2\pi\Gamma(-2s)}{\Gamma(-\beta)\,\Gamma(1-2s+\beta)}\,
\frac{\cos(\pi s)\cos(\pi(\beta-s))}{\sin(\pi\beta)\,\sin(\pi(2s-\beta))}\big(a\,\tan(\pi (s - \beta))-\tan(\pi s)\big).
\end{split}\]
The sign of $\kappa_{\beta,L}$ and the expression for $\gamma_L$ follow from the previous expression combined with Lemma~\ref{TRIG} below.

Finally, when~$s\in\left(\frac12,1\right)$ we have
\begin{equation*}\begin{split}&
K_1:=\int_\R\big( 1-(1+y)_+^\gamma+\gamma y\big)\,\frac{dy}{|y|^{1+2s}}\\
{\mbox{and }}\quad& K_2:=\int_\R
\big( 1-(1+y)_+^\gamma+\gamma y\big)\,\frac{\sign y\,dy}{|y|^{1+2s}},
\end{split}\end{equation*}
and
\begin{eqnarray*}
&& K_1=\frac{\Gamma(1-2s)}{2s}\,\left( \frac{\Gamma(2s-\gamma)}{\Gamma(-\gamma)}
+\frac{\Gamma(1+\gamma)}{\Gamma(1-2s+\gamma)}\right)\\
{\mbox{and }}&& K_2=\frac{\Gamma(1-2s)}{2s}\,\left( \frac{\Gamma(2s-\gamma)}{\Gamma(-\gamma)}
-\frac{\Gamma(1+\gamma)}{\Gamma(1-2s+\gamma)}\right),
\end{eqnarray*}
which coincide with the expressions for $s\in(0,\frac12)$.
Hence, one can follow the previous computations and complete the proof in this case as well.
\end{proof}

Our analysis of the one-dimensional solution relies
on an elementary, but not trivial, trigonometric
observation, given in the following result:

\begin{lem}\label{TRIG}
Let~$s\in(0,1)$, with $s\neq\frac12$. 
Let~$a>0$ and~$b\in(-a,a)$, with $b\neq0$. 
Then,
there exists a unique~$\gamma\in(0,2s)$ satisfying
\begin{equation*}
\tan(\pi s)=\frac{a}{b}\,\tan\big(\pi(s-\gamma)\big).
\end{equation*}
In addition, when $s>\frac12$ such~$\gamma$ satisfies
\begin{equation*}
\gamma\in(2s-1,1).
\end{equation*}
\end{lem}

\begin{proof} First, we deal with the case~$s\in\left(0,\frac12\right)$.
In this case, using the substitution~$x=s-\gamma$, the desired
claim reduces to:
there exists a unique~$x\in(-s,s)$ satisfying
\begin{equation}\label{TAN1}
\tan(\pi s)=\frac{a}{b}\,\tan (\pi x),
\end{equation}
and, in addition, such~$x$ satisfies
\begin{equation}\label{TAN2}
x\in(s-1,1-s).
\end{equation}
To prove this claim, we set
$$ x:=\frac1\pi\,\arctan\left(\frac{b}{a}\tan(\pi s)\right).$$
By construction, we have that such~$x$ satisfies~\eqref{TAN1} and
$$ |x|\le \frac1\pi\,\arctan\left(\frac{|b|}{a}|\tan(\pi s)|\right)
<\frac1\pi\,\arctan\left( |\tan(\pi s)|\right)=s,$$
and thus~$x\in(-s,s)$. 

We also observe that such~$x$
is unique. Indeed, if~$y\in(-s,s)\setminus\{x\}$ is another solution of~\eqref{TAN1},
it follows that
$$ \tan(\pi s)=\frac{a}{b}\,\tan (\pi x)=\frac{a}{b}\,\tan (\pi y),$$
and therefore~$\tan(\pi x)=\tan(\pi y)$. This gives that~$x-y\in\mathbb Z$ and consequently
$$ 1\le |x-y|\le |x|+|y|<s+s<1,$$
which is a contradiction.

The additional
condition in~\eqref{TAN2} is obvious in this case, using that~$x\in(-s,s)$
and~$s\in\left(0,\frac12\right)$.

These observations prove the desired claim when~$s\in\left(0,\frac12\right)$,
and we now assume that~$s\in\left(\frac12,1\right)$. In this case,
we define~$\sigma:=s-\frac12\in\left(0,\frac12\right)$, and we observe that,
in view of the substitution~$x=\sigma-\gamma$, the desired
claim reduces to:
there exists a unique~$x\in(-1-\sigma,\sigma)$ satisfying
\begin{equation}\label{STAN1}
\tan(\pi \sigma)=\frac{b}{a}\,\tan (\pi x),
\end{equation}
and, in addition, such~$x$ satisfies
\begin{equation}\label{STAN2}
x\in(\sigma-1,-\sigma).
\end{equation}
To prove this claim, we set
$$ \bar x:=\frac1\pi\,\arctan\left(\frac{a}{b}\tan(\pi \sigma)\right)\in\left(-\frac12,\frac12\right).$$
Notice that~$\bar x$ is a solution of~\eqref{STAN1}.
We claim that
\begin{equation}\label{YU}
\bar x\not\in\{-1-\sigma,-\sigma\},
\end{equation}
otherwise
$$ \tan(\pi \sigma)=\frac{b}{a}\,\tan (\pi \bar x)=-\frac{b}{a}\,\tan (\pi \sigma),$$
which leads to~$b=-a$, that is a contradiction.

In view of~\eqref{YU}, there exists a unique~$\bar k\in\mathbb Z$
such that
\begin{equation}\label{7XA}
\bar x+\bar k\in
(-1-\sigma,-\sigma).\end{equation} We define~$x:=\bar x+\bar k$
and we remark that~$x$ is also a solution of~\eqref{STAN1}
and it clearly lies in~$(-1-\sigma,\sigma)$.

We claim that~$x$ satisfies the additional property in~\eqref{STAN2}.
To prove this, suppose not.
Then, from~\eqref{7XA}, we have that~$x\in(-1-\sigma,\sigma-1]$.
Accordingly, we have that~$\pi x\in(-\pi-\pi\sigma,\pi\sigma-\pi]=
-\pi+(-\pi\sigma,\pi\sigma]$, and thus~$\tan(\pi x)\in
\big(-\tan(\pi\sigma),\tan(\pi\sigma)\big]$.
Hence, using~\eqref{STAN1},
$$ |\tan(\pi \sigma)|=\frac{|b|}{a}\,|\tan (\pi x)|\le
\frac{|b|}{a}\,|\tan(\pi \sigma)|,$$
which leads to~$|b|=a$, that is a contradiction.

Now, we check that such~$x$ is unique. Suppose not,
then there exists~$y\in(-1-\sigma,\sigma)\setminus\{x\}$
satisfying~\eqref{STAN1}.
Hence, we have that
\begin{equation}\label{SAT} \tan(\pi \sigma)=\frac{b}{a}\,\tan (\pi x)
=\frac{b}{a}\,\tan (\pi y),\end{equation}
and therefore~$x-y\in\mathbb Z$.
By the observation above (replacing~$x$ with~$y$) we also know that~$y$
satisfies~\eqref{STAN2}. Consequently,
$$ 1\le |x-y|\le {\rm diam}\,(-1+\sigma,-\sigma)=1.$$
Then necessarily~$x$ and~$y$ lie at the boundary
of~$(-1+\sigma,-\sigma)$, and so either~$x$ or~$y$
must coincide with~$-\sigma$. Then, we find from~\eqref{SAT}
that
$$ \tan(\pi \sigma)=-\frac{b}{a}\,\tan (\pi \sigma),$$
and therefore~$b=-a$, which is a contradiction.
\end{proof}

Finally, we have the following.

\begin{cor}\label{TRIG2}
Let~$s\in(0,1)\setminus\left\{\frac12\right\}$. Let~$a>0$ and~$b\in(-a,a)\setminus\{0\}$.
Then,
there exists a unique~$\gamma\in(0,2s)$ satisfying
\begin{equation*}
\tan(\pi s)=\frac{a}{b}\,\tan\big(\pi(s-\gamma)\big),
\end{equation*}
which can be written as
\begin{equation}\label{6YAA}
\gamma = s-\frac1\pi\,\arctan\left(\frac{b}{a}\,\tan(\pi s)\right).
\end{equation}
and it also satisfies
\begin{equation}\label{6YA}
\gamma\in(2s-1,1).
\end{equation}
\end{cor}

\begin{proof} The existence and uniqueness of the desired~$\gamma$ follows from
Lemma~\ref{TRIG}, which also ensures~\eqref{6YA}. As a consequence of~\eqref{6YA},
we also have that
$$s-\gamma\in(0,2s)\cap
(s-1, 1-s)\subseteq \left( 0,\min\{2s,1-s\}\right)\subseteq
\left(0,\frac12\right),$$
and therefore~$\pi(s-\gamma)\in\left(0,\frac\pi2\right)$.
This and~\eqref{6YAA} give~\eqref{6YA}.
\end{proof}

\section{Characterization of stable operators}
\label{secB}

In this appendix we prove Proposition \ref{stable-operators}.
To this end, we give the following result, in which
one can compare the Fourier transform of~\eqref{operator-L} with respect to the variable~$x$
with the forthcoming expression in~\eqref{Ecces},
making~$s$ correspond to~$\alpha/2$ and~$K(y)\,dy$ to~$d\nu(y)$.
Also, the additional assumption in~\eqref{98988w019375=A} when~$s=1/2$
reduces to~\eqref{98988w019375} here.

\begin{prop}\label{P.LEVYf}
Let $\alpha\in(0,2)$ and $X_t$ be an $\alpha$-stable, $n$-dimensional, L\'evy process. 
Then, for every $\zeta\in \R^n$ the expected value of $e^{i\zeta \cdot X_t}$ can be written as $e^{-t\Psi(\zeta)}$,
where
\begin{equation}\label{Ecces} \Psi(\zeta):=\begin{cases}
\displaystyle \int_{\R^n}
\big(1-e^{i\zeta\cdot y} \big)\,d\mu(y)& {\mbox{ if }}\alpha\in(0,1),\\ \displaystyle
i b\cdot\zeta+\int_{\R^n}
\Big(1-e^{i\zeta\cdot y}+i{\zeta\cdot y}\,\chi_{B_1}(y)
\Big)\,d\mu(y)&{\mbox{ if }}\alpha=1,\\ \displaystyle
\int_{\R^n}
\Big(1-e^{i\zeta\cdot y}+i{\zeta\cdot y}
\Big)\,d\mu(y)
&{\mbox{ if }}\alpha\in(1,2),
\end{cases}\end{equation}
for some $b\in \R^n$ and a measure~$\mu$ such that
\begin{equation}\label{GHLiKihi}\int_{\R^n} \min\{1,|y|^2\}\,d\mu(y)<+\infty.\end{equation}
Furthermore, when~$\alpha=1$,
\begin{equation}\label{98988w019375}
\int_{B_{R}\setminus B_r}y\,d\mu(y)=0,
\end{equation}
for all~$R>r>0$.
\end{prop}

\begin{proof} By the L\'evy-Khintchine formula
(see e.g. Theorem~43 in \cite{MR2273672} or Theorem~4 in~\cite{MR1861717}), we know that the desired claim holds true with
\begin{equation}\label{LAPSO}\Psi(\zeta):=
i b\cdot\zeta+Q\zeta\cdot\zeta+\int_{\R^n}
\Big(1-e^{i\zeta\cdot y}+i{\zeta\cdot y}\,\chi_{B_1}(y)
\Big)\,d\mu(y),\end{equation}
for some~$b\in\R^n$ and~$Q\in{\rm Mat}(n\times n)$.
Moreover, see Corollary~8.9 in~\cite{MR3185174},
$$ \lim_{t\searrow0}\frac{{\mathbb{E}}\big( f(X_t)\big)}{t}=\int_{\R^n}
f(y)\,d\mu(y).$$
Now we exploit the fact that the process is~$\alpha$-stable, hence
the distribution of~$\frac{X_{t}}{t^{1/\alpha}}$
coincides with that of~$X_1$ for all~$t>0$. Then, 
if~$V\subset\R^n$ and~$\lambda>0$,
we let~$\tau:=t/\lambda^\alpha$ and we see that
\begin{equation}\label{MEASURE}\begin{split}&
\mu(\lambda V)=\int_{\R^n}\chi_{\lambda V}(y)\,d\mu(y)=
\lim_{t\searrow0}\frac{1}{t}\,
{\mathbb{E}}\big( \chi_{\lambda V}(X_t)\big)=
\lim_{t\searrow0}\frac{1}{t}\,
{\mathbb{E}}\big( \chi_{t^{-1/\alpha}\lambda V}(X_1)\big)\\&\qquad=
\lim_{\tau\searrow0}\frac{1}{\lambda^\alpha\tau}\,
{\mathbb{E}}\big( \chi_{\tau^{-1/\alpha}V}(X_1)\big)=
\lim_{\tau\searrow0}\frac{1}{\lambda^\alpha\tau}\,
{\mathbb{E}}\big( \chi_{V}(X_\tau)\big)=\frac1{\lambda^\alpha}\,
\int_{\R^n}\chi_{ V}(y)\,d\mu(y)=\frac{ \mu( V) }{\lambda^\alpha},
\end{split}\end{equation}
which means that~$\mu$ is $\alpha$-homogeneous (the interested reader
may also look at
Lemma~3.10 in~\cite{MR3842154} for a general approach).

In addition,
\begin{eqnarray*}
e^{-t\Psi(\zeta/\lambda)}={\mathbb{E}}
\big(
e^{i \lambda^{-1}\zeta\cdot X_t} \big)
={\mathbb{E}}\big(
e^{i t^{1/\alpha} \lambda^{-1} \zeta \cdot X_1}   
\big)=
{\mathbb{E}}\big(
e^{i\tau^{1/\alpha}\zeta \cdot X_1}
\big)=
{\mathbb{E}}\big(e^{i\zeta\cdot X_\tau}\big)=
e^{-\tau\Psi(\zeta)}=
e^{-t\lambda^{-\alpha}\Psi(\zeta)},
\end{eqnarray*}
which gives that
\begin{equation}\label{MEASURE2}
\Psi\left(\frac\zeta\lambda\right)=\frac{\Psi(\zeta)}{\lambda^{\alpha}}.\end{equation}
Now, we focus on the case~$\alpha\in(0,1)$. In this case, using~\eqref{MEASURE},
for every~$\rho>0$
we find that
\begin{equation}\label{PALLAr} \begin{split}&
\int_{B_\rho} | y|\,d\mu(y)
\le\sum_{k=0}^{+\infty}\int_{B_{\rho2^{-k}}\setminus B_{\rho2^{-k-1}}} | y|\,d\mu(y)
\le\sum_{k=0}^{+\infty}\int_{B_{\rho2^{-k}}\setminus B_{\rho2^{-k-1}}} \rho2^{-k}\,d\mu(y)\\&\qquad=
\sum_{k=0}^{+\infty} \rho2^{-k}\mu \big( {B_{\rho2^{-k}}\setminus B_{\rho2^{-k-1}}}\big)=
\sum_{k=0}^{+\infty} \rho2^{-k}\mu\Big(\rho2^{-k} \big( B_{1}\setminus B_{1/2}\big)\Big)
\\&\qquad=\rho^{1-\alpha}
\mu\big( B_{1}\setminus B_{1/2}\big)
\sum_{k=0}^{+\infty} 2^{-k(1-\alpha)}\le
C(\mu,\alpha)\,\rho^{1-\alpha},
\end{split}\end{equation}
for some~$C(\mu,\alpha)$, thanks to~\eqref{GHLiKihi},
and thus, in particular, taking~$\rho:=1$,
$$ \int_{\R^n} | y|\,\chi_{B_1}(y)\,d\mu(y)<+\infty.$$
Hence we can define
$$ \tilde b:=
b+\int_{\R^n} y\,\chi_{B_1}(y)\,d\mu(y),$$
and write~\eqref{LAPSO} in the form
\begin{equation} \label{MISURA3}\Psi(\zeta)=
i \tilde b\cdot\zeta+Q\zeta\cdot\zeta+\int_{\R^n}
\big(1-e^{i\zeta\cdot y}\big)\,d\mu(y).\end{equation}
{F}rom this and~\eqref{MEASURE2}, we deduce that
\begin{equation}\label{iupieb754P}\begin{split}
0\,&=
\lambda^{\alpha}\Psi\left(\frac\zeta\lambda\right)-\Psi(\zeta)\\
&= (\lambda^{\alpha-1}-1)i \tilde b\cdot\zeta+
(\lambda^{\alpha-2}-1)Q\zeta\cdot\zeta+
\int_{\R^n}
\big(\lambda^\alpha-1-\lambda^\alpha e^{i\lambda^{-1}\zeta\cdot y}
+e^{i\zeta\cdot y}
\big)\,d\mu(y)\\&= 
-i \tilde b\cdot\zeta-Q\zeta\cdot\zeta+
\int_{\R^n}
\big(-1+e^{i\zeta\cdot y}
\big)\,d\mu(y)
+\lambda^{\alpha-1} i \tilde b\cdot\zeta+
\lambda^{\alpha-2} Q\zeta\cdot\zeta+\lambda^\alpha
\int_{\R^n}
\big(1- e^{i\lambda^{-1}\zeta\cdot y}
\big)\,d\mu(y)
.\end{split}
\end{equation}
Now, given~$\zeta\in\R^n\setminus\{0\}$, we remark that
\begin{equation*}
\begin{split}
&\left|
\int_{\R^n}
\big(1- e^{i\lambda^{-1}\zeta\cdot y}
\big)\,d\mu(y)\right|\le \lambda^{-1}|\zeta|
\int_{B_{\lambda/|\zeta|}}|y|\,d\mu(y)
+2\int_{\R^n\setminus B_{\lambda/|\zeta|}}\,d\mu(y)\\&\quad\le
C(\mu,\alpha)\, \lambda^{-1}|\zeta|\,\left(\frac{\lambda}{|\zeta|}\right)^{1-\alpha}
+\frac{ 2\mu( \R^n\setminus B_{1}) }{(\lambda/|\zeta|)^\alpha}\le\frac{
C(\mu,\alpha,\zeta)}{\lambda^\alpha},
\end{split}\end{equation*} for some~$C(\mu,\alpha,\zeta)>0$,
thanks to~\eqref{GHLiKihi}, \eqref{MEASURE} and~\eqref{PALLAr}.

As a consequence, we can write~\eqref{iupieb754P} as
$$ 0=\lambda^{\alpha-1} i \tilde b\cdot\zeta+
\lambda^{\alpha-2} Q\zeta\cdot\zeta+O(1),$$
where~$O(1)$ remains bounded as~$\lambda\searrow0$.
This gives that~$Q\zeta\cdot\zeta=0$ and~$ \tilde b\cdot\zeta=0$.

Inserting these pieces of information into~\eqref{MISURA3},
we obtain that
$$ \Psi(\zeta)=\int_{\R^n}
\big(1-e^{i\zeta\cdot y}\big)\,d\mu(y),$$
thus completing the proof of~\eqref{Ecces} when~$\alpha\in(0,1)$.

Now we deal with the case~$\alpha=1$. In this case,
we deduce from~\eqref{MEASURE2} that
\begin{equation}\label{Inthef}\begin{split}
0\,&=
\lambda\Psi\left(\frac\zeta\lambda\right)-\Psi(\zeta)\\
&=
(\lambda^{-1} -1)Q\zeta\cdot\zeta
+\lambda\int_{\R^n}
\Big(1-e^{i\lambda^{-1}\zeta\cdot y}+i{\lambda^{-1}\zeta\cdot y}\,\chi_{B_1}(y)
\Big)\,d\mu(y)\\&\qquad
-\int_{\R^n}
\Big(1-e^{i\zeta\cdot y}+i{\zeta\cdot y}\,\chi_{B_1}(y)
\Big)\,d\mu(y)
.\end{split}\end{equation}
Also, from~\eqref{MEASURE}, applied here with~$\alpha=1$, we have that
\begin{equation*} \int_{\R^n}\chi_V\left(\frac{y}\lambda\right)\,d\mu(y)=
\int_{\R^n}\chi_{\lambda V}(y)\,d\mu(y)=\mu(\lambda V)=\frac{\mu(V)}{\lambda}=
\frac{1}{\lambda}\,\int_{\R^n}\chi_V(y)\,d\mu(y),\end{equation*}
from which we obtain the general ``change of variable formula''
\begin{equation}\label{COVF} \int_{\R^n} f\left(\frac{y}\lambda\right)\,d\mu(y)=\frac{1}{\lambda}\,
\int_{\R^n} f(y)\,d\mu(y).\end{equation}
On that account, we see that $$
\int_{\R^n}
\Big(1-e^{i\lambda^{-1}\zeta\cdot y}+i{\lambda^{-1}\zeta\cdot y}\,\chi_{B_1}(y)
\Big)\,d\mu(y)=\frac1\lambda\,
\int_{\R^n}
\Big(1-e^{i\zeta\cdot y}+i{\zeta\cdot y}\,\chi_{B_{1/\lambda}}(y)
\Big)\,d\mu(y)
.$$
Consequently, we can write~\eqref{Inthef} in the form
\begin{equation}\label{BECJk01}
0=
(\lambda^{-1} -1)Q\zeta\cdot\zeta+i\zeta\cdot
\int_{\R^n}y\,\Big(\chi_{B_{1/\lambda}}(y)-\chi_{B_{1}}(y)\Big)\,d\mu(y).
\end{equation}
Now, for every~$\rho>1$, we let~$M\in\mathbb N\cap [\log_2\rho,\log_2\rho+1)$, and we observe that
\begin{eqnarray*}&&
\int_{\R^n}|y|\,\Big(\chi_{B_{\rho}}(y)-\chi_{B_{1}}(y)\Big)\,d\mu(y)
\le\int_{\R^n}|y|\,\Big(\chi_{B_{2^M}}(y)-\chi_{B_{1}}(y)\Big)\,d\mu(y)\\&&\quad\le
\sum_{k=1}^M
\int_{\R^n}|y|\,\Big(\chi_{B_{2^k}}(y)-\chi_{B_{2^{k-1}}}(y)\Big)\,d\mu(y)\le
\sum_{k=1}^M
2^k\,\mu( {B_{2^k}}\setminus{B_{2^{k-1}}})\\&&\quad=
\sum_{k=1}^M
2^k\,\mu\big( 2^k({B_{1}}\setminus{B_{1/2}})\big)=\sum_{k=1}^M
\mu({B_{1}}\setminus{B_{1/2}})=M\mu({B_{1}}\setminus{B_{1/2}})\\&&\quad\le
(\log_2\rho+1)\,\mu({B_{1}}\setminus{B_{1/2}}).
\end{eqnarray*}
Applying this result with~$\rho:=\lambda^{-1}$, we obtain that, as~$\lambda\searrow0$,
$$ \int_{\R^n}y\,\Big(\chi_{B_{1/\lambda}}(y)-\chi_{B_{1}}(y)\Big)\,d\mu(y)=O(|\log\lambda|).$$
Therefore, we can write~\eqref{BECJk01} as
\begin{equation*}
0=
\lambda^{-1}  Q\zeta\cdot\zeta+O(|\log\lambda|),
\end{equation*}
from which we deduce that
\begin{equation}\label{NOAwe}
Q\zeta\cdot\zeta=0.\end{equation}
Then, we can reformulate~\eqref{LAPSO} as
\begin{equation*}\Psi(\zeta)=
i b\cdot\zeta+\int_{\R^n}
\Big(1-e^{i\zeta\cdot y}+i{\zeta\cdot y}\,\chi_{B_1}(y)
\Big)\,d\mu(y),\end{equation*}
which gives~\eqref{Ecces} when~$\alpha=1$.

To prove the additional result in~\eqref{98988w019375},
it is convenient to recall~\eqref{BECJk01} and~\eqref{NOAwe} and write that
\begin{equation}\label{QUias} 0=\int_{\R^n}y\,\Big(\chi_{B_{1/\lambda}}(y)-\chi_{B_{1}}(y)\Big)\,d\mu(y)=
\int_{\R^n}y\,\chi_{B_{1/\lambda}\setminus B_1}(y)\,d\mu(y),\end{equation}
for all~$\lambda\in(0,1)$.

That is, for all~$R>r>0$, taking~$\lambda:=r/R$ in~\eqref{QUias},
and using~\eqref{COVF} with~$\lambda:=1/r$,
\begin{equation*}\begin{split}&
\int_{B_{R}\setminus B_r}y\,d\mu(y)=
\int_{\R^n}y\,\chi_{B_{R}\setminus B_r}(y)\,d\mu(y)\\&\qquad=
\frac1r\,\int_{\R^n}ry\,\chi_{B_{R}\setminus B_r}(ry)\,d\mu(y)=
\int_{\R^n}y\,\chi_{B_{R/r}\setminus B_1}(y)\,d\mu(y)=0,
\end{split}
\end{equation*}
which gives~\eqref{98988w019375}.

Now we deal with the case~$\alpha\in(1,2)$.
In this case, the use of~\eqref{MEASURE2}
gives that
\begin{equation}\label{STyperbe}\begin{split}
0\,&=
\lambda^{\alpha}\Psi\left(\frac\zeta\lambda\right)-\Psi(\zeta)\\
&= 
i(\lambda^{\alpha-1} -1)b\cdot\zeta+
(\lambda^{\alpha-2}-1)
Q\zeta\cdot\zeta\\&\qquad+\int_{\R^n}
\Big(\lambda^{\alpha}\big(1- e^{i\lambda^{-1}\zeta\cdot y}
+i\lambda^{-1} {\zeta\cdot y}\,\chi_{B_1}(y)
\big)+\big(-1+
e^{i\zeta\cdot y}
-i{\zeta\cdot y}\,\chi_{B_1}(y)
\big)
\Big)\,d\mu(y)
.\end{split}
\end{equation}
Now we point out that, for every~$r\ge1$,
\begin{equation} \label{eadsncora}\big|-1+e^{ir\zeta\cdot y}
-i{r\zeta\cdot y}\,\chi_{B_1}(y)\big|\le \begin{cases}
r^2|\zeta|^2\,| y|^2 & {\mbox{ if }}y\in B_1\cap B_{{1}/{(r|\zeta|)}},\\
2r|\zeta|\,| y|-1 & {\mbox{ if }}y\in B_1\setminus B_{{1}/{(r|\zeta|)}},\\
2 & {\mbox{ if }}y\in\R^n\setminus B_1.
\end{cases},\end{equation}
Indeed, if~$y\in B_1$, we write the left hand side of~\eqref{eadsncora}
as
\begin{equation}\label{castatnzurhf}\begin{split}&
\big|-1+e^{ir\zeta\cdot y}-i{r\zeta\cdot y}\big|=
\left|i\int_0^{r\zeta\cdot y} e^{i\theta}\,d\theta
-i{r\zeta\cdot y}
\right|\\&\qquad=\left|\int_0^{r\zeta\cdot y} \big(e^{i\theta}-1\big)\,d\theta
\right|\le2\int_0^{r|\zeta|\,| y|}\min\{1,\theta\}\,d\theta.
\end{split}\end{equation}
Then, if~$|y|\le\frac{1}{r|\zeta|}$, we use~\eqref{castatnzurhf} to write that
$$ \big|-1+e^{ir\zeta\cdot y}-i{r\zeta\cdot y}\big|\le
2\int_0^{r|\zeta|\,| y|} \theta\,d\theta=r^2|\zeta|^2\,| y|^2,$$
which gives~\eqref{eadsncora} in this case.

If instead~$\frac{1}{r|\zeta|}\le|y|\le1$, we deduce from~\eqref{castatnzurhf} 
that
$$ \big|-1+e^{ir\zeta\cdot y}-i{r\zeta\cdot y}\big|\le
2\int_0^{1} \theta\,d\theta+
2\int_1^{r|\zeta|\,| y|}\,d\theta=
2r|\zeta|\,| y|-1,$$
that is~\eqref{eadsncora} in this case.

Now, we focus on the case~$|y|>1$. In this case, the left hand side of~\eqref{eadsncora}
reduces to~$ |-1+e^{ir\zeta\cdot y}|\le2$, and this completes the proof
of~\eqref{eadsncora}.

As a result of~\eqref{eadsncora}, we obtain that
\begin{eqnarray*}&&
\int_{\R^n}\big|-1+e^{ir\zeta\cdot y}
-i{r\zeta\cdot y}\,\chi_{B_1}(y)\big|\,d\mu(y)\\&\le&
\int_{B_1\cap B_{{1}/{(r|\zeta|)}}} r^2|\zeta|^2\,| y|^2\,d\mu(y)+
\int_{ B_1\setminus B_{{1}/{(r|\zeta|)}} }\big(
2r|\zeta|\,| y|-1\big)\,d\mu(y)+2\mu(\R^n\setminus B_1)\\
&\le& r^2|\zeta|^2\,
\int_{B_{{1}/{(r|\zeta|)}}} | y|^2\,d\mu(y)+2r|\zeta|\,\mu(
B_1\setminus B_{{1}/{(r|\zeta|)}} )+2\mu(\R^n\setminus B_1).
\end{eqnarray*}
Also, arguing as in~\eqref{PALLAr},
we see that
\begin{equation*}
\int_{B_\rho} | y|^2\,d\mu(y)\le
C(\mu,\alpha)\,\rho^{2-\alpha},\end{equation*}
and therefore we conclude that
\begin{equation}\label{FRami} \int_{\R^n}\big|-1+e^{ir\zeta\cdot y}
-i{r\zeta\cdot y}\,\chi_{B_1}(y)\big|\,d\mu(y)\le
C(\mu,\alpha)\,
r^\alpha|\zeta|^\alpha+2r|\zeta|\,\mu(
B_1\setminus B_{{1}/{(r|\zeta|)}} )+2\mu(\R^n\setminus B_1).\end{equation}
In particular, when~$r>1/|\zeta|$, we have that~$B_1\setminus B_{{1}/{(r|\zeta|)}}=\varnothing$,
and then~\eqref{FRami} becomes
$$ \int_{\R^n}\big|-1+e^{ir\zeta\cdot y}
-i{r\zeta\cdot y}\,\chi_{B_1}(y)\big|\,d\mu(y)\le
C(\mu,\alpha)\,
r^\alpha|\zeta|^\alpha+2\mu(\R^n\setminus B_1).$$
As a consequence, taking~$r:=\lambda^{-1}$ here,
$$ \lambda^{\alpha}\int_{\R^n}\big|-1+e^{i\lambda^{-1}\zeta\cdot y}
-i{\lambda^{-1}\zeta\cdot y}\,\chi_{B_1}(y)\big|\,d\mu(y)=O(1),$$
as~$\lambda\searrow0$.

And of course, taking~$r:=1$ in~\eqref{FRami},
$$ \int_{\R^n}\big|-1+e^{i\zeta\cdot y}
-i{\zeta\cdot y}\,\chi_{B_1}(y)\big|\,d\mu(y)=O(1).$$
Using these pieces of information in~\eqref{STyperbe}, we thereby conclude that
$$ 0= 
i\lambda^{\alpha-1} b\cdot\zeta+
\lambda^{\alpha-2}Q\zeta\cdot\zeta+O(1),$$
as~$\lambda\searrow0$, and consequently
\begin{equation}\label{CONt} Q\zeta\cdot\zeta=0\qquad{\mbox{and}}\qquad
b\cdot\zeta=0.\end{equation}
As a consequence, we can write~\eqref{LAPSO} as
\begin{equation} \label{Pizz4224pa}
\Psi(\zeta)=\int_{\R^n}
\Big(1-e^{i\zeta\cdot y}+i{\zeta\cdot y}\,\chi_{B_1}(y)
\Big)\,d\mu(y).\end{equation}
It is also convenient to exploit~\eqref{CONt} and~\eqref{STyperbe},
thus finding that
\begin{equation}\label{hstavc} 0=\int_{\R^n}
\Big(\lambda^{\alpha}\big(1- e^{i\lambda^{-1}\zeta\cdot y}
+i\lambda^{-1} {\zeta\cdot y}\,\chi_{B_1}(y)
\big)+\big(-1+
e^{i\zeta\cdot y}
-i{\zeta\cdot y}\,\chi_{B_1}(y)
\big)
\Big)\,d\mu(y).\end{equation}
Also, arguing as in~\eqref{COVF},
we find that, in this case,
\begin{equation} \label{COnque}\int_{\R^n} f\left(\frac{y}\lambda\right)\,d\mu(y)=\frac{1}{\lambda^\alpha}\,
\int_{\R^n} f(y)\,d\mu(y)\end{equation}
and therefore
$$ \lambda^{\alpha}\int_{\R^n}
\Big( 1- e^{i\lambda^{-1}\zeta\cdot y}
+i\lambda^{-1} {\zeta\cdot y}\,\chi_{B_1}(y)
\Big)\,d\mu(y)=
\int_{\R^n}
\Big( 1- e^{i\zeta\cdot y}
+i {\zeta\cdot y}\,\chi_{B_{1/\lambda}}(y)
\Big)\,d\mu(y).
$$
Comparing with~\eqref{hstavc}, we conclude that
$$ \int_{\R^n} y\,\Big(\chi_{B_{1/\lambda}}(y)-\chi_{B_1}(y)
\Big)\,d\mu(y)=0,$$
and therefore
$$ \int_{\partial B_1}y\,d\mu(y)=0.$$
From this and~\eqref{COnque} we obtain that the same holds with~$\partial B_1$
replaced by~$\partial B_r$, for all~$r>0$, from which one deduces that
$$ \int_{\R^n} y\,\chi_{\R^n\setminus B_1}(y)\,d\mu(y)=0.$$
In light of this and~\eqref{Pizz4224pa}, one concludes that
$$ \Psi(\zeta)=\int_{\R^n}
\Big(1-e^{i\zeta\cdot y}+i{\zeta\cdot y}
\Big)\,d\mu(y),$$
and this completes the proof of~\eqref{Ecces} when~$\alpha\in(1,2)$.
\end{proof}

\vfill
\end{document}